\documentclass[ejs]{imsart}

\usepackage{amsthm,amsmath,amsfonts,amssymb}
\RequirePackage[numbers]{natbib}
\RequirePackage[colorlinks,citecolor=blue,urlcolor=blue]{hyperref}
\RequirePackage{graphicx}
\usepackage{subcaption}
\usepackage{multirow} 
\usepackage{bbm}
\usepackage{bigstrut}
\usepackage[plain,linesnumbered]{algorithm2e}
\arxiv{2303.110540}
\startlocaldefs
\theoremstyle{plain}

\newtheorem{theorem}{Theorem}[section]
\newtheorem{lemma}[theorem]{Lemma}
\newtheorem{assumption}[theorem]{Assumption}
\newcommand{\E}{{\rm E}}
\newcommand{\R}{\mathbb{R}}
\renewcommand{\P}{\mathbb{P}}
\newcommand{\n}{^{(n)}}
\newcommand{\cF}{\mathcal{F}}
\newcommand{\rmd}{ {\rm d}}
\newcommand{\bea}{\begin{eqnarray}}
	\newcommand{\eea}{\end{eqnarray}}
\newcommand{\beq}{\begin{equation}}
	\newcommand{\eeq}{\end{equation}}
 
\theoremstyle{definition}
\newtheorem{definition}[theorem]{Definition}
\newtheorem{corollary}[theorem]{Corollary}

\theoremstyle{remark}

\endlocaldefs

\begin{document}
\begin{frontmatter}
\title{Inference via robust optimal transportation: theory and methods}
\runtitle{Robust optimal transportation}

\begin{aug}
\author[A]{\fnms{Yiming}~\snm{Ma}\ead[label=e1]{mayiming@mail.ustc.edu.cn}},
\author[B]{\fnms{Hang}~\snm{Liu}\ead[label=e2]{hliu01.ustc.edu.cn}}
\author[C]{\fnms{Davide}~\snm{La Vecchia}\ead[label=e3]{Davide.LaVecchia@unige.ch}}
\and
\author[D]{\fnms{Matthieu}~\snm{Lerasle}\ead[label=e4]{matthieu.lerasle@ensae.fr}}
\address[A]{Department of Statistics and Finance, School of Management,
	University of Science and Technology of China\printead[presep={,\ }]{e1}}

\address[B]{International Institute of Finance, School of Management, 
	University of Science and Technology of China \printead[presep={,\ }]{e2}}

\address[C]{University of Geneva\printead[presep={,\ }]{e3}}

\address[D]{CREST, CNRS, \'Ecole polytechnique, GENES, ENSAE Paris, Institut Polytechnique de Paris, 91120 Palaiseau, France\printead[presep={,\ }]{e4}}

\runauthor{Yiming Ma et al.}

\end{aug}

\begin{abstract}
Optimal transportation theory and the related $p$-Wasserstein distance ($W_p$, $p\geq 1$) are widely-applied in statistics and machine learning. In spite of their popularity, inference based on these tools has some issues. For instance, it is sensitive to outliers and it may not be even defined when the underlying model has infinite moments. To cope with these problems, first we consider a robust version of the primal transportation problem and show that it defines the {robust Wasserstein distance}, $W^{(\lambda)}$, depending on a tuning parameter $\lambda > 0$.  Second, we illustrate the link between $W_1$ and  $W^{(\lambda)}$ and study its key measure theoretic aspects. Third, we derive some concentration inequalities for $W^{(\lambda)}$. Fourth, we use $W^{(\lambda)}$ to define  minimum distance estimators, we provide their statistical guarantees and we illustrate how to apply the derived concentration inequalities for a data driven selection of $\lambda$.  Fifth,  we provide the {dual} form of the robust optimal transportation problem and we apply it to machine learning problems (generative adversarial networks and domain adaptation).  Numerical exercises 
provide evidence of the benefits yielded by our novel methods. 
\end{abstract}



\end{frontmatter}

\section{Introduction}
\label{intro}

\subsection{Related literature}

Optimal transportation (OT) dates back to the  work of \cite{monge1781memoire}, who considered the  problem of finding the optimal way to move given piles of sand to fill up given holes of the same total volume. Monge’s problem remained open until it was revisited and solved by Kantorovich, who characterized the optimal transportation plan in relation to the economic
problem of optimal allocation of resources. We refer to  \cite{villani2009optimal}, \cite{santambrogio2015optimal} for a book-length presentation in mathematics, to \cite{peyre2019computational} for a data science view, and to \cite{panaretos2020invitation} for a statistical perspective. 

Within the OT setting, the Wasserstein distance ($W_p, p \geq 1$) is defined as the minimum cost of transferring the probability mass from a source distribution to a target distribution, for a given transfer cost function. Nowadays, OT theory and $W_p$ are of fundamental importance for many scientific areas, including statistics, econometrics, information theory and machine learning. For instance,  in  machine learning, OT based inference is  a popular tool in generative models; see e.g.  \cite{arjovsky2017wasserstein}, \cite{genevay2018learning}, \cite{NGC22}  and references therein. Similarly, OT techniques  are widely-applied for problems related to domain adaptation; see  \cite{courty2014domain} and \cite{balaji2020robust}. In statistics and econometrics, estimation methods based on  $W_p$ are related to minimum distance estimation (MDE); see \cite{bassetti2006minimum} and  \cite{bernton2019parameter}. MDE is a research area in continuous development and the estimators  based on this technique are called minimum distance estimators. They are linked to M-estimators and may be robust. We refer to \cite{hampel1986robust},  \cite{van2000asymptotic} and  \cite{basu2011statistical} for a book-length statistical  introduction. As far as econometrics is concerned,  we refer to \cite{kitamura1997information} and \cite{hayashi2011econometrics}. The extant approaches for MDE are usually defined via Kullback-Leibler (KL) divergence, Cressie-Read power divergence, 
total variation (TV) distance, 
Hellinger distance, to mention a few. Some of those approaches (e.g. Hellinger distance) require kernel density estimation of the underlying probability density function. Some others (e.g. KL divergence) have poor performance when the considered distributions do not have a common support.   Compared to other notions of distance or divergence, $W_p$ avoids the mentioned drawbacks, e.g. it can be applied when the considered distributions do not share the same support and inference is conducted in a generative fashion---namely, a sampling mechanism is considered, in which non-linear functions map a low dimensional latent random vector
to a larger dimensional space; see \cite{genevay2018learning}.  This explains why $W_p$ is a popular inference tool; see e.g. \cite{courty2014domain,frogner2015learning,kolouri2017optimal,carriere2017sliced, peyre2019computational}; see \cite{H22} and \cite{la2022some} for a recent overview of technology transfers among different areas.

In spite of their appealing theoretical properties, OT and $W_p$ have two important inference aspects to deal with: (i) implementation cost and (ii) lack of robustness. (i) The computation of the solution to OT problem and of the computation of $W_p$ can be demanding. 
To circumvent this problem, \cite{cuturi2013sinkhorn} proposes the use of an  entropy regularized version of the OT problem. The resulting divergence  is called 
Sinkhorn divergence; see \cite{amari2018information} for further details.
\cite{genevay2018learning} illustrate the use of this divergence for inference in generative models, via  MDE. 
The proposed 
estimation method is attractive in terms of performance and
speed of computation. However (to the best of our knowledge) the statistical guarantees of the minimum Sinkhorn divergence estimators have not been provided.  (ii) Robustness is becoming a crucial aspect for many complex statistical and machine learning applications. The reason is essentially two-fold. First, large datasets (e.g. records of images or large-scale data collected over the internet) can be cluttered with outlying values (e.g. due to recording device failures). Second,  every statistical model represents only an approximation to reality. Thus, the need for inference procedures that remain informative even in the
presence of small deviations from the assumed statistical
(generative) model is in high demand, in statistics and in machine learning. Robust statistics deals with this problem. 
We refer to  \cite{hampel1986robust} and \cite{ronchetti2009robust} for a book-length discussion on robustness principles. 

Other papers have already mentioned the robustness issues of  $W_p$ and OT; see e.g. \cite{alvarez2008trimmed},   \cite{hallin2022cent},\cite{hallin2023rank},\cite{hallin2023cente}
\cite{del2022nonparametric},
\cite{mukherjee2021outlier}, \cite{yatracos2022limitations} and \cite{ronchetti2022}.  
%
In the OT literature,  
\cite{chizat2017unbalanced}  
proposes to deal with outliers via unbalanced OT, namely allowing the source and target distribution to be non-standard probability distributions.  This sensitivity is an undesired consequence of exactly satisfying the marginal constraints in OT transportation problem: intuitively, using Monge's formulation, the piles of sand and the holes need to have the same total volume, thus the optimal transportation plan is affected by the outlying values inducing a high transportation cost. Relaxing the exact marginal constraints in a way that the transportation plan does not assign large weights to outliers yields the unbalanced OT theory:  a mass variation is allowed in the OT problem  (intuitively, the piles of sand and the holes do not need to have the same volume) and outlying values can be ignored by the transportation plan. Working along the lines of this approach, \cite{balaji2020robust}  define a robust optimal transportation (ROBOT) problem and study its applicability to machine learning. In the same spirit, starting from OT formulation with a regularization term which relaxes the marginal constraints, \cite{mukherjee2021outlier} derive a novel  robust  OT problem,
that they label as ROBOT and they apply it to machine learning. Moreover,
building on the Minimum Kantorovich Estimator (MKE) of \cite{bassetti2006minimum}, Mukherjee et al. propose a class of estimators  obtained through
minimizing the ROBOT  and illustrate numerically their performance. 
In the same spirit, recently, \cite{NGC22}  defined a notion of  outlier-robust Wasserstein distance  which allows for an $\epsilon$-percentage of outlier mass to be removed from contaminated distributions. 

\subsection{Our contributions}


Working at the intersection between mathematics, probability, statistics, machine learning and econometrics, we study the theoretical, methodological and computational aspects of ROBOT. Our contributions to these different research areas can be summarized as follows.

We prove that the primal ROBOT problem of \cite{mukherjee2021outlier} defines a robust distance that we  call robust Wasserstein distance $W^{(\lambda)}$, which depends on a tuning parameter $\lambda$ controlling the level of robustness. We study the measure theoretic properties of $W^{(\lambda)}$, proving that  it induces a metric space whose  
separability and completeness are proved.   Moreover, we derive  concentration inequalities for the robust Wasserstein distance and illustrate their practical applicability for the selection of $\lambda$ in the presence of outliers.

These theoretical (in mathematics and probability) findings are instrumental to another contribution (in statistics and econometrics) of this paper: the development of a novel inferential theory based on robust Wasserstein distance.  Making use of convex analysis \citep{rockafellar2009variational} and of the techniques in \cite{bernton2019parameter}, we prove that the minimum robust Wasserstein distance estimators exist (almost surely) and are consistent (at the reference model, possibly misspecified). Numerical experiments show that the robust Wasserstein distance estimators remain reliable even in the presence of local departures from the postulated statistical model. These results complement the methodological investigations already available in the literature on minimum distance estimation (in particular the MKE of \cite{bassetti2006minimum} and \cite{bernton2019parameter}) and on ROBOT (\cite{balaji2020robust} and \cite{mukherjee2021outlier}).  Indeed, beside the numerical evidence available in the mentioned papers, there are (to the best of our knowledge) no statistical guarantees on the estimators obtained by the method of \cite{mukherjee2021outlier} and \cite{balaji2020robust}. 
Moreover, our results extend the use of OT to conduct reliable inference on random variables with infinite moments.


Another contribution is the derivation of the dual form of ROBOT for application to machine learning problems. This result not only completes the theory in \cite{mukherjee2021outlier}, but also  provides  the stepping stone for the implementation of ROBOT: we explain how our duality  can be applied to define robust Wasserstein Generative Adversarial Networks (RWGAN). Numerical exercises provide evidence that RWGAN have better performance than the extant Wasserstein GAN (WGAN) models. 
Moreover, we illustrate how our ROBOT can be applied to domain adaptation, another popular machine learning problem; see Section~\ref{Sec.DA} and Appendix~\ref{App.ROBOT}.

 
{Finally, as far as the approach in \cite{NGC22}  is concerned, we remark that their notion of robust Wasserstein distance has a similar spirit to our $W^{(\lambda)}$ (both are based on a total variation penalization of the original OT problem, but yield different duality formulations) and, similarly to our results in \S\ref{Sect_DA}, it is proved to be useful in generative adversarial network (GAN). In spite of these similarities, our investigation goes deeper in the probabilistic and statistical aspects related to inference methods based on $W^{(\lambda)}$. With this regard, some key difference between  \cite{NGC22} and our paper are that Nietert et al.:  (a) do not study concentration inequalities for their distance; (b) do not provide indications on how to select and control the $\epsilon$-percentage of outlier mass ignored by the OT solution---for a numerical illustration of this aspect in the setting of GAN, see Section \ref{Sect_DA}; (c) do not define  the estimators related to the minimization of their outlier-robust Wasserstein distance; (d) use regularization in the dual and a constraint in the primal, whereas ROBOT does the reverse: this simplifies our derivation of the dual problem.  In the next sections we take care of all these and other aspects for the ROBOT, providing the theoretical and the applied statistician with a novel toolkit for robust inference via OT.}

\section{Optimal transport}\label{Sec:Preliminaries}


\subsection{Classical OT}\label{Optimal transport}

{We recall briefly the fundamental notions of OT as formulated by Kontorovich; more details  are available in Appendix~\ref{App.OT}. Let $\mathcal{X}$ and $\mathcal{Y}$ denote two  Polish spaces, and let
	$\mathcal{P}(\mathcal{X})$ represent the set of all  probability measures on $\mathcal{X}$. }
Let $\Pi (\mu,\nu)$  denote the set of all joint probability measures of $\mu \in \mathcal{P(X)}$ and $\nu \in \mathcal{P(Y)}$. Kontorovich's problem aims at finding a joint distribution $\pi   \in \Pi (\mu,\nu)$ which minimizes the expectation of the coupling between $X$ and $Y$ in terms of the cost function $c$. This problem can be formulated as
\begin{equation} \label{equ2}
	\inf \left\lbrace \int_{\mathcal{X\times Y}} c(x,y) d \pi(x,y) : \pi \in \Pi(\mu,\nu) \right\rbrace .
\end{equation}
A solution to Kantorovich’s problem (KP) \eqref{equ2} is called an { optimal transport plan}.  Note that the problem is  convex, and its solution  exists under some mild assumptions on $c$, e.g., lower semicontinuous; see e.g. \citep[Chapter 4]{villani2009optimal}.

KP as in \eqref{equ2} has a dual form, which is related to  Kantorovich dual  (KD) problem
	$$\sup \left\lbrace \int_{\mathcal{Y}} \phi d\nu - \int_{\mathcal{X}} \psi  d\mu : \phi \in C_{b}(\mathcal{Y}) , \psi \in C_b(\mathcal{X})  \right\rbrace \quad \text{s.t.} \quad \phi(y) - \psi(x) \leq c(x,y),\enspace$$
for  $\forall(x,y)$,
where $C_b(\mathcal{X})$ is the set of bounded continuous functions on $\mathcal{X}$.
According to Th. 5.10 in \cite{villani2009optimal}, if  function $c$ is lower semicontinuous, there exists a solution to the dual problem such that the solutions to KD and KP coincide  (no duality gap). In this case, the solution is
$$\phi(y) =\inf \limits_{x \in \mathcal{X}} [\psi(x)+c(x, y)] \quad \text{and} \quad
\psi(x)=\sup \limits_{y \in \mathcal{Y}}[\phi(y)-c(x, y)],$$
where the functions $\phi$ and $\psi$ are called { $c$-concave} and { $c$-convex}, respectively, and $\phi$ (resp. $\psi$) is called the { $c$-transform} of $\psi$ (resp. $\phi$). When $ c$ is a metric on $ \mathcal{X} $,  OT problem becomes 
$	\sup \left\lbrace \int_{\mathcal{X}} \psi d\nu - \int_{\mathcal{X}} \psi  d\mu : \psi \, \text{is 1-Lipschitz continuous}    \right\rbrace$, which 
is the  Kantorovich-Rubenstein duality.
Now, let $(\mathcal{X},d)$ denote a complete metric space equipped with a metric $d:\mathcal{X} \times \mathcal{X} \rightarrow \mathbb{R}$, and let $\mu$ and $\nu$ be two probability measures on $\mathcal{X}$. 
Solving the optimal transport problem in \eqref{equ2}, with the cost function  $c(x,y) = d^p(x,y)$, introduces a distance, called  Wasserstein distance 
of order $p$ $(p \geq 1)$:  
\begin{align}
	W_p(\mu,\nu)&= \left(\inf  \int_{\mathcal{X\times X}} (d(x,y))^p d \pi(x,y)\right)^{1/p}. \label{equ5} 
\end{align}


With regard to (\ref{equ5} ), we emphasize that the ability to lift the ground distance $d(x,y)$ is one of the perks of $W_p$ and it makes it a suitable tool in statistics and machine learning; see Remark 2.18 in~\cite{peyre2019computational}.  Interestingly, this desirable feature becomes a negative aspect as far as robustness is concerned.  Intuitively, OT embeds the distributions geometry: when the underlying distribution is contaminated by outliers, the marginal constraints force OT to transport outlying values, inducing an undesirable extra cost, which in turns entails large changes in the $W_p$. More in detail, 
let $\delta_{x_0}$ be  point mass measure centered at $x_0$ (a point in the sample space), and $d(x, y)$ be the ground distance. Then, we introduce the $O\cup I$ framework, in which 
the majority of the data $(X_i)_{i\in I}$ are {inliers}, that is i.i.d.  informative data with distribution $\mu$, whilst a few data $(X_i)_{i\in O}$ are {outliers}, meaning that they are not i.i.d. with distribution $\mu$. 
We refer to \cite{LCL20} (and references therein) for more detail on  $O\cup I$ and for discussion on its connection to Huber gross-error neighbourhood.
Thus, for $\varepsilon >0$, we have
$W_p({\rm \mu}, (1-\varepsilon){\rm \mu}+\varepsilon \delta_{x_0})   \rightarrow \infty$,
as $x_0 \rightarrow \infty$. Thus, $W_p$ can be arbitrarily inflated by a small number of large observations and it is  sensitive to outliers.  For instance, let us consider $W_p$, with $p=1,2$. In Figure~\ref{fig:1a} and \ref{fig:1b}, we display the scatter plot for two bivariate distributions: the distribution in  panel (b) has some abnormal points (contaminated sample) compared to the distribution in panel (a) (clean sample). Although only a small fraction of the sample data is contaminated, both $W_1$ and $W_2$  display a large increase. Anticipating some of the results that we will derive in the next pages, we compute our $W^{(\lambda)}$ (see \S \ref{Sect_Wl}) and we notice that its value remains roughly the same in both panels. 
We refer
to \S \ref{Sec.SC} for additional insights on the resistance to outliers of $W^{(\lambda)}$.


\begin{figure}[t!]
	\centering
	\begin{subfigure}[b]{0.35\textwidth}
		\centering
		\includegraphics[width=\textwidth]{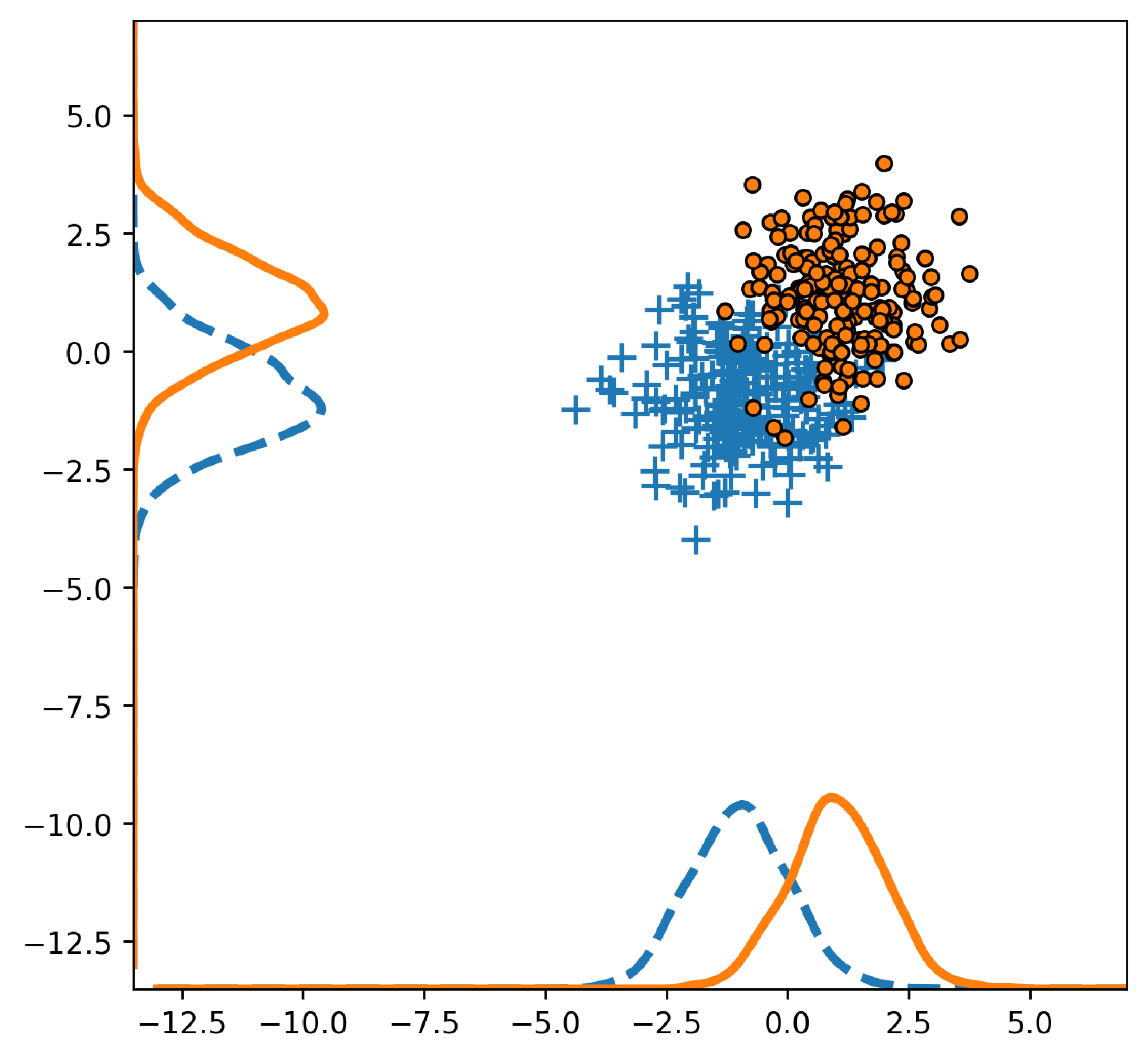}
		\caption{$W_1=3.0$, $W_2=3.0$ and $W^{(\lambda)}=2.9 (\lambda=3)$}
		\label{fig:1a}
	\end{subfigure}
 \hspace{10mm}
	\begin{subfigure}[b]{0.35\textwidth}
		\centering
		\includegraphics[width=\textwidth]{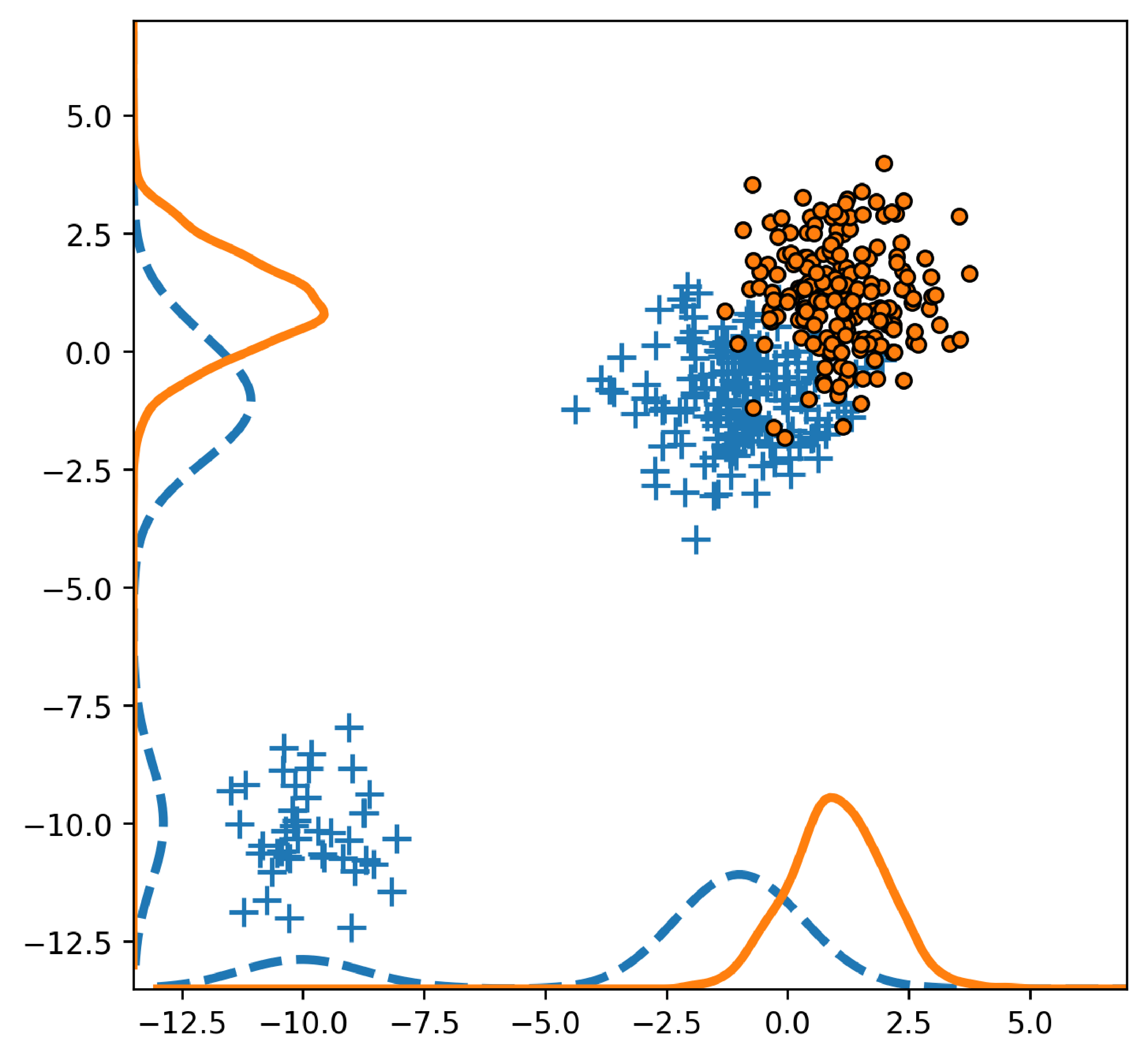}
		\caption{$W_1=5.3$, $W_2 = 6.6$ and $W^{(\lambda)}=3.3 (\lambda=3)$}
		\label{fig:1b}
	\end{subfigure}
	\caption{ Wasserstein distance ($W_1$ and $W_2$) and robust Wasserstein distance ($W^{(\lambda)}$, with $\lambda=3$) between two bivariate distributions. 
		The scatter plot of data in panel (a) represents a sample from the reference model:  cross points (blue) are generated from $ \mathcal{N}\left( \binom{-1}{-1},I_2 \right) $,  points (orange) are generated from $ \mathcal{N}\left( \binom{1}{1},I_2 \right) $. The plot in panel (b) contains some outliers: cross points (blue) are generated from $0.8  \mathcal{N}\left( \binom{-1}{-1},I_2 \right)+0.2\mathcal{N}\left( \binom{-9}{-9},I_2 \right) $, points (orange) are generated from $ \mathcal{N}\left( \binom{1}{1},I_2 \right) $. The marginal distributions are plotted on the $x$- and $y$-axis.}
	\label{fig1}
\end{figure}

\subsection{Robust OT (ROBOT)}

We recall the primal formulation of ROBOT problem, as defined in
\cite{mukherjee2021outlier}, who introduce a modified version of Kantorovich OT problem and work on the notion of unbalanced OT. Specifically, they consider 
a TV-regularized OT re-formulation
\begin{equation}\label{equ7}
	\begin{array}{ll}
		& \min _{\pi, s}  \iint c(x, y) \pi(x, y) d x d y+\lambda\|s\|_{\mathrm{TV}} \\
		\text {s.t.} & \int \pi(x, y) d y=\mu(x)+s(x) \geq 0 \\
		& \int \pi(x, y) d x=\nu(y) \\
		& \int s(x)dx=0,
	\end{array}
\end{equation}
where $\lambda>0$ is a regularization parameter. In \eqref{equ7}, the original source  measure $\mu$ is modified by adding $s$ (nevertheless, the first and last constraints ensure that $\mu + s$ is still a valid probability measure). Intuitively, having $\mu(x) +s(x)=0$ means that  $x \in {\cal X}$ has strong impact on the OT problem and hence can be labelled as an outlier. The outlier is eliminated from the sample, since the probability measure $\mu+s$ at this point is zero. 

\cite{mukherjee2021outlier} prove that a simplified, computationally efficient  formulation equivalent to \eqref{equ7} is 
\begin{equation} \label{equ8}
	\inf \left\lbrace  \int_{\mathcal{X\times X}} c_{\lambda}(x,y) d \pi(x,y) : \pi \in \Pi(\mu,\nu) \right\rbrace,
\end{equation}
which has a functional form similar to \eqref{equ2}, but the cost function $c$ is replaced by the trimmed cost function 
$c_{\lambda}=\min\left\lbrace c, 2\lambda \right\rbrace$   
that is bounded from above by $2\lambda$. Following \cite{mukherjee2021outlier}, we call formulations \eqref{equ7} and \eqref{equ8} ROBOT. 
Beside the primal formulation, in the following theorem we drive the dual form of the ROBOT, which is not available in  \cite{mukherjee2021outlier}.
{We need this duality for the development our inferential approach and, for ease of reference, we state it in form of theorem---the proof can be obtained from Th. 5.10 in \cite{villani2009optimal} and it is available in Appendix \ref{App:proof}.}


\begin{theorem}\label{thm3}
	Let $ \mu,\nu $ be probability measures in $ \mathcal{P(X)} $,
	with  $c_{\lambda}(x,y)$ as in (\ref{Eqc_l}). Then the {\rm c-transform} of a c-convex function $\psi(x)$ is itself, i.e. $\psi^{c}(x)= \psi(x)$. Moreover, the  dual form of ROBOT is related to the Kantorovich potential $\psi$, which is a solution to
	\begin{equation}
		\sup \left\lbrace \int_{\mathcal{X}} \psi d\mu - \int_{\mathcal{X}} \psi  d\nu : \psi \in C_b(\mathcal{X})  \right\rbrace, \label{Eq_DualROBOT}
	\end{equation}
	where $ \psi $ satisfies  $\vert \psi(x)-\psi(y) \vert \leq d(x,y)$ and $\mathrm{range} (\psi) \leq 2\lambda $.	
\end{theorem}

Note that the  difference between the dual form of OT and  that of ROBOT is that the latter imposes a restriction on the range of $\psi$. Thanks to this, we may think of using the ROBOT in theoretical proofs and inference problems where the dual form of OT is already applied, with the advantage that ROBOT remains stable even in the presence of outliers. In the next sections, we illustrate these ideas  {and we start the construction of our inference methodology by proving that to ROBOT is associated a robust distance, which will be the pivotal tool of our approach}.

\section{Robust Wasserstein distance $W^{(\lambda)}$ }\label{Sec:Theory}


\subsection{Key notions} 
\label{Sect_Wl}

Setting $c(x, y) = d(x, y)$, \eqref{equ5} implies that the corresponding minimal cost $W_1(\mu, \nu)$ is a metric between two probability measures $\mu$ and $\nu$. A similar result can be obtained also in the ROBOT setting of~\eqref{equ8}, if the modified cost function takes the form 
\begin{equation}
	c_{\lambda}(x,y) := \min \left\lbrace  d(x,y),2\lambda \right\rbrace. \label{Eqc_l}
\end{equation}
\begin{lemma}\label{lemma1}
	Let $d(\cdot,\cdot)$ denote a metric on ${\cal X}$. Then $c_{\lambda}(x, y) $ in (\ref{Eqc_l}) is also a metric on ${\cal X}$.
\end{lemma}

Based on Lemma~\ref{lemma1},  we define the robust Wasserstein distance and prove that it is a metric on $\mathcal{P(X)}$.

\begin{theorem}\label{thm1}
	For $c_{\lambda}(x, y), x, y \in {\cal X},$ defined in (\ref{Eqc_l}),  
	\begin{equation} \label{Eq:Wl}
		W^{(\lambda)}(\mu,\nu) :=  \inf \limits_{\pi \in \Pi(\mu,\nu)} \left\lbrace  \int_{\mathcal{X\times X}} c_{\lambda}(x,y) d \pi(x,y) \right\rbrace  
	\end{equation}
	is a metric on $\mathcal{P(X)}$, and we call it robust Wasserstein distance.
\end{theorem}

It is worth noting that $W_p$ is well defined  only if probability measures have finite $p$-th order moments. In contrast,  thanks to the boundedness of the $c_\lambda$, our  $W^{(\lambda)}$ in (\ref{Eq:Wl}) can be applied to  all  probability measures, even those with infinite moments of any order. Making use of $W^{(\lambda)}$ we introduce the following

\begin{definition}[Robust Wasserstein space]\label{def1}
	A robust Wasserstein space is an ordered pair $(\mathcal{P(X)}, W^{(\lambda)})$, where $\mathcal{P(X)}$ is the set of all probability measures on a complete separable metric space $\mathcal{X}$, and $W^{(\lambda)}$ defined in Th.~\ref{thm1} is a (finite) distance on $\mathcal{P(X)}$.
\end{definition}

\subsection{Some measure theoretic properties} \label{Sec: MT}

{In this section we state, briefly, some preliminary results. The proofs (which are based on \cite{villani2009optimal}) of these results are available in Appendix \ref{App:proof}, to which we refer the reader interested in the mathematical details. Here, we provide essentially the key aspects, which we prefer to state in the form of theorems and corollary for the ease of reference in the theoretical developments available in the next sections.}

Equipped with Th. \ref{thm1} and Definition \ref{def1}, it is interesting to connect $W_1$ to $W^{(\lambda)}$, via the characterization of its limiting behavior 
as $\lambda \rightarrow \infty$. Since $W^{(\lambda)}$ is continuous and  monotonically increasing with respect to $\lambda$, one can prove (via dominated convergence theorem) that its limit coincides with the Wasserstein distance of order $p=1$. Thus:
\begin{theorem}\label{thm2}
	For any probability measures $\mu$ and $\nu$ in ${\cal P}({\cal X})$, $W^{(\lambda)}$ is continuous and  monotonically non-decreasing with respect to $\lambda \in [0,\infty)$. Moreover, if $ W_1(\mu,\nu)  $ exists, we have
	$$\lim_{\lambda \rightarrow \infty} W^{(\lambda)}(\mu,\nu) =W_1(\mu,\nu).$$ 
\end{theorem}  
Th. \ref{thm2} has a two-fold implication: first, $W^{(\lambda)}\leq W_1$; second, 
for large values of the regularization parameter, $W^{(\lambda)}$ and $W_1$ behave similarly. 
Another  important property of   $W^{(\lambda)}$ is that weak convergence  is entirely characterized by convergence in the robust Wasserstein distance. More precisely, we have
%
%
\begin{theorem}\label{thm4}
	Let $\left(\mathcal{X},d \right)$ be a Polish space. Then 
	$W^{(\lambda)}$ metrizes the weak convergence in $\mathcal{P({X})}$. 
	In other words, if $\left(\mu_{k}\right)_{k \in \mathbb{N}}$ is a 
	sequence of measures in $\mathcal{P(X)}$ and $\mu$ is another 
	measure in $\mathcal{P(X)}$, then 
	$
	\left(\mu_{k}\right)_{k \in \mathbb{N}} \, \, \text{converges 
		weakly in} \, \, \mathcal{P(X)} \, \,  \text{to} \,\, \mu
	$ 
	and $ W^{(\lambda)}\left(\mu_{k}, \mu\right) \rightarrow 0$
	are equivalent.
\end{theorem}

The next result follows immediately from Theorem~\ref{thm4}.
\begin{corollary}[Continuity of $ W^{(\lambda)} $]\label{cor1}
	For a Polish space   $ (\mathcal{X },d) $,  suppose that $ \mu_k \, (resp. \, \nu_k) $ converges weakly to $ \mu \, (resp.\, \nu) $ in $ \mathcal{P(X)} $ as $k \rightarrow \infty$. Then
	$$W^{(\lambda)}(\mu_k,\nu_k) \rightarrow W^{(\lambda)}(\mu,\nu).$$
\end{corollary}

Finally, we  prove the separability and completeness of robust Wasserstein space, when $ \mathcal{X} $ is separable and complete. To this end, we state
\begin{theorem}\label{thm5} 
	Let $\mathcal{X}$ be a complete separable metric space. Then the space $\mathcal{P(X)}$, metrized by the robust Wasserstein distance $W^{(\lambda)}$, is complete and separable.  Moreover, the set of  finitely supported measures with rational coefficients is dense in $ \mathcal{P(X)} $.
\end{theorem}

{We remark that our dual form in Th. \ref{thm3}  unveils a connection between $W^{(\lambda)}$  and the  Bounded Lipschitz (BL) metric discussed in \cite[p.41 and p.42]{dudley1969speed}:   the two metrics coincide for a suitable selection of $\lambda$. With this regard, we emphaisze that, since the BL  meterizes the weak-star topology,  Th. \ref{thm1}  and Th. \ref{thm4} are in line with the results  in Section 3 of \cite{dudley1969speed}.}

\subsection{Some novel probabilistic aspects} \label{ProbAsp}


The previous results 
offer the mathematical toolkit needed 
to derive some concentration inequalities for $W^{(\lambda)}$ building on the extant results for $W_1$.  To illustrate this aspect, let us consider the case of Talagrand transportation inequality $T_p$. We recall that
using $W_p$ as a distance between probability measures, transportation inequalities state that, given $\alpha>0$, a probability measure $\mu$ on $X$ satisfies $T_p(\alpha)$ if the inequality
$$ W_p(\mu,\nu) \leq \sqrt{{2}H(\mu \vert \nu)/{\alpha}}$$
holds for any probability measure $\nu$, with $H(\mu \vert \nu)$ being the Kullback-Leibler divergence.  
For $p=1$, it has been proven  that $T_1(\alpha)$ is equivalent to the existence of a square-exponential moment; see e.g. \cite{bolley2007quantitative}. 
Now, note that, from Th.~\ref{thm2}, we have $W^{(\lambda)} \leq W_1$: we can immediately claim  
that if $T_1(\alpha)$ holds, then we have 
$$W^{(\lambda)}(\mu,\nu) \leq \sqrt{{2}H(\mu \mid \nu)/{\alpha}}.$$ Thus,
we state 


\begin{theorem}\label{thmCI}
	
	Let $\mu$ be a probability measure on $\mathbb{R}^d$, which satisfies a $T_1(\alpha)$ inequality, and $\{X_1,...,X_n\}$ be a random sample of independent variables, all distributed according to $\mu$; let also
	$\hat{\mu}_n:= n^{-1} \sum_{i=1}^n \delta_{X_i}$
	be the associated empirical measure, where $\delta_x$ is the Dirac distribution with mass on $x \in \mathbb{R}^d$. Then, for any $d^{\prime}>d$ and $\alpha^{\prime}<\alpha$, there exists some constant $n_0$, depending only on $\alpha^{\prime}, d^{\prime}$ and some square-exponential moment of $\mu$, such that for any $\varepsilon>0$ and $$n \geq n_0 \max \left(\varepsilon^{-\left(d^{\prime}+2\right)}, 1\right),$$ we have
	$${\P}\left[W^{(\lambda)}\left(\mu, \hat{\mu}_n\right)>\varepsilon\right] \leq \exp\{- 0.5\alpha^{\prime} n \varepsilon^2\}.$$
\end{theorem}

The proof follows along the lines as the proof of Th. 2.1 in Bolley et al. combined with  Theorem~\ref{thm3} which yields 
\begin{equation}
	\mathrm{exp}\{- 0.5{\alpha^{\prime} }{2} n \varepsilon^2\} \geq {\P}\left[W_1\left(\mu, \hat{\mu}_n\right)>\varepsilon\right] \geq {\P} \left[W^{(\lambda)}(\mu,\hat\mu_n)>\varepsilon\right]. \label{Arg}
\end{equation} 

We remark that condition $n\geq n_0 \max (\varepsilon^{-\left(d^{\prime}+2\right)}, 1)$ implies that Th. \ref{thmCI} has an asymptotic nature.
Nevertheless, Bolley et al. prove that $${\P}\left[W_1\left(\mu, \hat{\mu}_n\right)>\varepsilon\right] \leq C(\epsilon) \exp\{- 0.5{\alpha^{\prime}}{} n \varepsilon^2\}$$ hods for any $n$ with $C(\epsilon)$ being a (large) constant depending on $\epsilon$. The argument in (\ref{Arg}) implies 
$${\P}\left[W^{(\lambda)} \left(\mu, \hat{\mu}_n\right)>\varepsilon\right] \leq C(\epsilon) \exp\{- 0.5{\alpha^{\prime}}{} n \varepsilon^2\}$$  for any $n \in \mathbb{N}$. This yields a concentration inequality which is valid for any sample size.

We obtained the results in Th. \ref{thmCI} using  an upper bound on $W_1$. The simplicity of our argument comes with a cost: for Th. \ref{thmCI}  to hold,  the underlying measure $\mu$ needs to satisfy some conditions required for $W_1$, like e.g. some moment restrictions. To obtain  bounds that do not need these stringent conditions,  in the next theorem we derive novel non asymptotic upper bounds working directly  on $W^{(\lambda)}$. The techniques that we apply to obtain our new results 
put under the spotlight the role of $\lambda$ to control the concentration in  the $O\cup I$ framework. 

\begin{theorem}\label{thm:NewCI}
	Assume that $X_1,\ldots,X_n$ are generated according to the $O\cup I$ setting and let $\mu$ denote the common distribution of the inliers.
	Let 
	$\hat{\mu}_n^{(I)}=|I|^{-1}\sum_{i\in I}\delta_{X_i}$ denote the empirical distribution based on inliers.
	Then, for any $t>0$ an, when $O=\varnothing$, we have
	\begin{gather}
		\P\bigg(|W^{(\lambda)}(\mu,\hat{\mu}^{(I)}_n)-\E[W^{(\lambda)}(\mu,\hat{\mu}^{(I)}_n)]|>\sigma\sqrt{\frac{2t}{|I|}}+\frac{4\lambda t}{|I|}\bigg)\leqslant 2\exp\big(-t\big)\label{IneqI}\enspace,
	\end{gather}
	where the variance term  is $$\sigma^2=|I|^{-1}\sum_{i\in I}\E[\min(d^2(X_i,X_i'),(2\lambda)^2)]$$ and $X_i'$ is an independent copy of $X_i$. When $ O \neq \varnothing$, we have that, for any $t>0$,
	\begin{gather}
		\P\bigg(\bigg|W^{(\lambda)}(\mu,\hat{\mu}_n)-\frac{|I|}n\E[W^{(\lambda)}(\mu,\hat{\mu}^{(I)}_n)]\bigg|>\sigma\sqrt{\frac{|I|}n}\sqrt{\frac{2t}{n}}+\frac{4\lambda t}{n}+\frac{2\lambda|O|}n\bigg)\leqslant 2\exp\big(-t\big) \label{IneqO} \enspace 
	\end{gather}
\end{theorem}

Some remarks are in order. (i)  In the case where $O=\varnothing$, (\ref{IneqI}) 
holds {without any assumption on $\mu$} as the variance term $\sigma^2$ is always bounded, even without assuming a finite second moment of $X_i$. This strongly relaxes the exponential square moment assumption required for $W_1(\mu,\hat{\mu}_n)$ in Th.~\ref{thmCI}.

(ii) (\ref{IneqO}) illustrates that the presence of contamination entails the extra term $2\lambda|O|/n$:
one can make  use of the tuning parameter $\lambda$ to mitigate the influence of outlying values. 
 
(iii) In (\ref{IneqO}), 
it is always possible to replace the term $${|I|}\E[W^{(\lambda)}(\mu,\hat{\mu}^{(I)}_n)]/n$$ by either $\E[W^{(\lambda)}(\mu,\hat{\mu}_n)]$ or $\E[W^{(\lambda)}(\mu,\hat{\mu}'_n)]$, where $\mu_n'=n^{-1}\sum_{i=1}^n\delta_{X_i'}$ and $X_1',\ldots,X_n'$ are i.i.d. with common distribution $\mu$.
This can be proved noticing that it is possible to bound $W^{(\lambda)}(\mu,\hat{\mu}_n)$ exploiting the fact that  $c_\lambda$ is bounded by $2\lambda$, so Th.~\ref{thm3} implies 
\begin{equation}\label{eq:LinkedIn}
	\frac{|I|}nW^{(\lambda)}(\mu,\hat{\mu}_n^{(I)})-\frac{2\lambda|O|}n\leqslant W^{(\lambda)}(\mu,\hat{\mu}_n)\leqslant \frac{|I|}nW^{(\lambda)}(\mu,\hat{\mu}_n^{(I)})+\frac{2\lambda|O|}n\enspace. 
\end{equation}

(iv) Delving more into the link between \eqref{IneqO} and  \eqref{eq:LinkedIn}, consider that, for any $\psi \in C_b(\mathcal{X})$, 
	\begin{align*}
		\int_{\mathcal{X}} \psi d\hat{\mu}_n - \int_{\mathcal{X}} \psi  d\mu&=\frac1n\sum_{i=1}^n\big\{\psi(X_i)-\int_{\mathcal{X}} \psi  d\mu\big\}\\
		&=\frac1n\bigg(\sum_{i\in I}\big\{\psi(X_i)-\int_{\mathcal{X}} \psi  d\mu\big\}+\sum_{i\in O}\big\{\psi(X_i)-\int_{\mathcal{X}} \psi  d\mu\big\}\bigg)\\
		&=\frac{1}n\bigg(|I|\big\{\int_{\mathcal{X}} \psi d\hat{\mu}^{(I)}_n - \int_{\mathcal{X}} \psi  d\mu\big\}+\sum_{i\in O}\big\{\psi(X_i)-\int_{\mathcal{X}} \psi  d\mu\big\}\bigg)\enspace.
	\end{align*}
	As $\text{range}(\psi)\leqslant 2\lambda$, for any $i\in O$, it holds that
	\[
	\big|\psi(X_i)-\int_{\mathcal{X}} \psi  d\mu\big|\leqslant 2\lambda\enspace.
	\]
	Therefore, we have, for any $\psi \in C_b(\mathcal{X})$, 
	\begin{align*}
		\big|\big\{\int_{\mathcal{X}} \psi d\hat{\mu}_n - \int_{\mathcal{X}} \psi  d\mu\big\}-\frac{|I|}n\big\{\int_{\mathcal{X}} \psi d\hat{\mu}^{(I)}_n - \int_{\mathcal{X}} \psi  d\mu\big\}\big|\leqslant \frac{2\lambda|O|}n\enspace.
	\end{align*}
	Taking the suprema over $\psi$ in $C_b(\mathcal{X})$ yields \eqref{eq:LinkedIn}.
	From \eqref{eq:LinkedIn} and the triangular inequality, we have that (adding and subtracting $\frac{|I|}nW^{(\lambda)}(\mu,\hat{\mu}^{(I)}_n)$)
	\begin{align*}
	\bigg|W^{(\lambda)}(\mu,\hat{\mu}_n)-\frac{|I|}n\E[W^{(\lambda)}(\mu,\hat{\mu}^{(I)}_n)]\bigg|
		\leqslant \frac{|I|}n\bigg|W^{(\lambda)}(\mu,\hat{\mu}^{(I)}_n)-\E[W^{(\lambda)}(\mu,\hat{\mu}^{(I)}_n)]\bigg| \\
		+\frac{2\lambda|O|}n\enspace.
	\end{align*}
	Inequality~\eqref{IneqO}  follows by plugging \eqref{IneqI} into this last bound. Finally, combining (\ref{IneqO}) and \eqref{eq:LinkedIn}, we obtain 
\begin{equation}
	\P\bigg(\big|W^{(\lambda)}(\mu,\hat{\mu}_n)-\E[W^{(\lambda)}(\mu,\hat{\mu}_n)]\big|>\sigma\sqrt{\frac{|I|}n}\sqrt{\frac{2t}{n}}+\frac{4\lambda t}{n}+\frac{4\lambda|O|}n\bigg)\leqslant 2\exp\big(-t\big)\enspace. \label{MattCI}
\end{equation}
Beside the concentrations in Th. \ref{thm:NewCI}, we consider the problem of providing an upper bound to $\E[W^{(\lambda)}(\mu,\hat\mu_n^{(I)})]$. Following \cite{boissard2014mean}, we call this quantity  the mean rate of convergence. For the sake of exposition,  we assume that all data are inliers, that is $X_1,\ldots,X_n$ are i.i.d. with common distribution $\mu$---the case with outliers can be obtained using (\ref{eq:LinkedIn}). As in Th. \ref{thmCI}, we notice that it is always possible to use $W^{(\lambda)}\leqslant W_1$ from Th. \ref{thm2} and  derive bounds for $\E[W^{(\lambda)}(\mu,\hat\mu_n)]$ applying the results for $\E[W_1(\mu,\hat{\mu}_n)]$; see e.g.
\cite{boissard2014mean}, \cite{fournier2015rate, Lei20, fournier2022convergence}. 
Beside this option, here we derive novel results for $\E[W^{(\lambda)}(\mu,\hat\mu_n)]$. 
In what follows, we assume that $\mu$ is supported in $B(0,K)$, the ball in Euclidean distance of $\R^d$ with radius $K$ and that the reference distance is the Euclidean distance in $\R^d$. Then we state:

\begin{theorem} \label{Matt_Th2}
	Let $\mu$ denote a measure on $B(0,K)$, the ball in Euclidean distance of $\R^d$ with radius $K$.
	Let $X_1,\ldots,X_n$ denote an i.i.d. sample of $\mu$ and $\hat{\mu}_n$ denote the empirical measure.
	Then, we have
	\[
	\E[W^{(\lambda)}(\mu,\hat\mu_n)]\leqslant 
	\begin{cases}
		C\sqrt{\frac{K(K\wedge\lambda)}{n}}&\text{ when }d=1\enspace,\\
		\frac{CK}{\sqrt{n}}\log\bigg(\frac{\sqrt{n}}K\bigg)&\text{ when }d=2\enspace,\\
		\frac{CK}{n^{1/d}}&\text{ when }d\geqslant 3\enspace.
	\end{cases}
	\]
\end{theorem}
The bound is obtained using (i) the duality formula derived in Th.~\ref{thm3}, (ii) an  extension of Dudley's inequality in the spirit of \cite{boissard2014mean}, and (iii)  a computation of the entropy number of a set of $1$-Lipschitz functions defined on the ball $B(0,K)$ and taking value in $[-\lambda,\lambda]$. The technical details are available in Appendix~\ref{App:proof}. Here, we remark the key role of $\lambda$ when $d=1$ (univariate case), which allows to bound the mean rate of convergence. For $d\geq2$, our  bounds could have been obtained also from \cite{boissard2014mean}. 
However, the obtained results  can be substantially larger than our upper bounds when $\lambda\ll K$: Th. \ref{Matt_Th2} provides a refinement (related to the boundedness of $c_\lambda$) of the extant bounds; see 
also Proposition 9 in \cite{NGC22} for a similar result.

\section{Inference via minimum $W^{(\lambda)}$  estimation}\label{Robust Wasserstein estimator}

%



\subsection{Estimation method}\label{Sec.EstMethod}

Let us consider a probability space $(\Omega, \mathcal{F}, {\rm P})$. On this probability space,  we define random variables taking values on $\mathcal{X} \subset \mathbb{R}^d, d \geq 1$ and endowed with the Borel $\sigma$-algebra. We observe $n \in \mathbb{N}$ i.i.d. data points $\{X_1, \ldots, X_n\}$, which are distributed according to $\mu_{\star}^{(n)} \in \mathcal{P}\left(\mathcal{X}^n\right)$. 
A parametric statistical model  on $ \mathcal{X}^n $ is denoted by  $ \{\mu_{\theta}\n \}_{\theta \in \Theta} $ and it is a collection of probability distributions indexed by a parameter $\theta$ of dimension $d_{\theta}$.  The parameter space is $ \Theta \subset \mathbb{R}^{d_\theta}, d_{\theta} \geq 1  $, which is equipped with a distance  $\rho_{\Theta}$.  We let  $Z_{1: n}$ represent the observations from  $\mu_{\theta}\n$. For every $\theta \in$ $\Theta$, the sequence $\left(\hat{\mu}_{\theta, n}\right)_{n \geq 1}$ of random probability measures on $\mathcal{X}$ converges (in some sense) to a distribution $\mu_\theta \in \mathcal{P}(\mathcal{X})$, where $$\hat{\mu}_{\theta, n}=n^{-1} \sum_{i=1}^n \delta_{Z_i}$$ with $Z_{1: n} \sim \mu_\theta^{(n)}$. Similarly, we will  assume that the empirical measure $\hat{\mu}_n$ converges  (in some sense) to some distribution $\mu_{\star} \in \mathcal{P}(\mathcal{X})$ as $n \rightarrow \infty$. 
We say that the model is well-specified if there exists $\theta_\star \in \Theta$ such that $\mu_\star\equiv  \mu_{\theta_\star}$; otherwise, it is misspecified. Parameters are identifiable: $\theta_1 = \theta_2$ is implied by 
$\mu_{\theta_1}=\mu_{\theta_2}$. 

Our estimation method relies on selecting  a parametric model in $ \{\mu_{\theta}\n \}_{\theta \in \Theta} $, which is the closest, in robust Wasserstein distance, to the true model $\mu_{\star}$.
Thus, minimum $W^{(\lambda)}$ estimation refers to the minimization, over the parameter $\theta \in \Theta$, of the robust Wasserstein distance between the empirical distribution $\hat{\mu}_n$ and the reference model distribution $\mu_\theta$. This is similar to the approach described in \cite{bassetti2006minimum} and \cite{bernton2019parameter}, who derive minimum Kantorovich estimators by making use of  $W_p$. 

More formally, 
the 
{minimum robust Wasserstein estimator} (MRWE) is defined as
\begin{equation}\label{MRWE}
	\hat{\theta}_n^{\lambda} =\underset{\theta \in \Theta}{\operatorname{argmin}} W^{(\lambda)}\left(\hat{\mu}_n, \mu_\theta\right). 
\end{equation}
When there is no explicit expression for the probability measure characterizing the  parametric model (e.g. in complex generative models, see \cite{genevay2018learning} and \S~\ref{Sec.SimWesti} for some examples), the computation of MRWE can be difficult. To cope with this issue, in lieu of (\ref{MRWE}), we propose using the  { minimum expected robust Wasserstein estimator} (MERWE) 
defined as
\begin{equation}\label{MERWE}
	\hat{\theta}_{n, m}^{\lambda} =\underset{\theta \in \Theta}{\operatorname{argmin}} {\rm E}_m \left[  W^{(\lambda)}\left(\hat{\mu}_n, \hat{\mu}_{\theta, m}\right)\right] .
\end{equation}
where the expectation $ \mathrm{E}_m $  is taken over  the distribution $ \mu^{(m)}_{\theta} $.  
To implement the MERWE one can rely on  Monte Carlo methods and approximate numerically
$ {\rm E}_m [W^{(\lambda)}\left(\hat{\mu}_n, \hat{\mu}_{\theta, m}\right)] $. Replacing the robust Wasserstein distance with the $W_p$ in (\ref{MERWE}), one obtains the {minimum  Wasserstein estimator} (MWE) and the minimum expected  Wasserstein estimator (MEWE), studied  in \cite{bernton2019parameter}.

\subsection{Statistical guarantees }


Intuitively, the consistency of the MRWE and MERWE can be conceptualized as follows. We expect that the empirical measure converges to $\mu_\star$, in the sense that $W^{(\lambda)}(\hat{\mu}_n,\mu_\star)\rightarrow 0$ as $n\rightarrow \infty$; see Th. \ref{thm4}. Therefore,
the $\arg \min$ of $W^{(\lambda)}(\hat{\mu}_n,\mu_\star)$ should converge to  
$$
\theta^{\ast} = \arg \min W^{(\lambda)}(\mu_\star,\mu_{\theta} ),
$$
assuming its existence and unicity. The same can be said for the minimum of the MERWE, provided that $m\rightarrow \infty$. If the reference parametric model is correctly specified (e.g. no data contamination), $\theta^\ast$ is the limiting object of interest and it is
the minimizer of $W^{(\lambda)}(\mu_\star,\mu_{\theta} )$. In the case of model misspecification (e.g. wrong parametric form of $\mu_\theta$ and/or presence of data contamination), $\theta^\ast$ is not necessarily the parameter that minimizes the KL divergence between the empirical measure and the measure characterizing the reference model. We emphasize that this is at odd with the standard misspecification theory (see e.g. \cite{white1982maximum}) and it is due to the fact that we replace the KL divergence (which yields non robust estimators) with our novel robust Wasserstein distance.

To formalize these arguments, 
we introduce the following set of assumptions, which are standard in the literature on MKE; see \cite{bernton2019parameter}.

\begin{assumption}\label{asum1}
	The data-generating process is such that $W^{(\lambda)}\left(\hat{\mu}_n, \mu_{\star}\right) \rightarrow 0, {\rm P}$-almost surely as $n \rightarrow \infty$.
\end{assumption}

\begin{assumption}\label{asum2}
	The map $\theta \mapsto \mu_\theta$ is continuous in the sense that $\rho_{\Theta}\left(\theta_n, \theta\right) \rightarrow 0$ implies that $\mu_{\theta_n} $ convergences to $ \mu_\theta$ weakly as $n \rightarrow \infty$.
\end{assumption}

\begin{assumption}\label{asum3}
	For some $\varepsilon>0$, 
	$$B_{\star}(\varepsilon)=\left\{\theta \in \Theta: W^{(\lambda)}\left(\mu_{\star}, \mu_\theta\right) \leq \varepsilon_{\star}+\varepsilon\right\},$$ with  $\varepsilon_{\star}=\inf _{\theta \in \Theta} W^{(\lambda)}\left(\mu_{\star}, \mu_\theta\right)$,  is a  bounded set.
\end{assumption}

Then we state the following
\begin{theorem}[Existence of MRWE]\label{thm6}
	Under Assumptions~\ref{asum1},\ref{asum2} and \ref{asum3}, there exists a set ${\cal A} \subset \Omega$ with ${\rm P}({\cal A})=1$ such that, for all $\omega \in {\cal A}$,
	$$
	\inf_{\theta \in \Theta} W^{(\lambda)}\left(\hat{\mu}_n(\omega), \mu_\theta\right) \rightarrow \inf _{\theta \in \Theta} W^{(\lambda)}\left(\mu_{\star}, \mu_\theta\right)
	$$
	and there exists $n(\omega)$ such that, for all $n \geq n(\omega)$, the sets $$\operatorname{argmin}_{\theta \in \Theta} W^{(\lambda)}\left(\hat{\mu}_n(\omega), \mu_\theta\right)$$ are non-empty and form a bounded sequence with
	$$
	\limsup _{n \rightarrow \infty} \underset{\theta \in \Theta}{\operatorname{argmin}} W^{(\lambda)}\left(\hat{\mu}_n(\omega), \mu_\theta\right) \subset \underset{\theta \in \Theta}{\operatorname{argmin}} W^{(\lambda)}\left(\mu_{\star}, \mu_\theta\right) .
	$$
\end{theorem}
To prove Th.~\ref{thm6} we need to show that  the sequence of functions $\theta \mapsto W^{(\lambda)}\left(\hat{\mu}_n(\omega), \mu_\theta\right)$ epi-converges (see Definition~\ref{def.EpiConverge} in Appendix~\ref{App:proof}) to $\theta \mapsto W^{(\lambda)}\left(\mu_{\star}, \mu_\theta\right)$. The result follows from \cite{rockafellar2009variational}. 
We remark that Th.~\ref{thm6} generalizes the results of \cite{bassetti2006minimum} and \cite{bernton2019parameter}, where the model is assumed to be well
specified (\cite{bassetti2006minimum}) and moments of order $p\geq 1$ are needed (\cite{bernton2019parameter}). 

Moving along the same lines as in Th. 2.1 in \cite{bernton2019parameter}, we may
prove the measurability of MRWE; see Th. \ref{thm7} in Appendix A.
%
%
%
These  results  for the MRWE provide the stepping stone to derive similar  theorems for the MERWE $\hat{\theta}_{n, m}^{\lambda}$. To this end, the following assumptions are needed.

\begin{assumption}\label{asum4}
	For any $m \geq 1$, if $\rho_{\Theta}\left(\theta_n, \theta\right) \rightarrow 0$, then $\mu_{\theta_n}^{(m)}$ converges to $\mu_\theta^{(m)}$ weakly as $n \rightarrow \infty$.
\end{assumption}

\begin{assumption}\label{asum5}
	If $\rho_{\Theta}\left(\theta_n, \theta\right) \rightarrow 0$, then ${\rm E}_n W^{(\lambda)}\left(\mu_{\theta_n}, \hat{\mu}_{\theta_n, n}\right) \rightarrow 0$ as $n \rightarrow \infty$.
\end{assumption}


\begin{theorem}[Existence of MERWE]\label{thm8}
	Under Assumptions~\ref{asum1},\ref{asum2},\ref{asum3},\ref{asum4} and \ref{asum5}, there exists a set ${\cal A} \subset \Omega$ with ${\rm P}({\cal A})=1$ such that, for all $\omega \in {\cal A}$, $$
	\inf _{\theta \in \Theta} {\rm E}_{m(n)}\left[  W^{(\lambda)}\left(\hat{\mu}_n(\omega), \hat{\mu}_{\theta, m(n)}\right)\right]  \rightarrow \inf _{\theta \in \Theta} W^{(\lambda)}\left(\mu_{\star}, \mu_\theta\right) $$
	and there exists $n(\omega)$ such that, for all $n \geq n(\omega)$, the sets 
	$$\operatorname{argmin}_{\theta \in \Theta} W^{(\lambda)}\left(\hat{\mu}_n(\omega), \hat{\mu}_{\theta, m(n)}\right)$$
	are non-empty and form a bounded sequence with
	$$
	\limsup _{n \rightarrow \infty} \underset{\theta \in \Theta}{\operatorname{argmin}} {\rm E}_{m(n)} \left[ W^{(\lambda)}\left(\hat{\mu}_n(\omega), \hat{\mu}_{\theta, m(n)}\right)\right]  \subset \underset{\theta \in \Theta}{\operatorname{argmin}} W^{(\lambda)}\left(\mu_{\star}, \mu_\theta\right).
	$$
\end{theorem}
For a generic function $f$, let us define 
$$\varepsilon\text{-}\operatorname{argmin}_x f := \left\{x: f(x) \leq \varepsilon+\inf _x f\right\}.$$ Then, we state the following

\begin{theorem}[Measurability of MERWE]\label{thm9}
	Suppose that $\Theta$ is a $\sigma$-compact Borel measurable subset of  $\mathbb{R}^{d_\theta}$. Under Assumption~\ref{asum4}, for any $n \geq 1$ and $m \geq 1$ and $\varepsilon>0$, there exists a Borel measurable function $\hat{\theta}_{n, m}: \Omega \rightarrow \Theta$ that satisfies 
	$$\hat{\theta}_{n, m}^{\lambda} (\omega) \in \operatorname{argmin}_{\theta \in \Theta} {\rm E}_m\left[  W^{(\lambda)} \left(\hat{\mu}_n(\omega), \hat{\mu}_{\theta, m}\right)\right] ,$$   { if this set is non-empty}, otherwise,  $$
	\hat{\theta}_{n, m}^{\lambda} (\omega) \in \varepsilon\text{-}\operatorname{argmin}_{\theta \in \Theta} {\rm E}_m\left[  W^{(\lambda)}\left(\hat{\mu}_n(\omega), \hat{\mu}_{\theta, m}\right)\right].
	$$
	
\end{theorem}

Considering the case where the data and $n$ are fixed, the next result shows that the MERWE converges to the MRWE, as $m \rightarrow \infty$. In the next assumption,  the observed empirical distribution is kept fix and $\varepsilon_n=\inf _{\theta \in \Theta} W^{(\lambda)}\left(\hat{\mu}_n, \mu_\theta\right)$.

\begin{assumption}\label{asum6}
	For some $\varepsilon>0$, the set 
	$$B_n(\varepsilon)=\left\{\theta \in \Theta: W^{(\lambda)}\left(\hat{\mu}_n, \mu_\theta\right) \leq \varepsilon_n+\varepsilon\right\}$$ is bounded.
\end{assumption}

\begin{theorem}[MERWE converges to MRWE as $ m \rightarrow  \infty $]\label{thm10}
	Under Assumptions~\ref{asum2},\ref{asum4},\ref{asum5} and \ref{asum6},
	$$
	\inf _{\theta \in \Theta} {\rm E}_m\left[  W^{(\lambda)}\left(\hat{\mu}_n, \hat{\mu}_{\theta, m}\right) \right]  \rightarrow \inf _{\theta \in \Theta} W^{(\lambda)}\left(\hat{\mu}_n, \mu_\theta\right),
	$$
	and there exists a $\tilde{m}$ such that, for all $m \geq \tilde{m}$, the sets 
	$$\operatorname{argmin}_{\theta \in \Theta} {\rm E}_m W^{(\lambda)}\left(\hat{\mu}_n, \hat{\mu}_{\theta, m}\right)$$ are non-empty and form a bounded sequence with
	$$
	\limsup _{m \rightarrow \infty} \underset{\theta \in \Theta}{\operatorname{argmin}} {\rm E}_m \left[ W^{(\lambda)}\left(\hat{\mu}_n, \hat{\mu}_{\theta, m}\right)\right]  \subset \underset{\theta \in \Theta}{\operatorname{argmin}} W^{(\lambda)}\left(\hat{\mu}_n, \mu_\theta\right).
	$$
\end{theorem}

\section{Numerical experiments}\label{Sec:Applications}

\subsection{Sensitivity to outliers} \label{Sec.SC}
To complement the insights obtained looking at Figure \ref{fig1}, we build on the notion of  sensitive curve, which is an empirical tool 
that illustrates the stability of statistical functionals; see 
\cite{tukey1977exploratory} and \cite{hampel1986robust} for a book-length presentation. 
%
%
We consider a similar tool to study the sensitivity of $W_1$ and $W^{(\lambda)}$ in  finite samples, both interpreted as functionals of the empirical measure. More in detail, 
we fix the standard normal distribution as a reference model (for this distribution both $W^{(\lambda)}$ and $W_1$ are well-defined) and we use it to generate a  sample of size  $n=1000$.  We denote the resulting sample $ \mathbf{X}_{n}:=(X_1, \ldots, X_{n})$, whose empirical measure is $\hat{\mu}_{n}$. Then, we replace $X_n $ with a value $ x \in \mathbb{R} $  and  form a new set of sample points, denoted by $ \mathbf{X}_{n}(x)$,  whose empirical measure is $\hat{\mu}^{(x)}_{n}$.
We let $ T_n $ represent the  robust Wasserstein distance between an $ n $-dimensional sample and $ \mathbf{X}_{n} $.   So, $ T_{n}(\mathbf{X}_{n})  $ is the empirical robust Wasserstein distance between $ \mathbf{X}_{n} $ and its self, thus it is equal to $0$, while  $ T_{n}(\mathbf{X}_{n}(x)) $ is the empirical robust Wasserstein distance between $ \mathbf{X}_{n}(x) $ and $ \mathbf{X}_{n} $. Finally, for different values of $x$ and $\lambda$,  we  compute (numerically) 
$$\Delta(x,W^{(\lambda)})=n\left[T_n\left(X_1, \ldots, X_{n-1}, x\right)-T_{n}\left(X_1, \ldots, X_{n}\right)\right]=n W^{(\lambda)}(\hat{\mu}_{n},\hat{\mu}^{(x)}_{n}).
$$  A similar procedure  is applied to obtain  $\Delta(x,W_1)$. 
We display the results in Figure~\ref{fig:SCofWasserstein}. 
For each value of $\lambda$, the plots illustrate that $\Delta(x,W^{(\lambda)})$ first grows and then remains flat (at the value of $2\lambda$) even for very large values of $\vert x \vert$. In contrast, for $\vert x \vert \to \infty$, $\Delta(x,W_1)$ diverges, in line with the evidence from (and the comments on) Figure \ref{fig1}. 
In addition, we notice that as $ \lambda \rightarrow \infty $, the plot of  $\Delta(x,W^{(\lambda)})$ becomes more and more similar to the one of $\Delta(x,W_1)$: this aspect is in line with Th. \ref{thm2} and it  is reminiscent of the behaviour of the Huber loss function (\cite{huber1992robust}) for location estimation, which converges to the quadratic loss (maximum likelihood estimation), as the  constant tuning  the robustness goes to infinity. However, an important remark is order: the cost $c_\lambda$ yields, in the language of robust statistics, the so-called ``hard rejection'': it bounds the influence of outlying values (to be contrasted with the behavior of Huber loss, which downweights outliers to preserve efficiency at the reference model); see  \cite{ronchetti2022} for a similar comment.

\begin{figure}[htbp]
	\centering
	\includegraphics[scale=0.2]{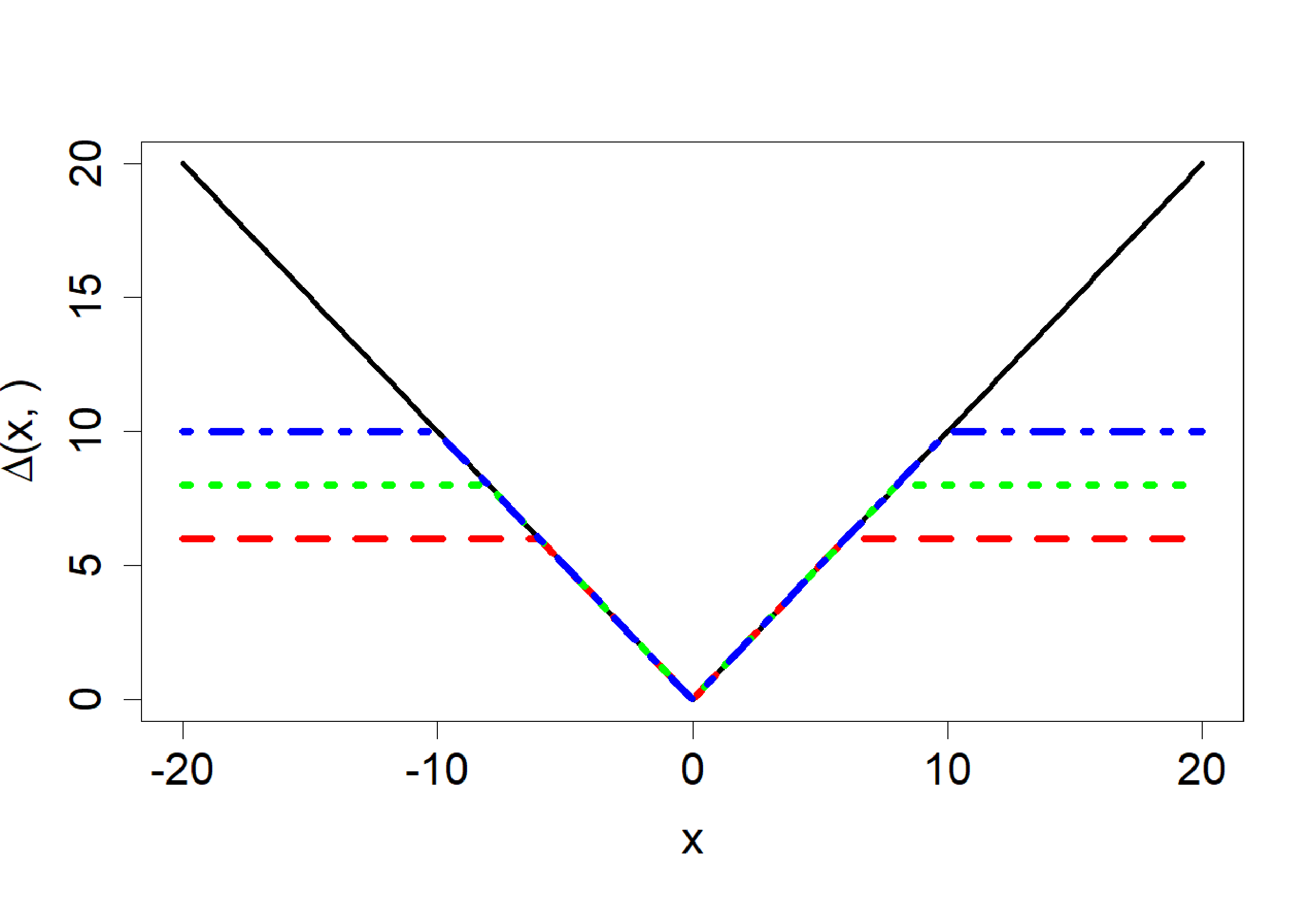}
	\caption{The continuous line represents $\Delta(x,W_1)$. The dashed (red), dot-dashed (blue) and dotted (green) line represents $\Delta(x,W^{(\lambda)})$ with $ \lambda=3,4, 5 $ respectively. }
	\label{fig:SCofWasserstein}
\end{figure}

\subsection{Estimation of location} \label{Sec.SimWesti}

Let us consider the statistical problem of inference on univariate location models. We study the robustness of MERWE, comparing its performance to the one of MEWE based on minimizing $W_1$, under different degrees of data contamination and for different underlying data generating models (with finite and infinite moments). 
Before delving into the numerical exercise, 
let us give some numerical details. We recall that the MERWE aims at minimizing the function  
$$\theta \mapsto {\rm{E}}_{m}\left[  W^{(\lambda)}\left(\hat{\mu}_n, \hat{\mu}_{\theta, m}\right) \right].$$  With the same notation as in \S \ref{Sec.EstMethod}, suppose we generate $k$ replications $Z_{1: m}^{(i)}, i = 1, \ldots, k,$ independently from a reference model $\mu_{\theta}^{(m)}$, and let $\hat{\mu}_{\theta, m}^{(i)}$ denote the empirical probability measure of $Z_{1: m}^{(i)}$. Then the loss function $$L_{\text{MERWE}} = k^{-1} \sum_{i=1}^k W^{(\lambda)}\left(\hat{\mu}_n, \hat{\mu}_{\theta, m}^{(i)}\right)$$ is a natural estimator of ${\rm E}_m [ W^{(\lambda)}\left(\hat{\mu}_n, \hat{\mu}_{\theta, m}\right)]$, since the former converges almost surely to the latter as $k \rightarrow \infty$.
We note that  the algorithmic complexity in $m$ is super-linear  while the complexity in $ k $ is linear: the incremental cost of increasing $ k $ is lower than that one of increasing $ m $. 
Moreover, a key aspect for the implementation of our estimator is related to the need for specifying $\lambda$. In the next numerical exercises we set $\lambda=5$: this value yields a good performance (in terms of Mean-Squared Error, MSE) of our estimators across the Monte Carlo settings, which consider different models and levels of contamination. In \S \ref{Sec:selection}, we describe a data driven procedure for selecting $\lambda$.

In our numerical investigation,
first, we consider a reference model which is the sum of log-normal random variables having finite moments. For a given $L \geq 1$, $\gamma \in \mathbb{R}$ and $\sigma>0$, we have $X = \sum_{l=1}^L \exp \left(Z_{l}\right),$
where $Z_1, \ldots, Z_L$ are sampled from $\mathcal{N}\left(\gamma, \sigma^2\right)$ independently. 
Suppose we are interested in estimating the location parameter $\gamma$. We investigate the performance of  MERWE and MEWE, under different scenarios: namely,  
in presence of outliers, for different sample sizes and  contamination values. Specifically, we consider the model
$$ X_i\n = \sum_{l=1}^L \exp(Z_{i_l}\n), l=1, \ldots, L, i = 1, \ldots, n,$$
where $Z_{i_l}\n \sim \mathcal{N}(\gamma,\sigma)$ for $i = 1, \ldots, n_1$ (clean part of the sample) and $Z_{i_l}\n \sim \mathcal{N}(\gamma + \eta,\sigma)$ for $i = n_1+1, \ldots, n$ (contaminated part of the sample). Therefore, in each sample, there are $n-n_1$ outliers of size $\eta$. 

In our simulation experiments, we set $ L=10 $, $\gamma = 0$ and $ \sigma=1 $. To implement the MERWE and MEWE of $\gamma$, we choose $m=1000 $, $k=20$ and $\lambda=5$.
The bias and MSE, based on $1000$ replications, of the estimators are displayed in Table~\ref{tab:Estimator}, for various sample sizes $n$, different contamination size $\eta$ and proportion of contamination $\varepsilon$.
The table illustrates the superior performance (both in terms of bias and MSE) of the MERWE with respect to the MEWE. In small samples ($n=100$), the MERWE has smaller bias and MSE than the MEWE, in all settings.  Similar results are available in moderate and large sample size $(n=200$ and $n=1000$). Interestingly, for $n=1000$, MERWE and MEWE have similar performance when $ \varepsilon=0 $ (no contamination), whilst the MERWE still has smaller MSE for $\varepsilon>0$. This implies that the MERWE maintains good efficiency with respect to MEWE at the reference model. 

\begin{table}[htbp]
	\centering
	\scalebox{0.6}[0.65]{
		\begin{tabular}{|c|c|c|c|c|c|c|c|c|c|c|c|c|}
			\hline
			\multicolumn{1}{|c|}{\multirow{3}[6]{*}{SETTINGS}} & \multicolumn{4}{c|}{$n=100$}    & \multicolumn{4}{c|}{$n=200$}   & \multicolumn{4}{c|}{$n=1000$} \bigstrut\\
			\cline{2-13}          & \multicolumn{2}{c|}{BIAS} & \multicolumn{2}{c|}{MSE} & \multicolumn{2}{c|}{BIAS} & \multicolumn{2}{c|}{MSE} & \multicolumn{2}{c|}{BIAS} & \multicolumn{2}{c|}{MSE} \bigstrut\\
			\cline{2-13}          & \multicolumn{1}{c|}{MERWE} & \multicolumn{1}{c|}{MEWE} & \multicolumn{1}{c|}{MERWE} & \multicolumn{1}{c|}{MEWE} & \multicolumn{1}{c|}{MERWE} & \multicolumn{1}{c|}{MEWE} & \multicolumn{1}{c|}{MERWE} & \multicolumn{1}{c|}{MEWE} & \multicolumn{1}{c|}{MERWE} & \multicolumn{1}{c|}{MEWE} & \multicolumn{1}{c|}{MERWE} & \multicolumn{1}{c|}{MEWE} \bigstrut\\
			\hline
			$ \varepsilon=0.1,\eta=1 $ &   0.049  & 0.092 & 0.003 & 0.010 &0.042 & 0.093 & 0.002 & 0.012  & 0.037 & 0.086 & 0.002 & 0.008  \bigstrut\\
			\hline
			$ \varepsilon=0.1,\eta=4  $&  0.035 & 0.090 & 0.002 & 0.012& 0.029 & 0.097 & 0.001 & 0.016   & 0.013 & 0.100 & $\approx 0$ & 0.018 \bigstrut\\
			\hline
			$ \varepsilon=0.2,\eta=1 $ &0.071  & 0.157 & 0.008 & 0.028  &  0.086 & 0.178 & 0.009 & 0.033  &  0.081 & 0.172 & 0.007& 0.031 \bigstrut\\
			\hline
			$ \varepsilon=0.2,\eta=4 $ & 0.046  &0.204 & 0.003& 0.045 & 0.035 & 0.203 & 0.002 & 0.043 & 0.017 & 0.195 & $\approx 0$ & 0.038\bigstrut\\
			\hline
			$ \varepsilon=0 $  & 0.036  & 0.034 & 0.002 & 0.002 & 0.022 & 0.022 &  0.001 &  0.001 &  0.012 & 0.010 & $\approx 0$ & $\approx 0$ \bigstrut\\
			\hline
		\end{tabular}%
	}
	\caption{The bias and MSE of the MERWE and MEWE of $\gamma$ for different $n$, $\varepsilon$ and $\eta$.}
\label{tab:Estimator}%
\end{table}%

Now, let us turn to the case of random variables having infinite moments. 
\cite{yatracos2022limitations} recently illustrates that minimum Wasserstein distance estimators do not perform well for heavy-tailed distributions and suggests to avoid the use of minimum Wasserstein distance inference when the underlying distribution has infinite $ p $-th moment, with $p\geq 1$. 
The theoretical developments of Section \ref{Sect_Wl} show that our MERWE does not suffer from the same criticism. In the next MC experiment, we illustrate  the good performance of MERWE when the data generating process does not admit finite first moment. We consider the problem of estimating the location parameters for a symmetric $\alpha$-stable distribution; see  e.g. \cite{samorodnitsky2017stable} for a book-length presentation.   A stable distribution is characterized by four parameter $ (\alpha,\beta,\gamma,\delta) $: $ \alpha $ is the index parameter, $ \beta $ is the skewness parameter, $ \gamma $ is the scale parameter and $ \delta $ is the location parameter. It is worth noting that stable distributions have undefined variance for $ \alpha <2 $, and undefined mean for $ \alpha \leq 1 $. We consider three parameters setting: 
\begin{itemize}
\item[(1)] $ (\alpha,\beta,\gamma,\delta) = (0.5,0,1,0) $, which represents a heavy-tailed distribution without defined mean; 
\item[(2)]  $ (\alpha,\beta,\gamma,\delta) = (1,0,1,0) $ which is the standard Cauchy distribution, having undefined moment of order $ p \geq 1$; 
\item[(3)] $ (\alpha,\beta,\gamma,\delta) = (1.1,0,1,0) $ representing a distribution having a finite mean. 
\end{itemize}
In each MC experiment, we estimate the location parameter, while the other parameters are supposed to be known. In addition, we consider contaminated  heavy-tailed data, where $(1-\varepsilon)$ proportion of $n$ observations is generated from $\alpha$-stable distribution with parameter $ (\alpha,\beta,\gamma,\delta) $ and the other $ \varepsilon $ proportion (outliers) comes from the distribution with parameter $ (\alpha,\beta,\gamma,\delta+\eta)  $  ($ \eta $ is the size of outliers).
We set $ m=1000$, $n =100$ and $k=20$ and repeat the experiment 1000 times  for each distribution and estimator, and $\lambda=5$.  We display the results in Table~\ref{tab:locationofheavytail}.
For the stable distributions, the  MEWE has larger bias and MSE than the ones yielded by the MERWE. This aspect is particularly evident for the distributions with undefined first moment, namely the Cauchy distribution and the stable distribution with $\alpha = 0.5$. These experiments, complementing the ones available in \cite{yatracos2022limitations}, illustrate that while MEWE (which is not well-defined in the considered setting) entails large bias and MSE values, MERWE is well-defined and performs well even for stable distributions with infinite moments.

\begin{table}[htbp]
	\centering
	\scalebox{0.6}[0.6]{
		\begin{tabular}{|c|c|c|c|c|c|c|c|c|c|c|c|c|}
			\hline
			\multicolumn{1}{|c|}{\multirow{3}[6]{*}{SETTINGS}} & \multicolumn{4}{c|}{Cauchy}    & \multicolumn{4}{c|}{Stable ($ \alpha = 0.5 $)}   & \multicolumn{4}{c|}{Stable ($ \alpha=1.1 $)} \bigstrut\\
			\cline{2-13}          & \multicolumn{2}{c|}{BIAS} & \multicolumn{2}{c|}{MSE} & \multicolumn{2}{c|}{BIAS} & \multicolumn{2}{c|}{MSE} & \multicolumn{2}{c|}{BIAS} & \multicolumn{2}{c|}{MSE} \bigstrut\\
			\cline{2-13}          & \multicolumn{1}{c|}{MERWE} & \multicolumn{1}{c|}{MEWE} & \multicolumn{1}{c|}{MERWE} & \multicolumn{1}{c|}{MEWE} & \multicolumn{1}{c|}{MERWE} & \multicolumn{1}{c|}{MEWE} & \multicolumn{1}{c|}{MERWE} & \multicolumn{1}{c|}{MEWE} & \multicolumn{1}{c|}{MERWE} & \multicolumn{1}{c|}{MEWE} & \multicolumn{1}{c|}{MERWE} & \multicolumn{1}{c|}{MEWE} \bigstrut\\
			\hline
			$ \varepsilon=0.1,\eta=1 $ &    0.084   &1.531    & 0.011     & 3.628  &0.088     & 3.179     &  0.011    & 13.731  & 0.090    & 0.658     &  0.011     & 1.030  \bigstrut\\
			\hline
			$ \varepsilon=0.1,\eta=4  $&  0.205    &  1.529     & 0.047     & 3.656   & 0.164     & 3.174    & 0.034     &13.706   & 0.207     &  0.746     & 0.048      & 1.051 \bigstrut\\
			\hline
			$ \varepsilon=0.2,\eta=1 $ &0.181     &   1.502     &0.038      &3.602  &  0.171      & 3.155      & 0.037     &12.840  &  0.181     &  0.676     &  0.038     & 0.942 \bigstrut\\
			\hline
			$ \varepsilon=0.2,\eta=4 $ & 0.459     & 1.820      &  0.224    & 4.691  & 0.384      & 3.140      & 0.166      &  12.713 & 0.485      &1.072       &0.245       & 1.802 \bigstrut\\
			\hline
			$ \varepsilon=0 $  & 0.046  &  1.551     &  0.003     &  3.740  & 0.045    &  3.118     & 0.003    &  12.600 &  0.042      & 0.613      & 0.003      &0.894  \bigstrut\\
			\hline
		\end{tabular}%
	}
	\caption{The bias and MSE of the MERWE and MEWE of location parameter for various contaminated $\alpha$-stable distributions. Simulation setting:  $m=1000 $, $k=20$.}
	\label{tab:locationofheavytail}
\end{table}%

\subsection{Generative Adversarial Networks (GAN)} \label{Sect_DA}

\textit{Synthetic data.} We propose two  RWGAN deep learning models: both approaches are based on ROBOT. The first one is  derived directly from the dual form in Th.~\ref{thm3}, while the second one is derived from~\eqref{equ7}.  We compare these two methods with routinely-applied  Wasserstein GAN (WGAN) and with the robust WGAN introduced by \cite{balaji2020robust}.  To provide some details on our new procedures, we recall that
GAN is a deep learning method, first proposed by \cite{goodfellow2014generative}. It is one of the most popular machine learning approaches for unsupervised learning on complex distributions. 
GAN includes a generator and a discriminator that are both neural networks and the procedure is based on mutual game learning between them. 
The  generator creates fake samples as well as possible to deceive the discriminator, while the discriminator needs to be more and more powerful to detect the fake samples. The ultimate goal of this procedure is to produce a generator with a great ability to produce high-quality samples just like the original sample.  
To illustrate the connections with the ROBOT, let us consider the following GAN architecture; more details are available in the Supplementary Material (Appendix \ref{app.rwgan}). Let $ {X} \sim {\rm P_r} $ and  ${Z} \sim {\rm P}_z$, where $\rm P_r$ and ${\rm  P}_z$ denote the distribution of the reference sample and of a random sample which is the input for the generator. Then, denote by
$ G_{\theta} $ the function applied by generator (it transforms $Z$ to create fake samples, which are the output of a statistical model ${\rm P}_\theta$, indexed by the finite dimensional parameter $ \theta $) and by  
$D_{\vartheta}$ the function applied by the discriminator (it is indexed by a finite dimensional  parameter $ \vartheta $, which outputs the estimate of the probability that the sample is true).
The objective function is 
\begin{equation}\label{eq:GAN}
	\min_{\theta} \max_{\vartheta}	\left\lbrace {\rm E}[\log D_{\vartheta}({X})]+{\rm E}[\log (1-D_{\vartheta}(G_{\theta}{(Z)}))]\right\rbrace ,
\end{equation}
where $D_{\vartheta}(X)$ represents the probability that $X$ came from the data rather than $P_\theta$. The GAN mechanism trains $D_{\vartheta}$ to maximize the probability of assigning the correct label to both training examples and fake samples. Simultaneously, it trains $G_{\theta}$ to minimize $\log(1-D(G(z)))$. 
In other words, 
we would like to train $ G_{\theta} $ such that $\mathrm{P_\theta}$ is very close (in some distance/divergence) to $\mathrm{P_r} $.

Despite its popularity, GAN has some inference issues. For instance, during the training, the generator may collapse to a setting where it always produces the same samples and face the  
fast vanishing gradient problem. Thus, training GANs is  a delicate and unstable numerical and inferential task;  see  \cite{arjovsky2017wasserstein}.
To overcome these problems, Arjovsky et al.  propose the so-called Wasserstein GAN (WGAN) based on Wasserstein distance of order $1$. The main idea is still based on mutual learning, but rather than using a discriminator to predict the probability of generated images as being real or fake, the WGAN replaces  $D_{\vartheta}$ with a function $ f_{\xi} $ (it corresponds to $ \psi $ in Kantorovich-Rubenstein duality), indexed by parameter $ \xi $, which is called ``a critic'', namely a function that evaluates the realness or fakeness of a given sample. In mathematical form, the WGAN objective function is 
\begin{equation}\label{equ11}
	\min_{\theta} \sup _{\|f_{\xi}\|_L \leq 1} \left\lbrace  {\rm E}[f_{\xi}({X})]-{\rm E}[f_{\xi}(G_{\theta}({Z}))] \right\rbrace ,
\end{equation}
Also in this new formulation, the task  is still to train $ G_{\theta} $ in such a way that $\mathrm{P_\theta}$ is very close (now, in Wasserstein distance) to $\mathrm{P_r} $.
\cite{arjovsky2017wasserstein} explain how the WGAN is connected to minimum distance estimation.
We remark that \eqref{equ11} has  the same form as the  Kantorovich-Rubenstein duality:  we apply our dual form  of ROBOT in Th.~\ref{thm3} to the WGAN to obtain a new objective function
\begin{equation}\label{equ12}
	\min_{\theta}\sup _{\|f_{\xi}\|_L \leq 1,{\rm range}(f_{\xi})\leq 2\lambda } \left\lbrace  {\rm E}[f_{\xi}({X})]-{\rm E}[f_{\xi}(G_{\theta}({Z}))] \right\rbrace.
\end{equation}
The central idea is to train a RWGAN by minimizing the robust Wasserstein distance (actually, using the dual form) between real and generative data. 
To this end, we  define a novel RWGAN model based on  \eqref{equ12}. The algorithm for training this RWGAN model is very similar to the one for WGAN available in \cite{arjovsky2017wasserstein}, to which we refer for the implementation.  We label this robust GAN model  as RWGAN-1.  Besides this model, we propose another approach derived from~\eqref{equ7}, where we use a new neural network to  represent  the modified distribution and add a penalty term to \eqref{equ11}
in order to control modification. Because of space constraint, the
details of this procedure are provided in Algorithm \ref{alg:robust wgan} in Appendix~\ref{app.rwgan}. We label this new RWGAN model as RWGAN-2. Different from \cite{balaji2020robust}'s robust GAN, which uses $\chi_2$-divergence to constrain the modification of distribution, our RWGAN-2 does this by  making use of the TV. 
In the sequel, we will write RWGAN-B for the robust GAN of \cite{balaji2020robust}, which we implement  using the same set up as Balaji et al. 
We remark that the RWGAN-1 is less computationally complex than RWGAN-2 and RWGAN-B. Indeed,  RWGAN-2 and RWGAN-B make use of some  regularized terms, thus an additional neural network is needed to represent the modification of the distribution. In contrast, RWGAN-1 has a simpler structure: for its implementation, it requires only a small modification of the activation function in the generator network; see Appendix \ref{app.rwgan}. 


To investigate  the robustness of  RWGAN-1 and RWGAN-2 we rely on synthetic data. We consider reference samples generated from  a  simple model, containing $ n $  points in total, with $ n- n_1 $ outliers, whose data generation process is 
\begin{equation}\label{true model}
	\begin{array}{cc}
		& X_{i_{1}}^{(n)} \sim \mathrm{U}(0,1),X_{i_{2}}^{(n)} =  X_{i_{1}}^{(n)} + 1,\\
		& {X}_i^{(n)}=(X_{i_{1}}^{(n)},X_{i_{2}}^{(n)} ), i = 1, 2, \ldots, n_1,\\
		& {X}_i^{(n)}=(X_{i_{1}}^{(n)},X_{i_{2}}^{(n)}+\eta) , i =  n_1 +1, n_1 +2, \ldots,n,
	\end{array}
\end{equation}
with $ \eta $ representing the size of outliers. We set $ n=1000 $ and try four different settings by changing values of  $\varepsilon = (n- n_1)/n$ and $\eta$. As it is common in the GAN literature, our generative adversarial  models are obtained using the \texttt{Python} library \texttt{Pytorch}; see Appendix \ref{app.rwgan}.  We display the results in Figure~\ref{syndata},  where we compare WGAN, RWGAN-1, RWGAN-2,  RWGAN-B, and  RWGAN-N (based on the outlier-robust Wasserstein distance $W_1^{(\epsilon)}$
of \cite{NGC22}, with $\epsilon=0.07$ and $\epsilon=0.25$, as in Nietert et al. paper) and RWGAN-D (based on the BL distance, as in \cite{dudley1969speed}). To measure the distance between the data simulated by the generator  and the input data, we report the Wasserstein distance of order 1. For visual comparison, we display the cloud of clean data points (blue triangles) and the cloud of GAN generated points (red squares).  The plots reveal that WGAN  is greatly affected by outliers. Differently, RWGAN-2, RWGAN-B, RWGAN-N and RWGAN-D are able to generate data roughly consistent with the uncontaminated distribution in most  of the settings. Nevertheless, they still produce  some obvious abnormal points, especially when the proportion and size of outliers increase. 
In particular, we notice that RWGAN-D performs well in the first two contamination settings ($\varepsilon = 0.1$, $\eta= 2$ and $\varepsilon = 0.1$, $\eta= 3$)
but it breaks down (namely, it generates points which are very different from the uncontaminated data) in the third and fourth contamination setting ($\varepsilon = 0.2$, $\eta= 2$ and $\varepsilon = 0.2$, $\eta= 3$), where the amount and the size of contamination increase and affect negatively the GAN performance.
In a very different way, RWGAN-1 performs better than its competitors and generates data that agree with the uncontaminated distribution, even when the proportion and size of outliers are large.  Several repetitions of the experiment lead to the same conclusions.

To conclude this experiment, we remark that RWGAN-N is very sensitive to the specified value of $\epsilon$. For instance, in the first contaminations setting, comparing the $W_1$ between the clean and the generated data by RWGAN-N  with $\epsilon=0.07$ to the $W_1$ of the RWGAN-1 illustrates that our method (whose $\lambda$ has been selected using the procedure described in Section \ref{Sec:selection}) performs better. Increasing $\epsilon$ to 0.25 improves on the RWGAN-N performance, making its $W_1$ closer to the one of the RWGAN-1. This is an important methodological consideration, which has a practical implication for the application of the RWGAN-N: a selection criterion
for the $\epsilon$ hyper-parameter needs to be proposed, discussed and studied similarly to what we do in Section \ref{Sec:selection}. To the best of our knowledge, such a criterion is not available in the literature. \\


	%

\begin{figure}[http]
	\centering
	\begin{subfigure}[b]{0.24\textwidth}
		\centering
		\includegraphics[width=\textwidth]{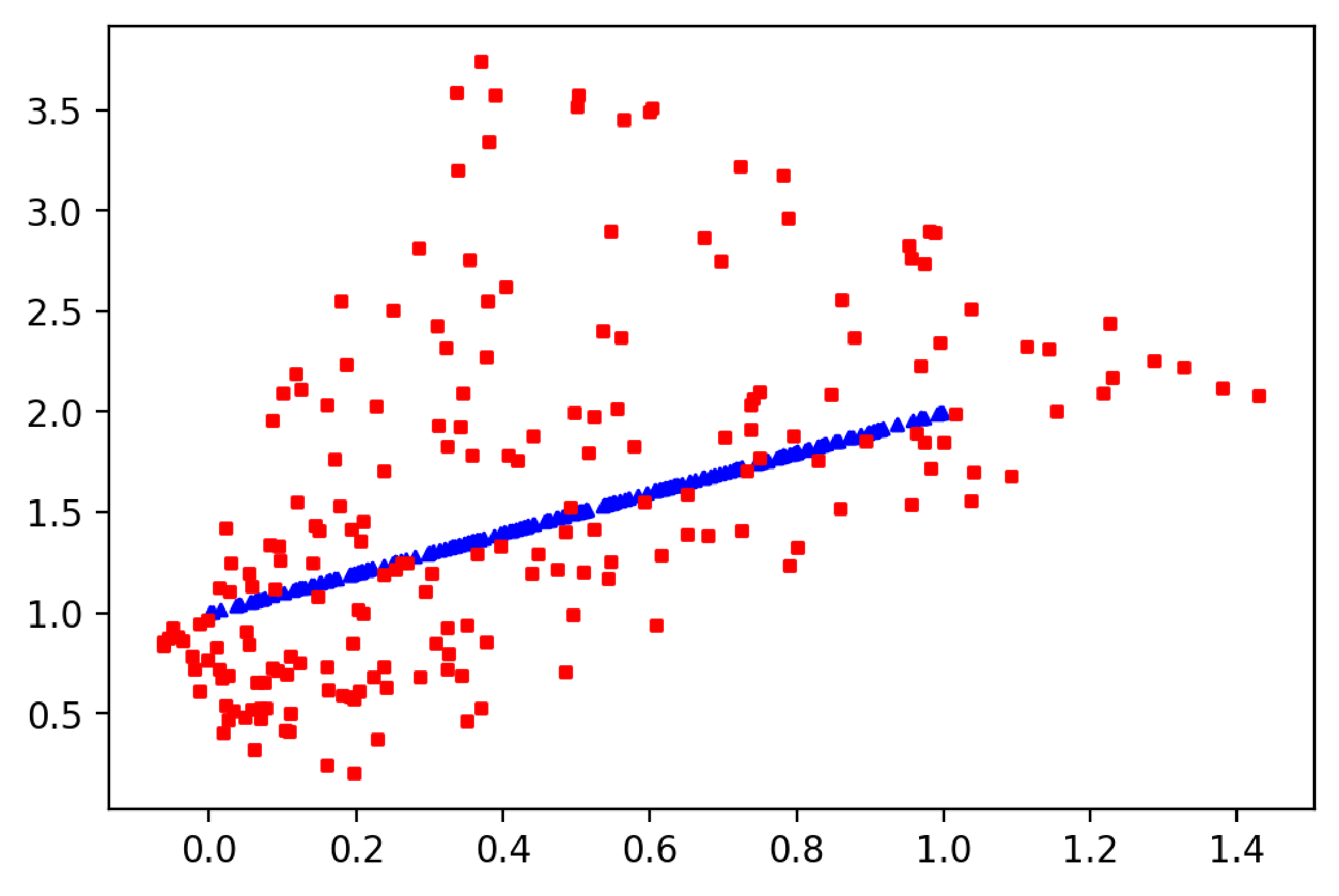}
		\caption*{ $ W_1= $  0.5864 }
	\end{subfigure}	
	\begin{subfigure}[b]{0.24\textwidth}
		\centering
		\includegraphics[width=\textwidth]{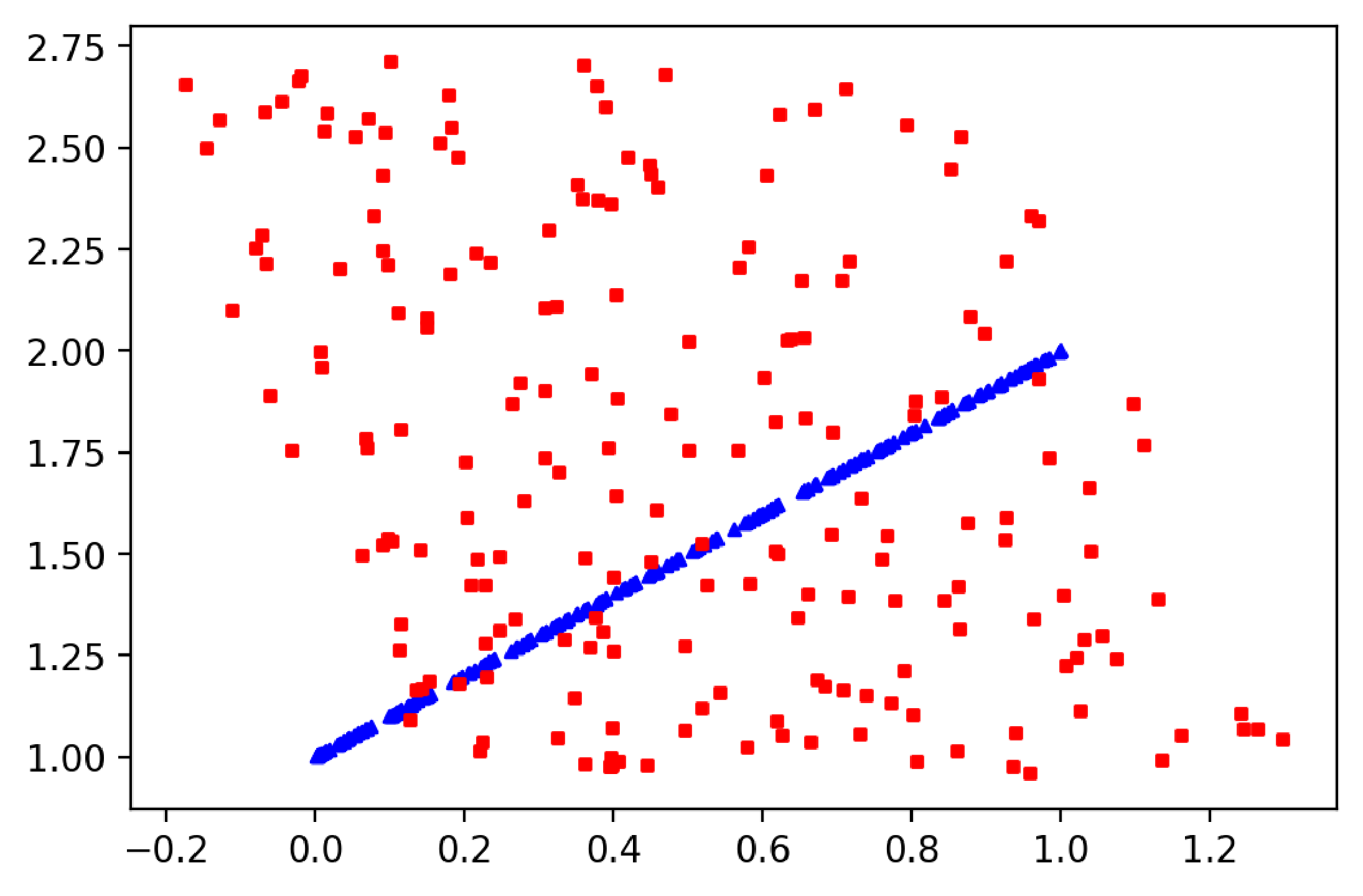}
		\caption*{$ W_1= $  0.5229}
	\end{subfigure}
	\begin{subfigure}[b]{0.24\textwidth}
	\centering
	\includegraphics[width=\textwidth]{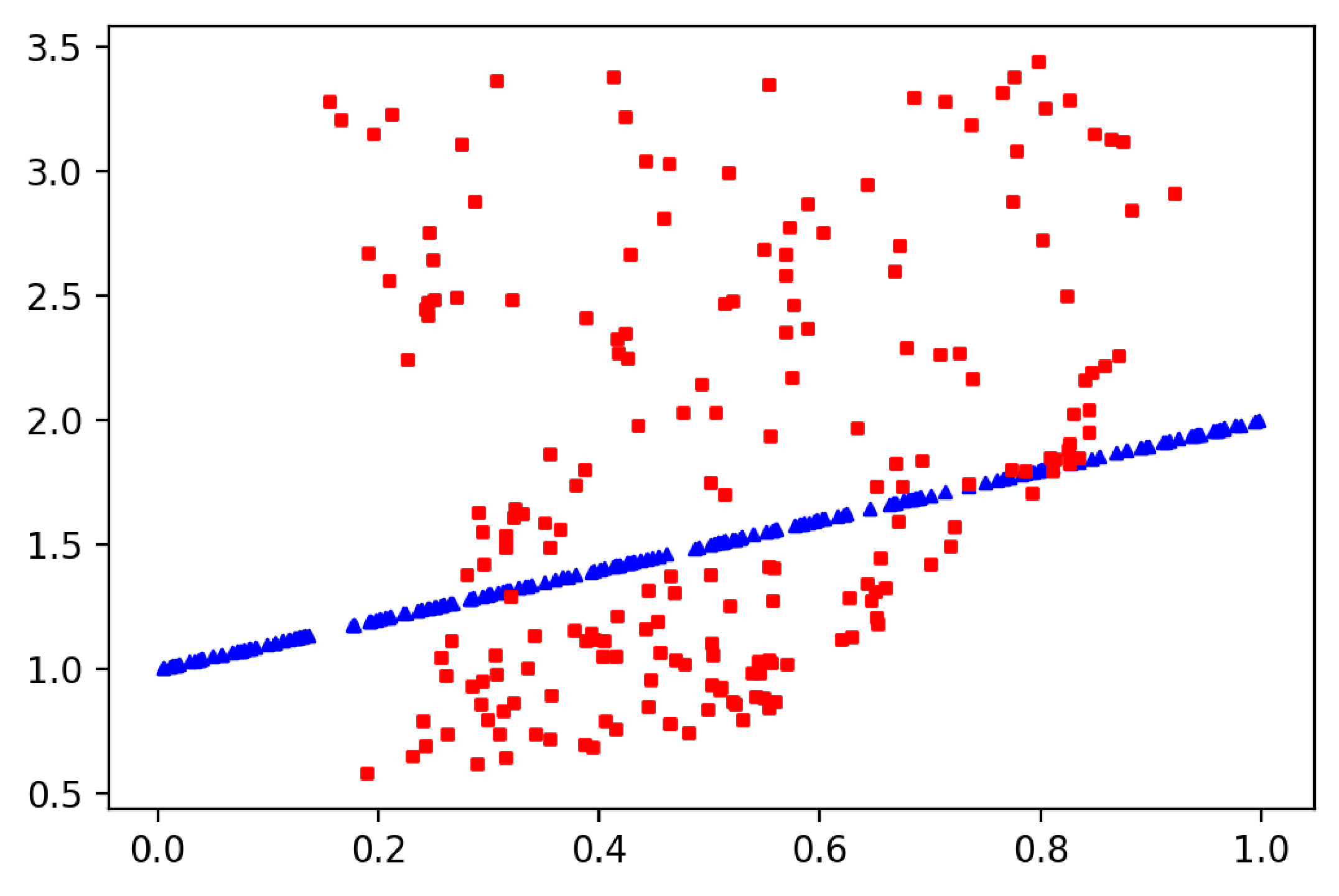}
	\caption*{ $ W_1= $ 0.5646}
\end{subfigure}
\begin{subfigure}[b]{0.24\textwidth}
	\centering
	\includegraphics[width=\textwidth]{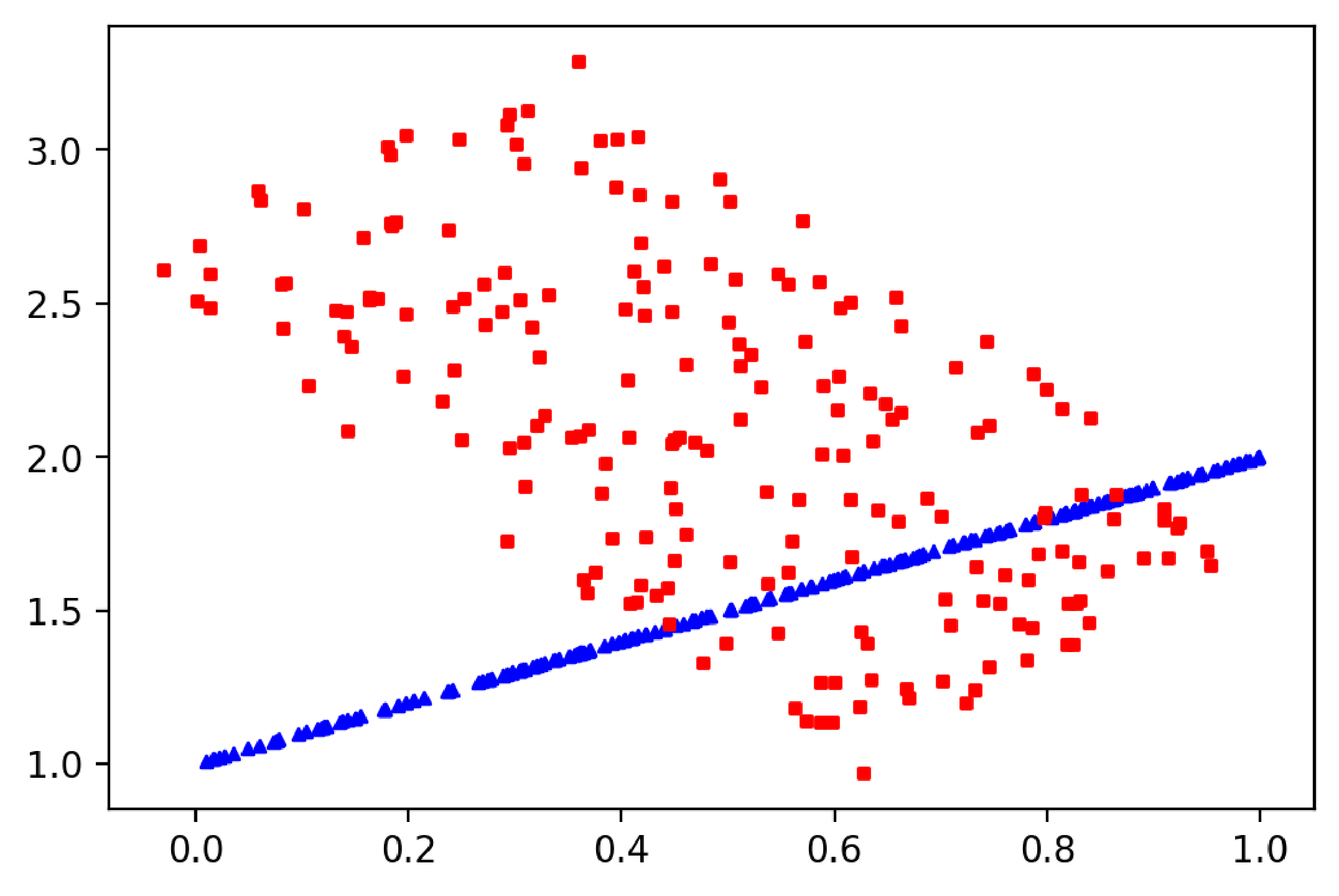}
	\caption*{ $ W_1= $ 0.6518 }
\end{subfigure}
		\\
	\begin{subfigure}[b]{0.24\textwidth}
		\centering
		\includegraphics[width=\textwidth]{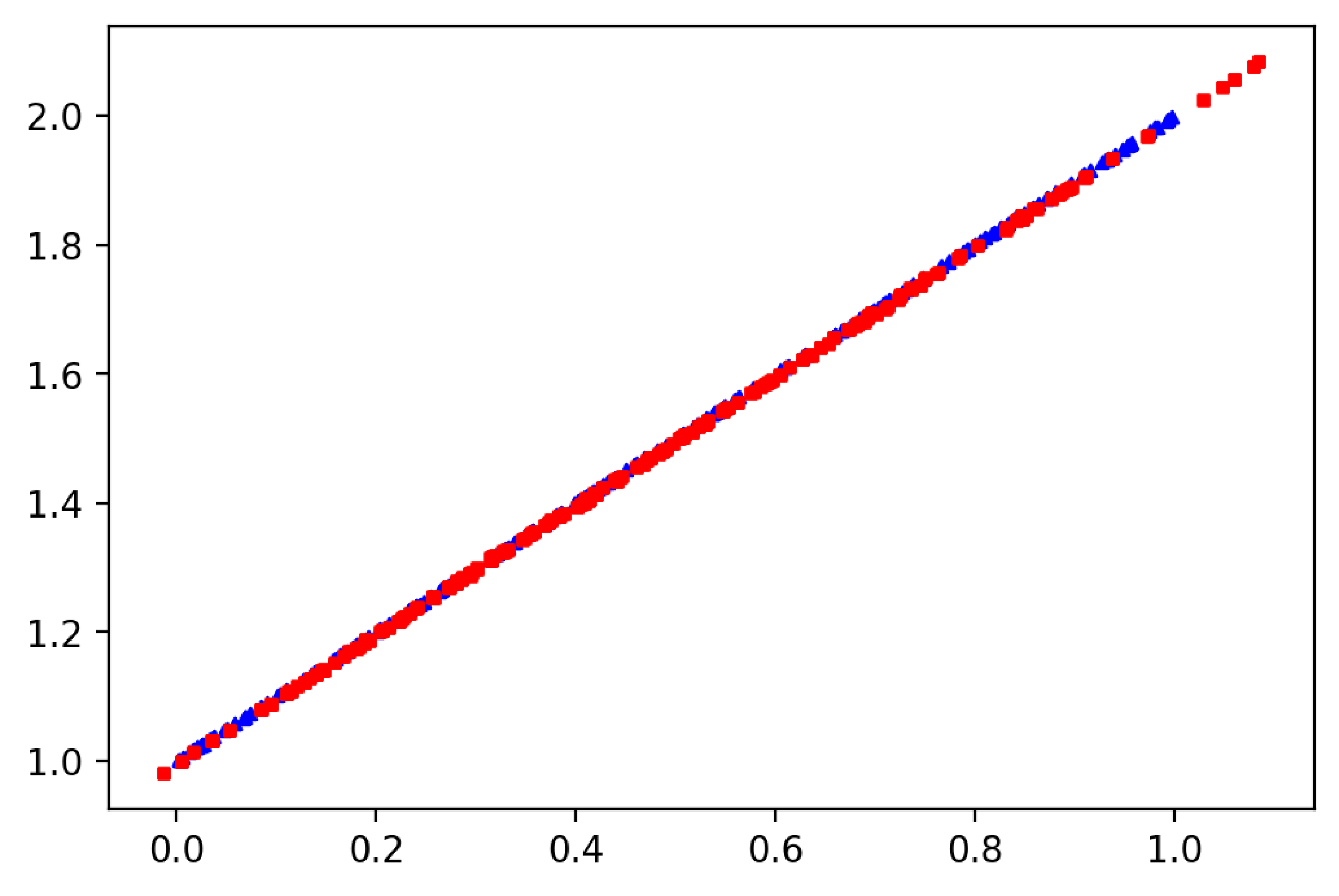}
		\caption*{$ W_1= $  0.0514 }
	\end{subfigure}	
	\begin{subfigure}[b]{0.24\textwidth}
		\centering
		\includegraphics[width=\textwidth]{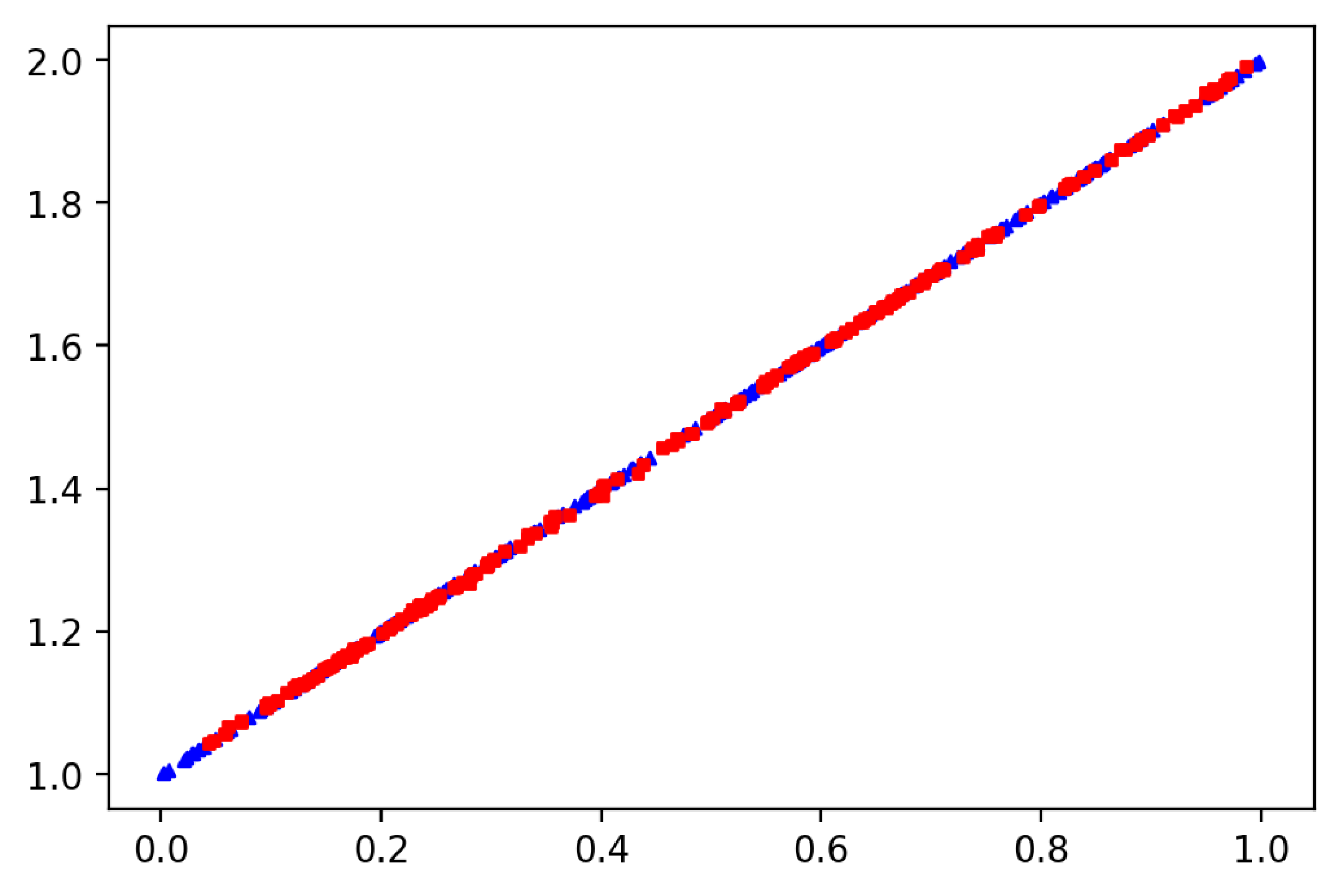}
		\caption*{$ W_1= $ 0.0495 }
	\end{subfigure}
		\begin{subfigure}[b]{0.24\textwidth}
		\centering
		\includegraphics[width=\textwidth]{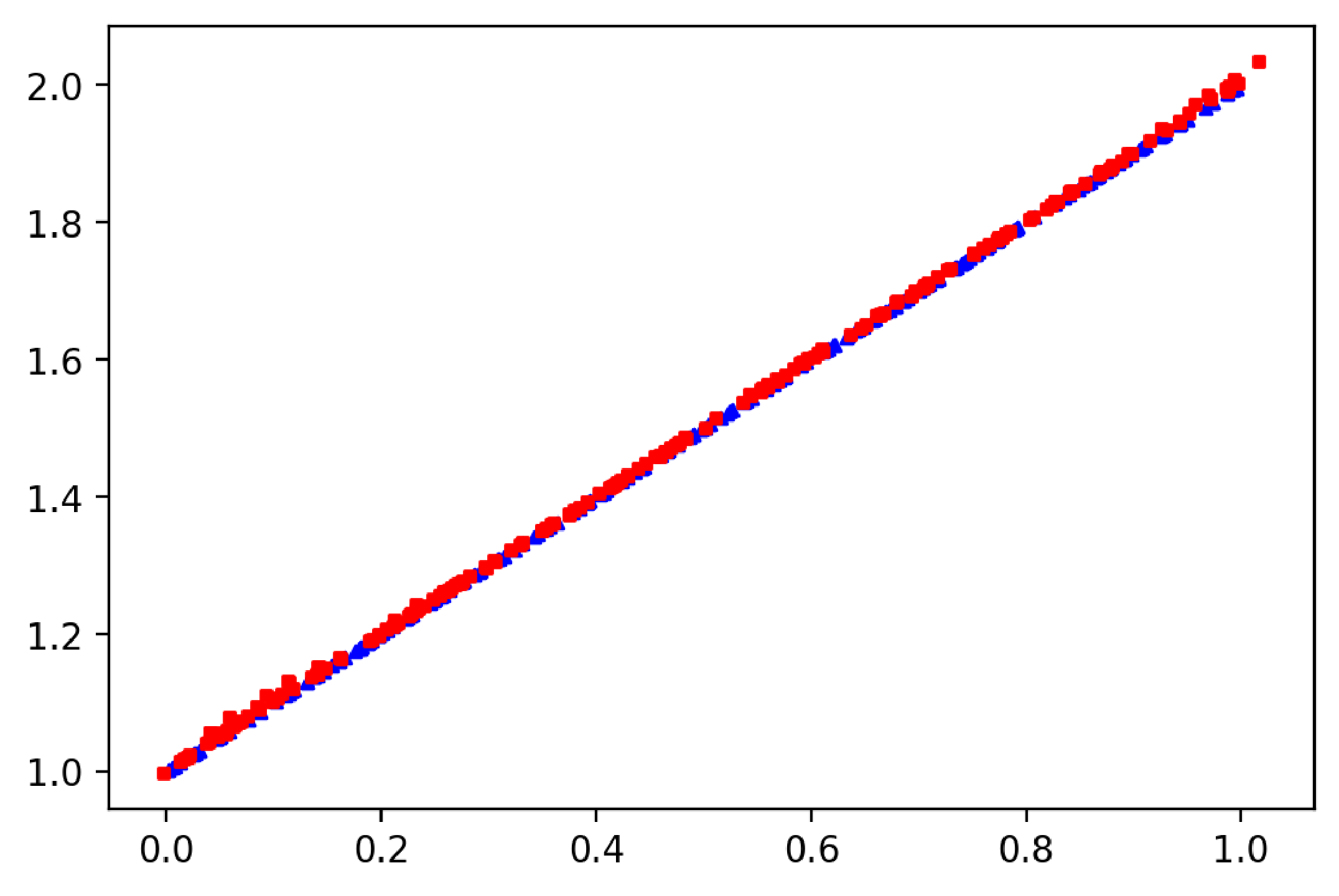}
		\caption*{ $ W_1= $ 0.0470 }
	\end{subfigure}
	\begin{subfigure}[b]{0.24\textwidth}
	\centering
	\includegraphics[width=\textwidth]{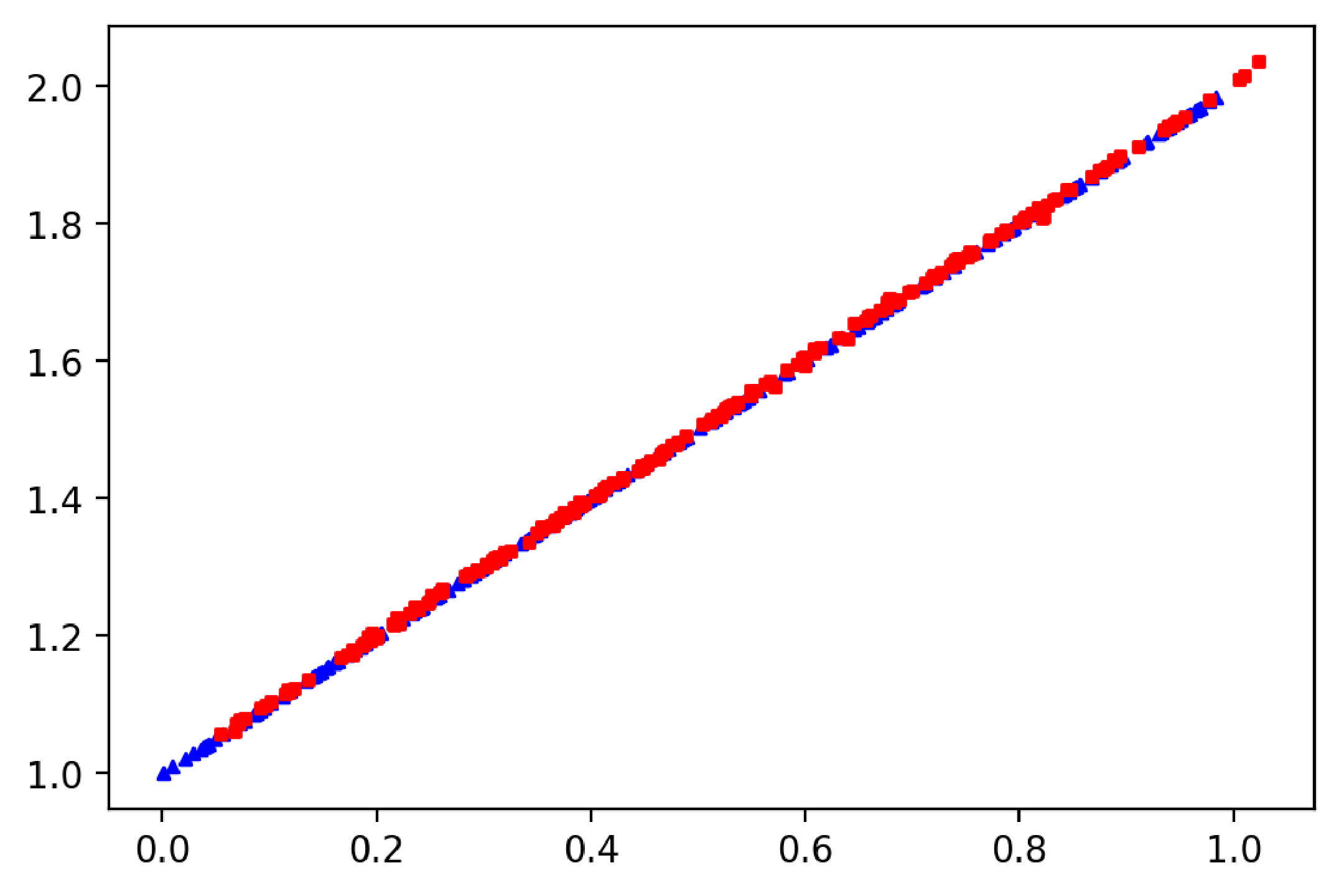}
	\caption*{ $ W_1= $ 0.0471 }
\end{subfigure}
	\\
	\begin{subfigure}[b]{0.24\textwidth}
		\centering
		\includegraphics[width=\textwidth]{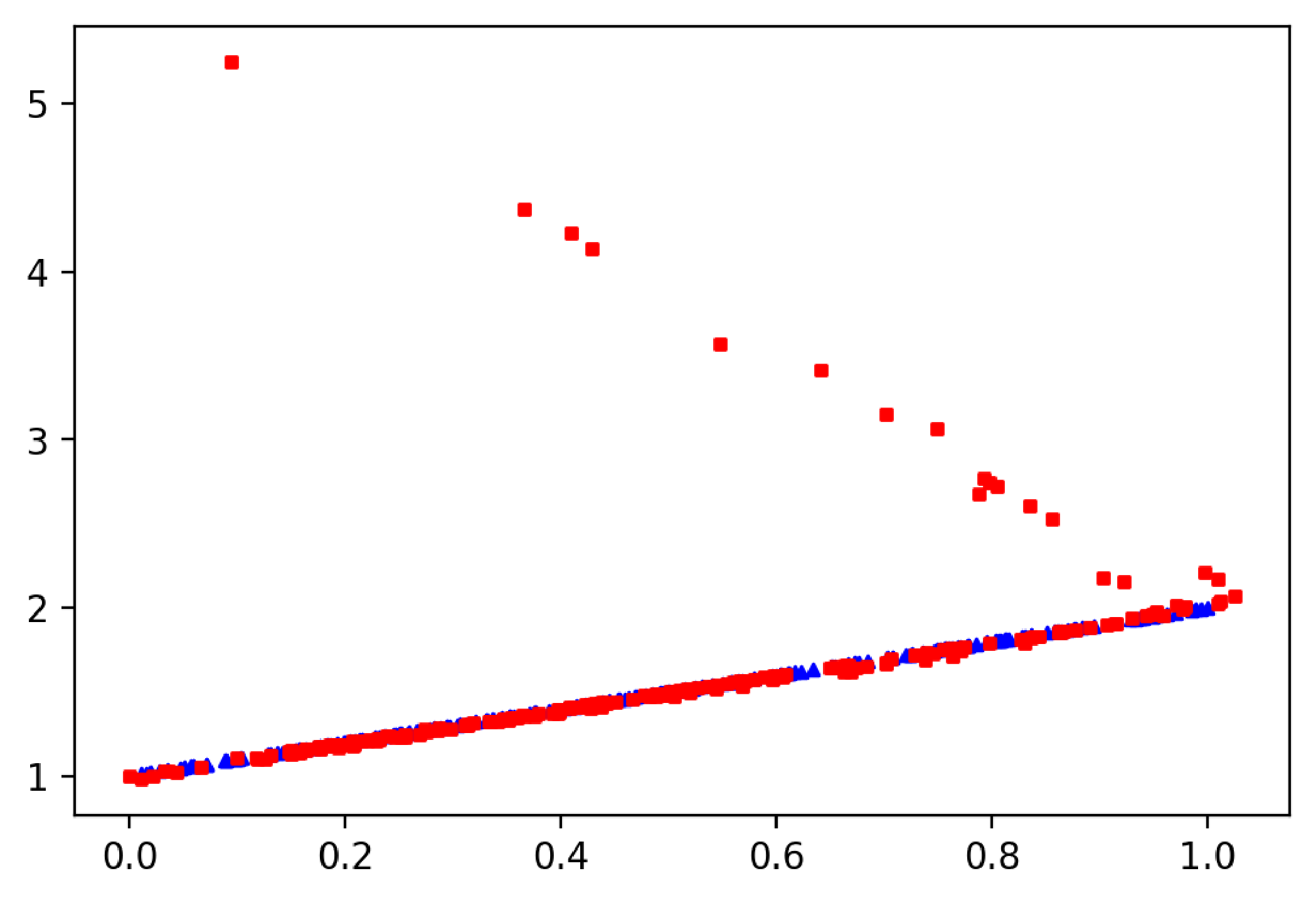}
		\caption*{$ W_1= $ 0.1560}
	\end{subfigure}	
	\begin{subfigure}[b]{0.24\textwidth}
		\centering
		\includegraphics[width=\textwidth]{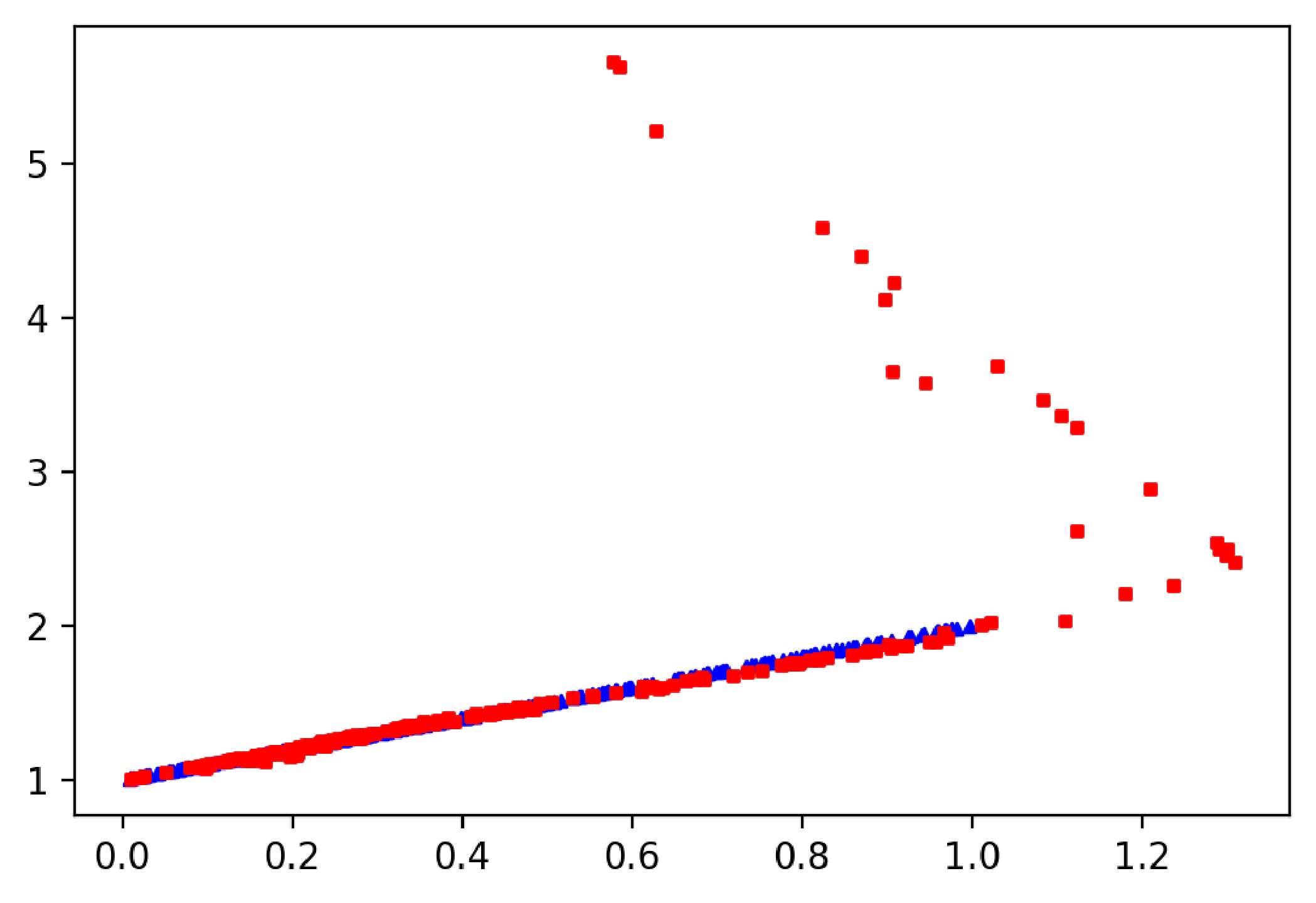}
		\caption*{ $ W_1= $ 0.2312 }
	\end{subfigure}
	\begin{subfigure}[b]{0.24\textwidth}
		\centering
		\includegraphics[width=\textwidth]{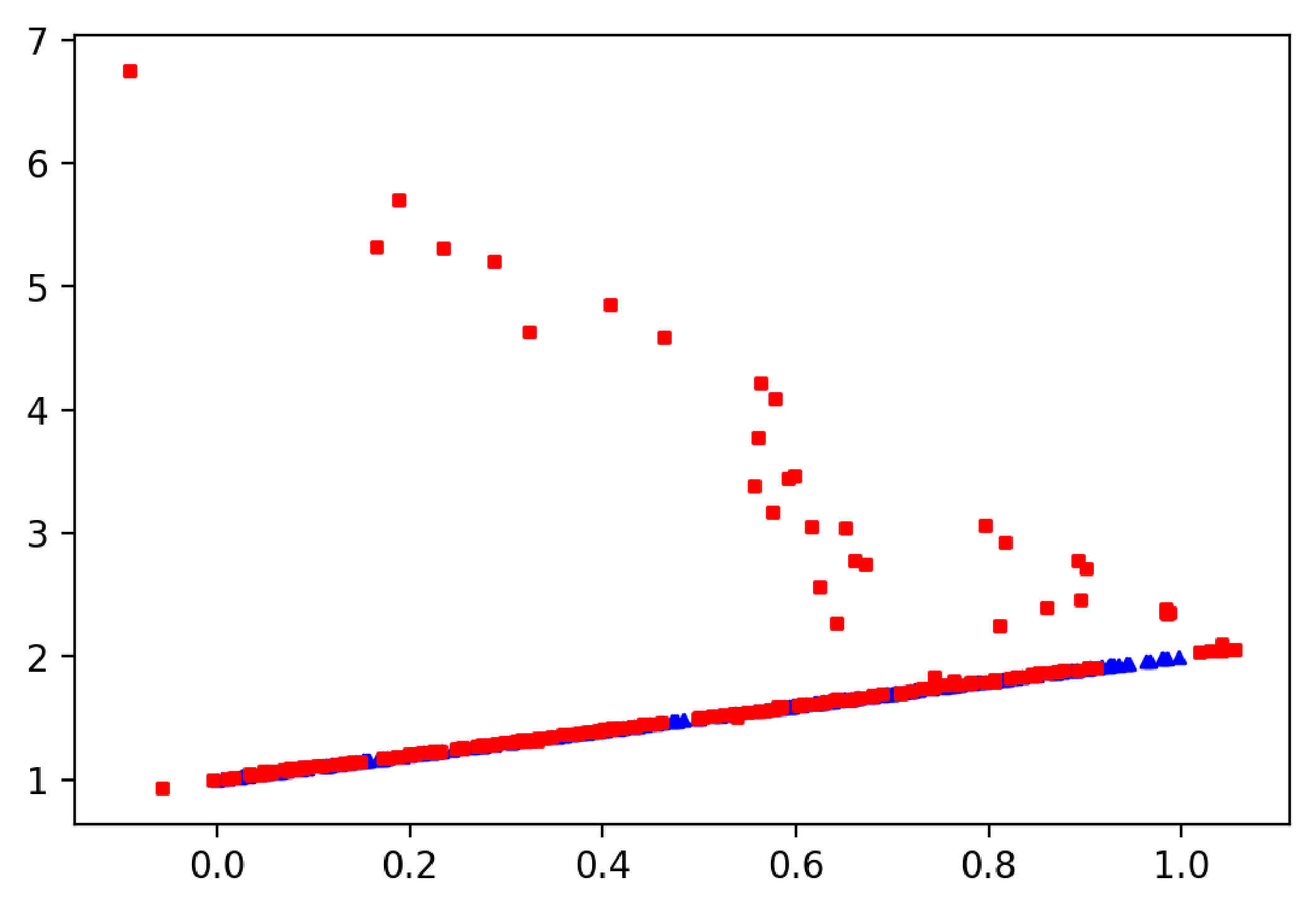}
		\caption*{ $ W_1= $ 0.2938 }
	\end{subfigure}
	\begin{subfigure}[b]{0.24\textwidth}
		\centering
		\includegraphics[width=\textwidth]{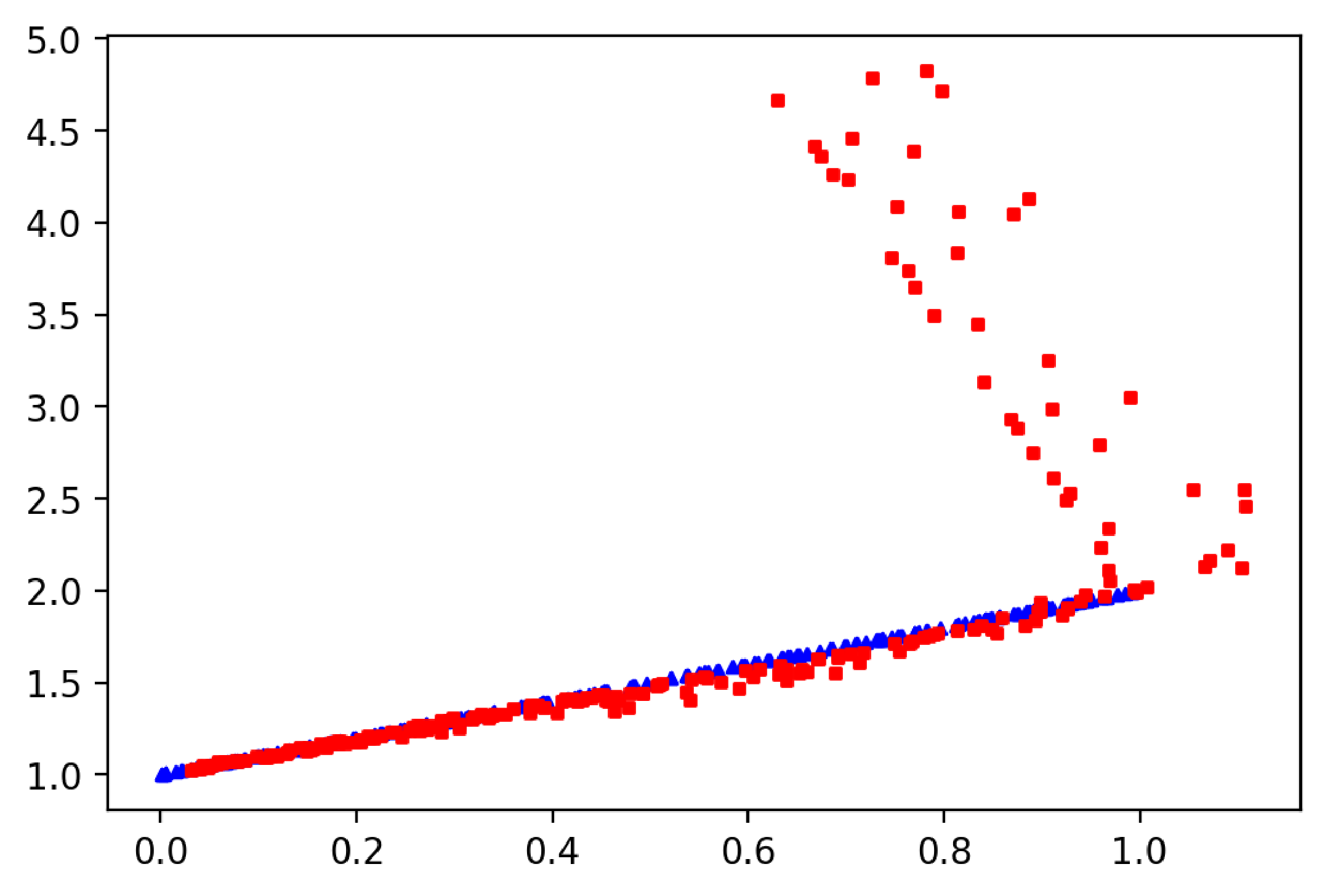}
		\caption*{ $ W_1= $ 0.3954}
	\end{subfigure}
\\
	\begin{subfigure}[b]{0.24\textwidth}
		\centering
		\includegraphics[width=\textwidth]{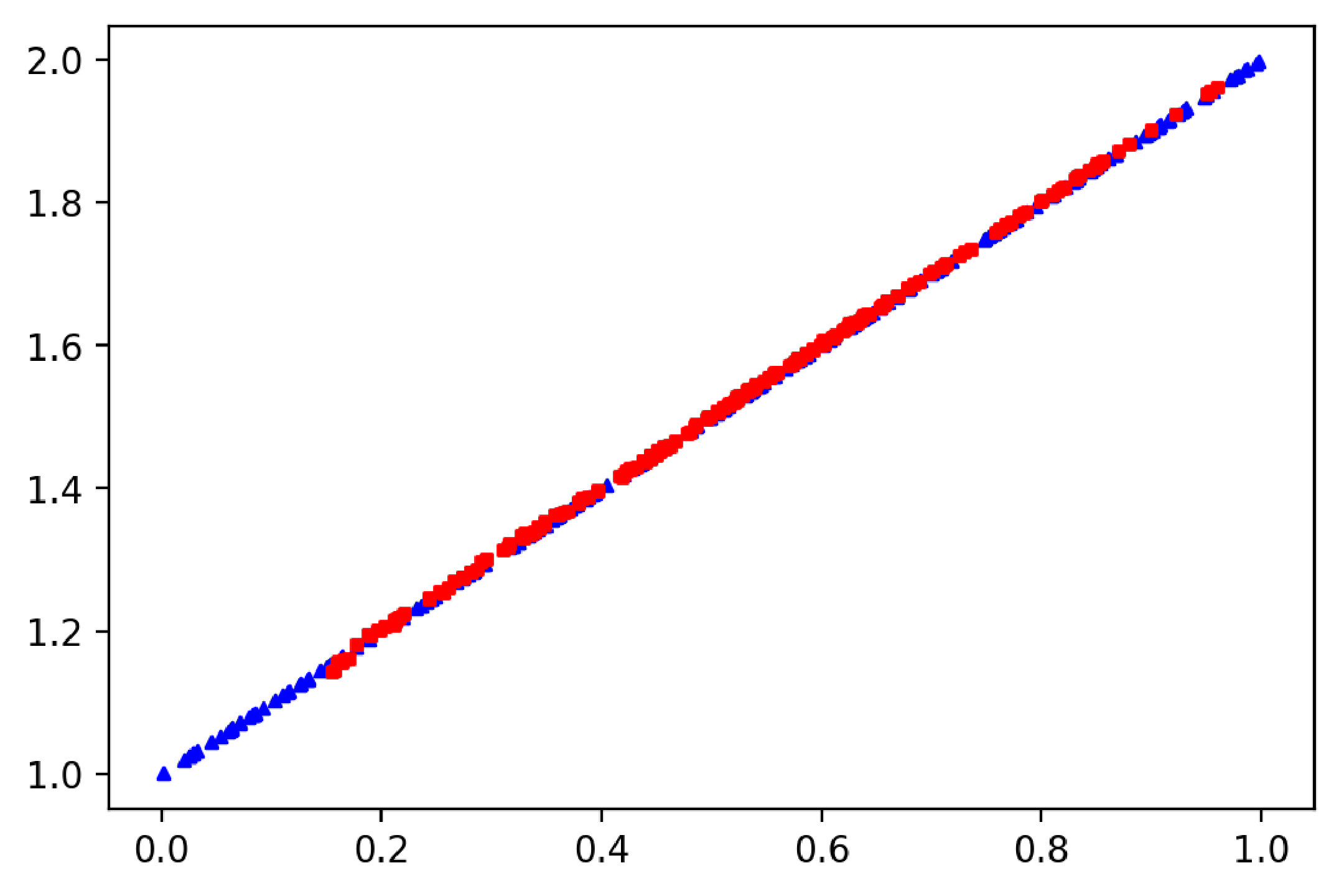}
		\caption*{ $ W_1= $  0.1102 }
	\end{subfigure}	
	\begin{subfigure}[b]{0.24\textwidth}
	\centering
	\includegraphics[width=\textwidth]{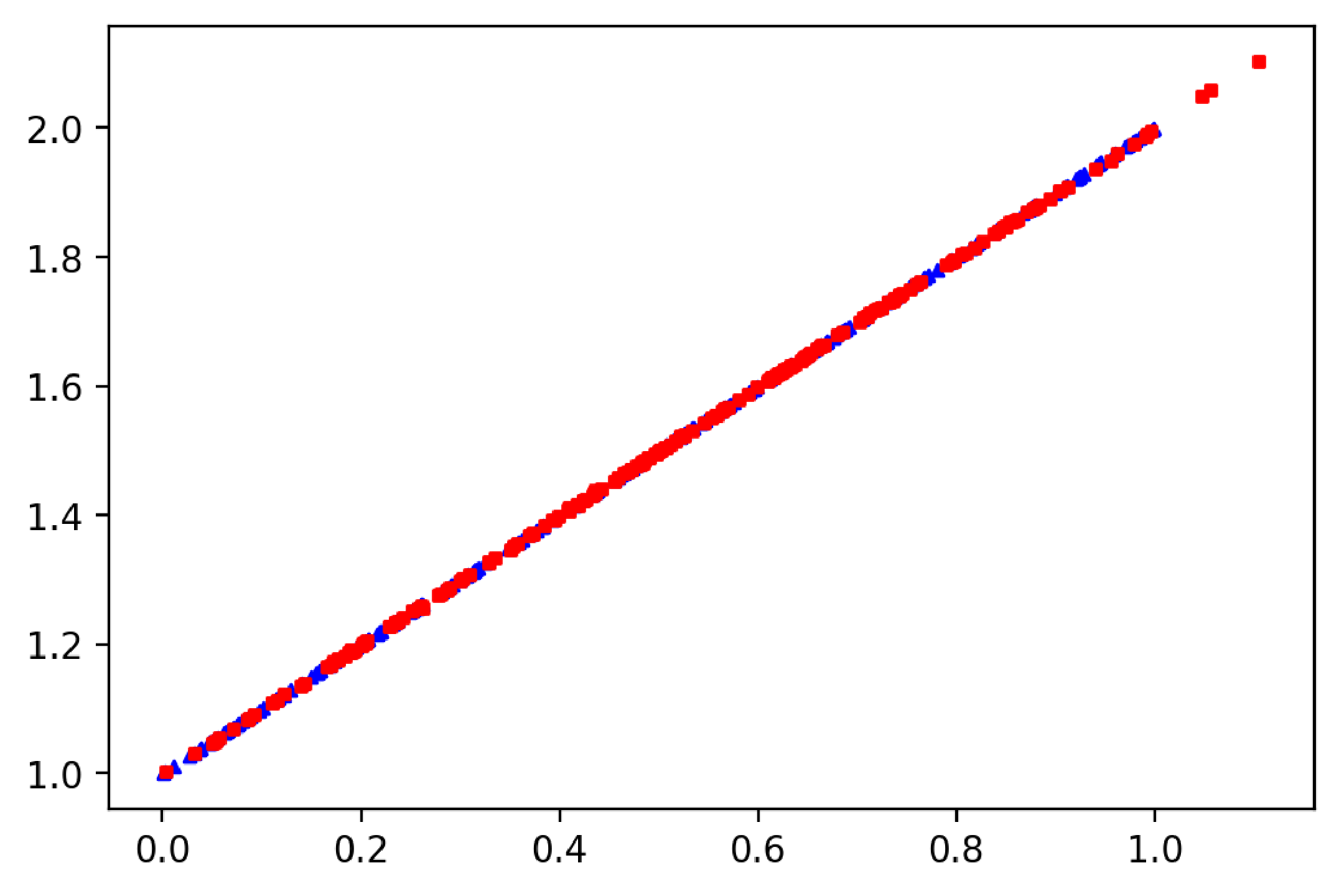}
	\caption*{ $ W_1= $  0.0983 }
\end{subfigure}
		\begin{subfigure}[b]{0.24\textwidth}
		\centering
		\includegraphics[width=\textwidth]{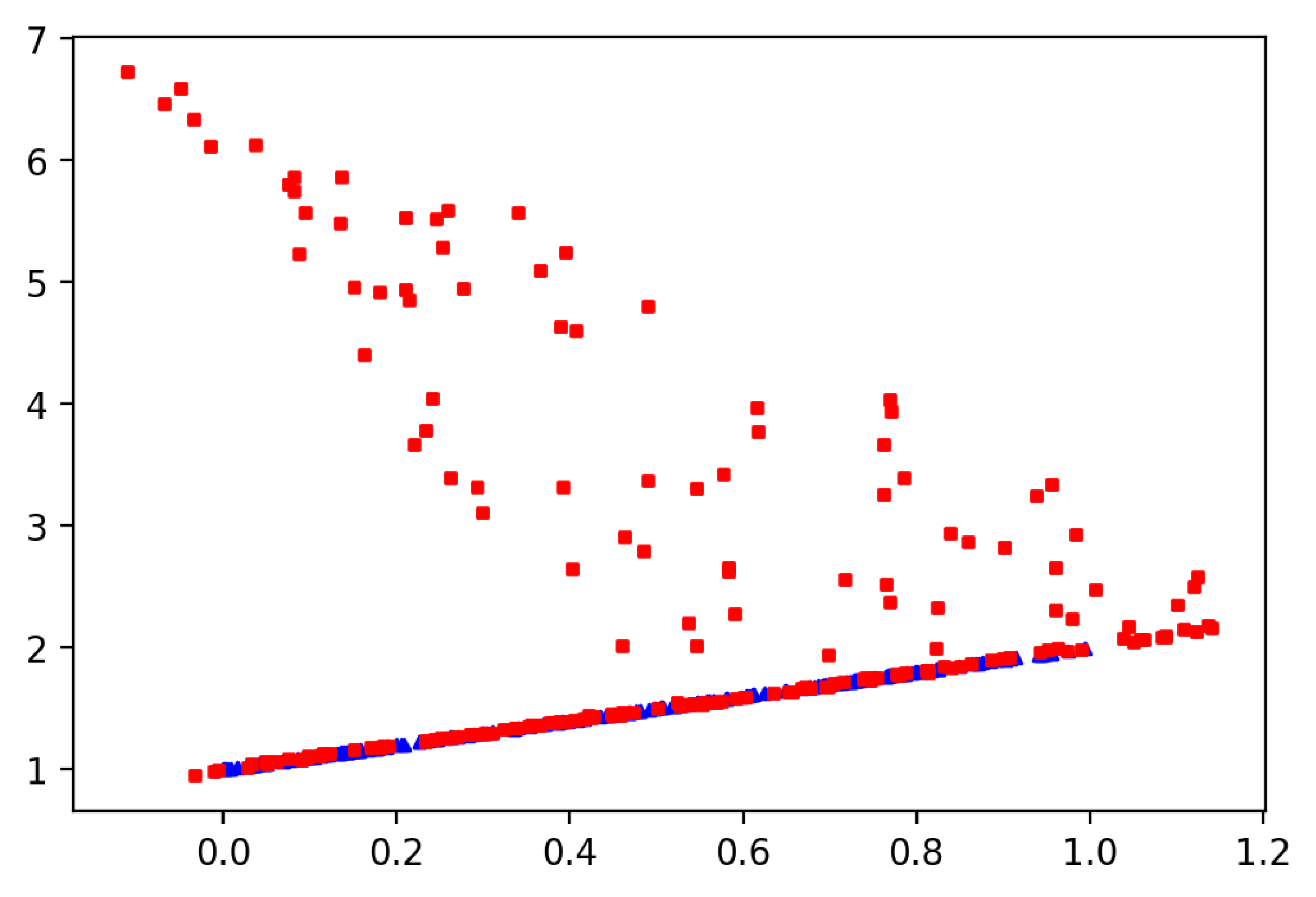}
		\caption*{ $ W_1= $ 0.3229 }
	\end{subfigure}
	\begin{subfigure}[b]{0.24\textwidth}
		\centering
		\includegraphics[width=\textwidth]{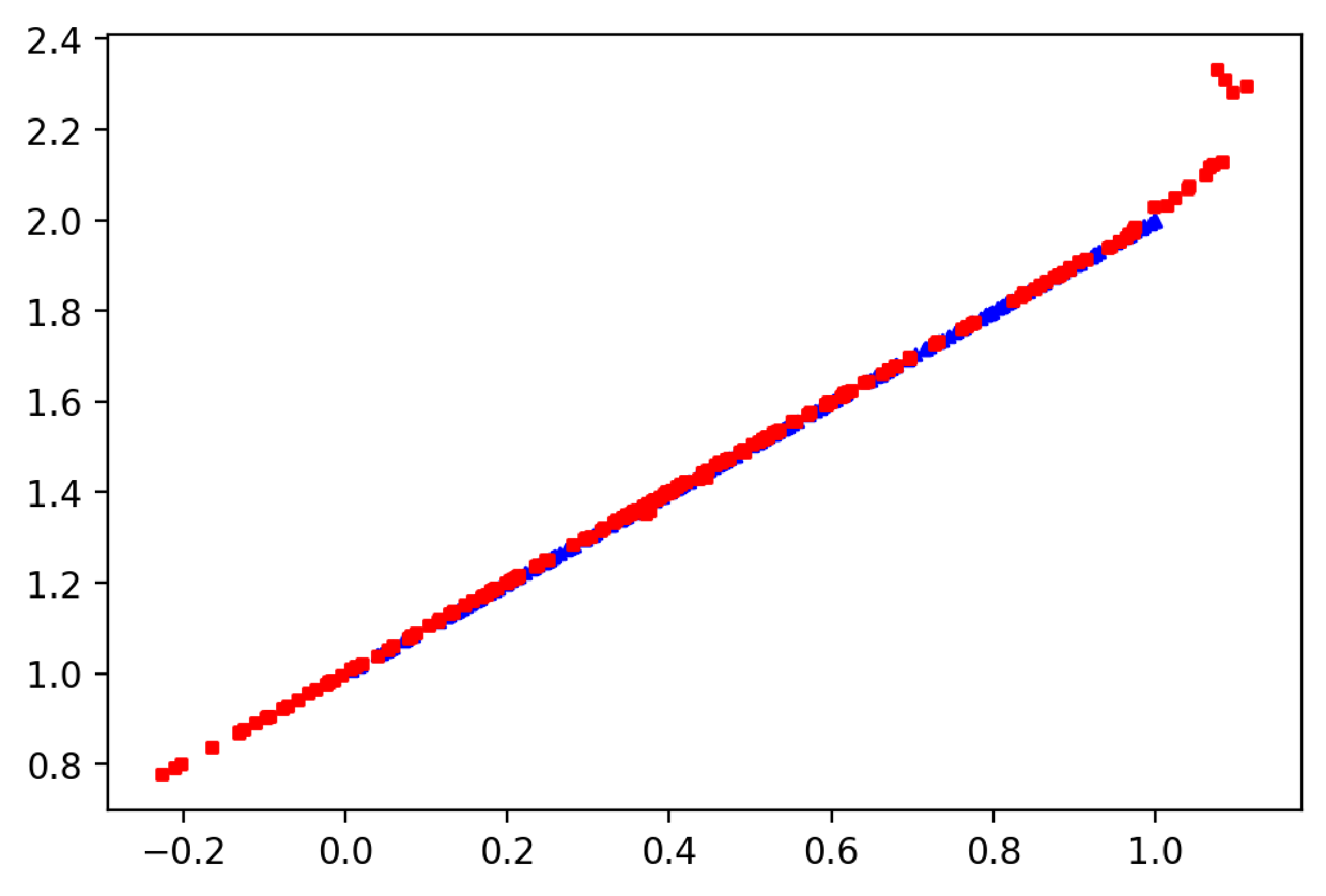}
		\caption*{ $ W_1= $ 0.1206}
	\end{subfigure}
\\
	\begin{subfigure}[b]{0.24\textwidth}
	\centering
	\includegraphics[width=\textwidth]{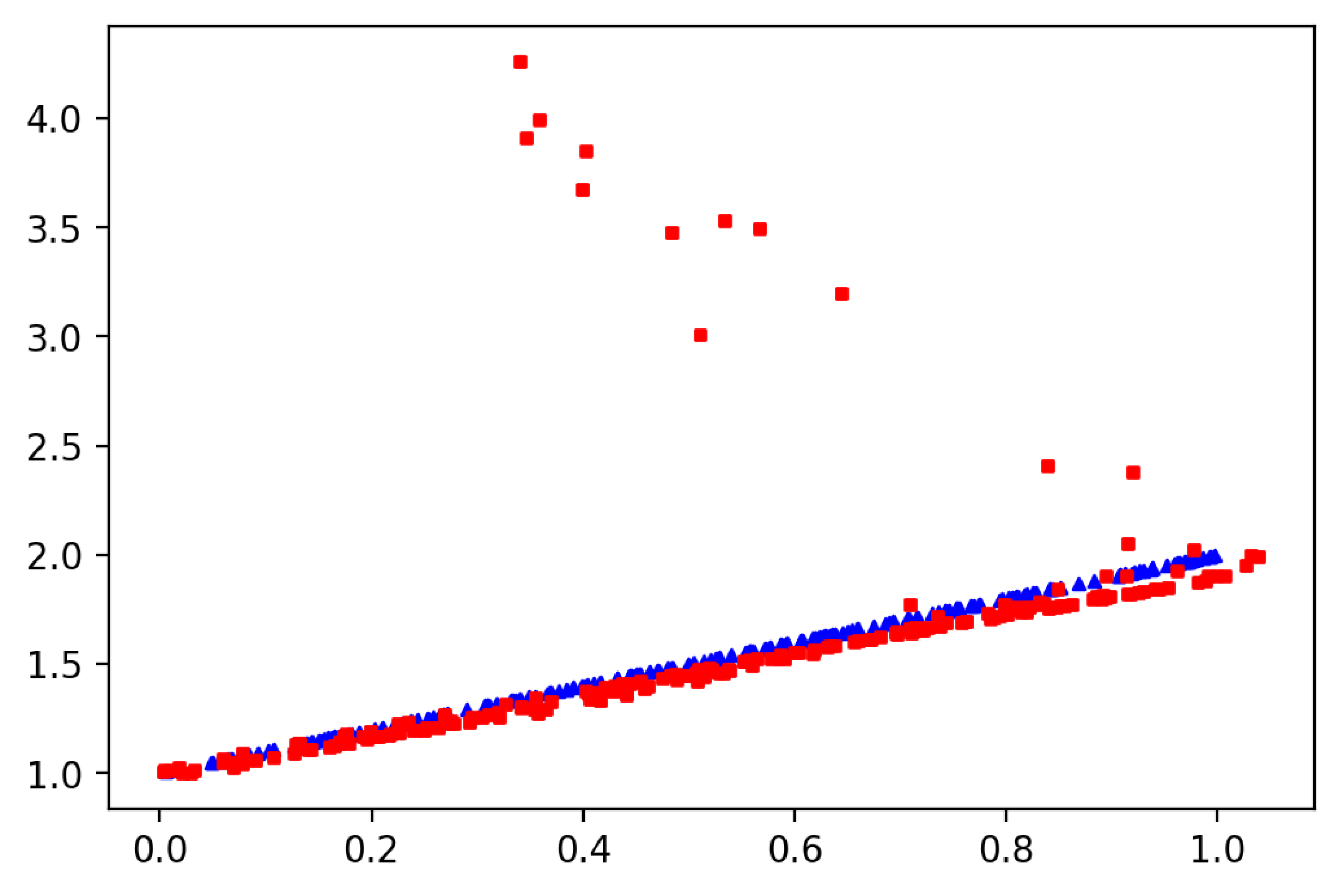}
	\caption*{ $ W_1= $  0.1932}
\end{subfigure}	
\begin{subfigure}[b]{0.24\textwidth}
	\centering
	\includegraphics[width=\textwidth]{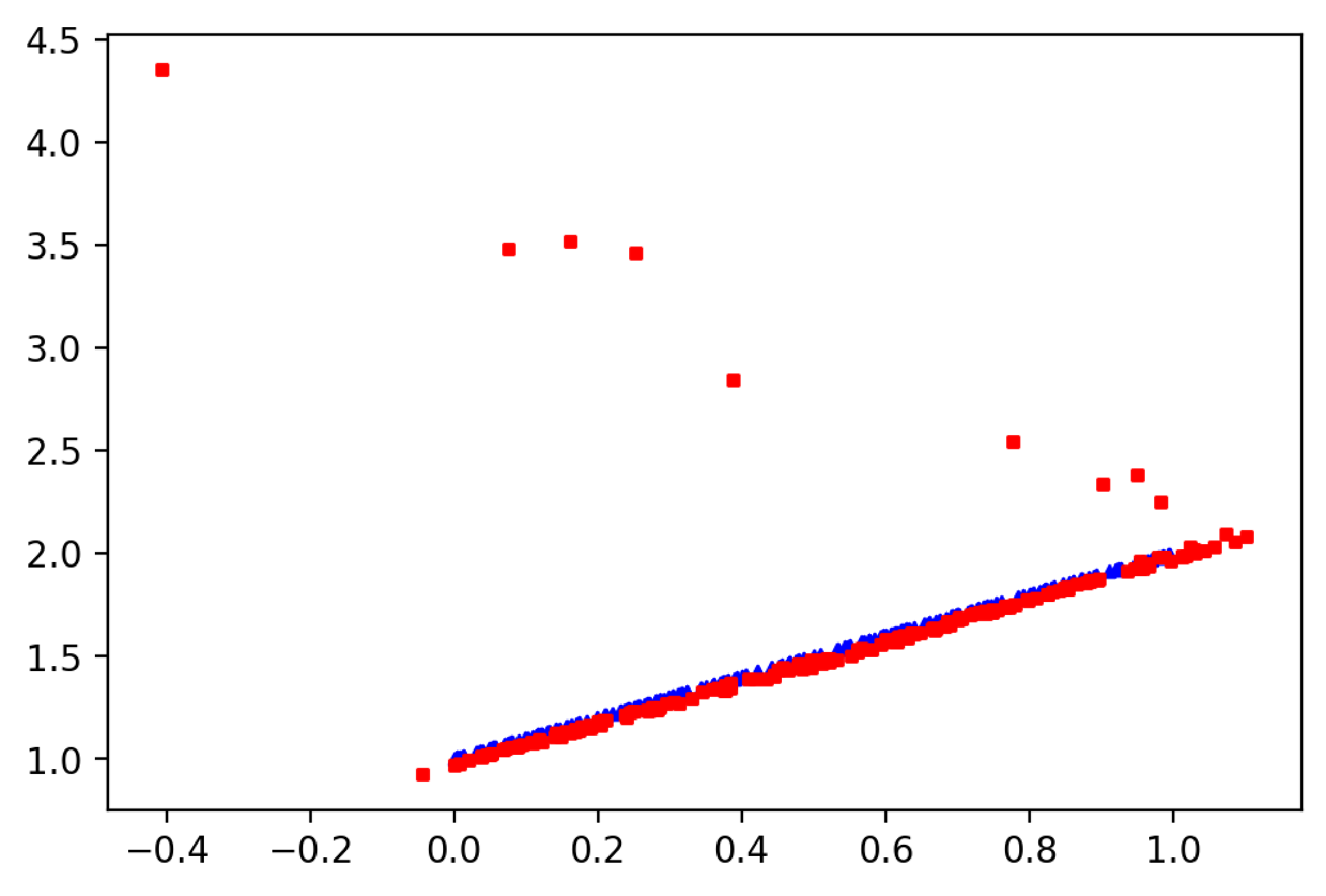}
	\caption*{ $ W_1= $  0.2077 }
\end{subfigure}
\begin{subfigure}[b]{0.24\textwidth}
	\centering
	\includegraphics[width=\textwidth]{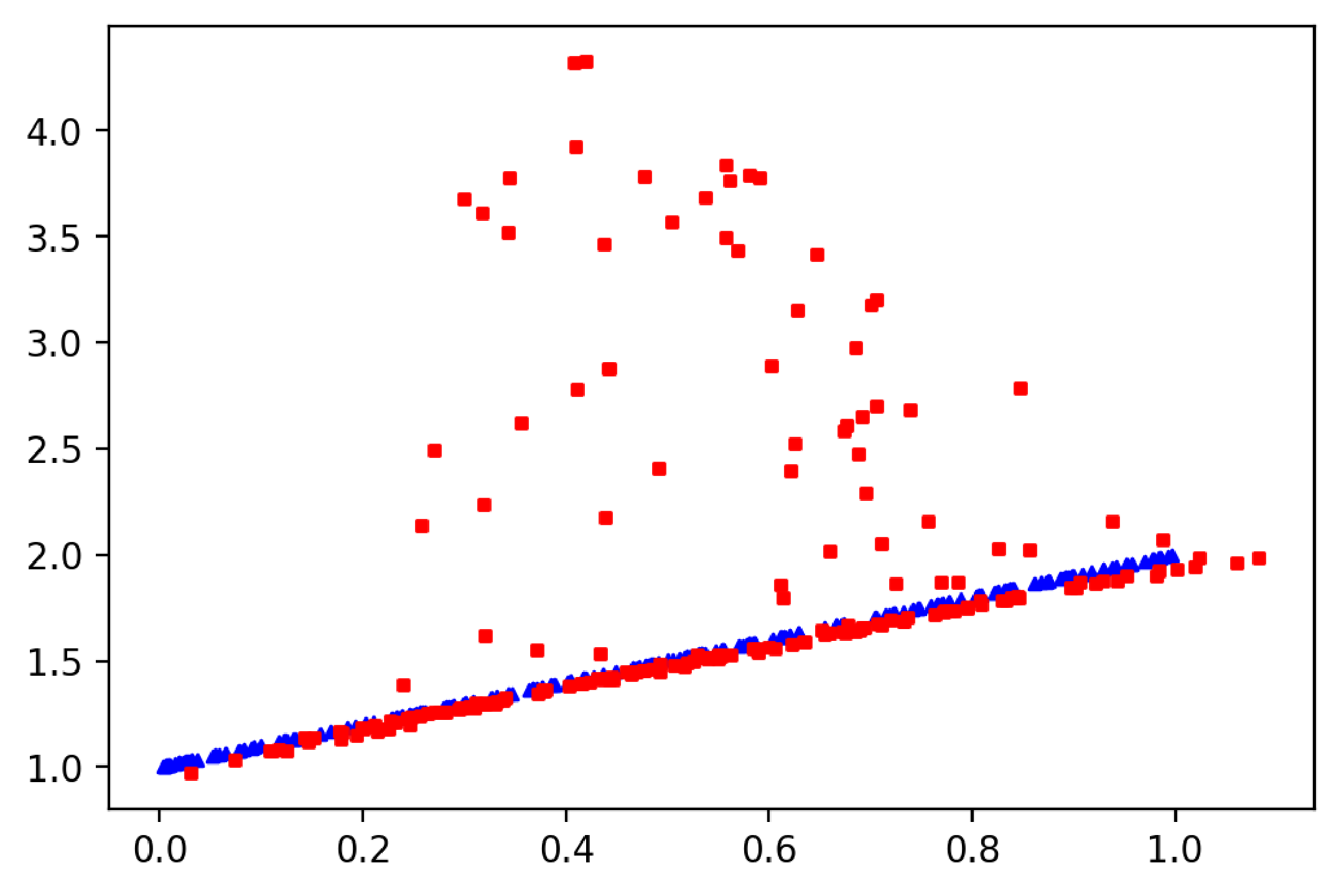}
	\caption*{ $ W_1= $ 0.3911 }
\end{subfigure}
\begin{subfigure}[b]{0.24\textwidth}
	\centering
	\includegraphics[width=\textwidth]{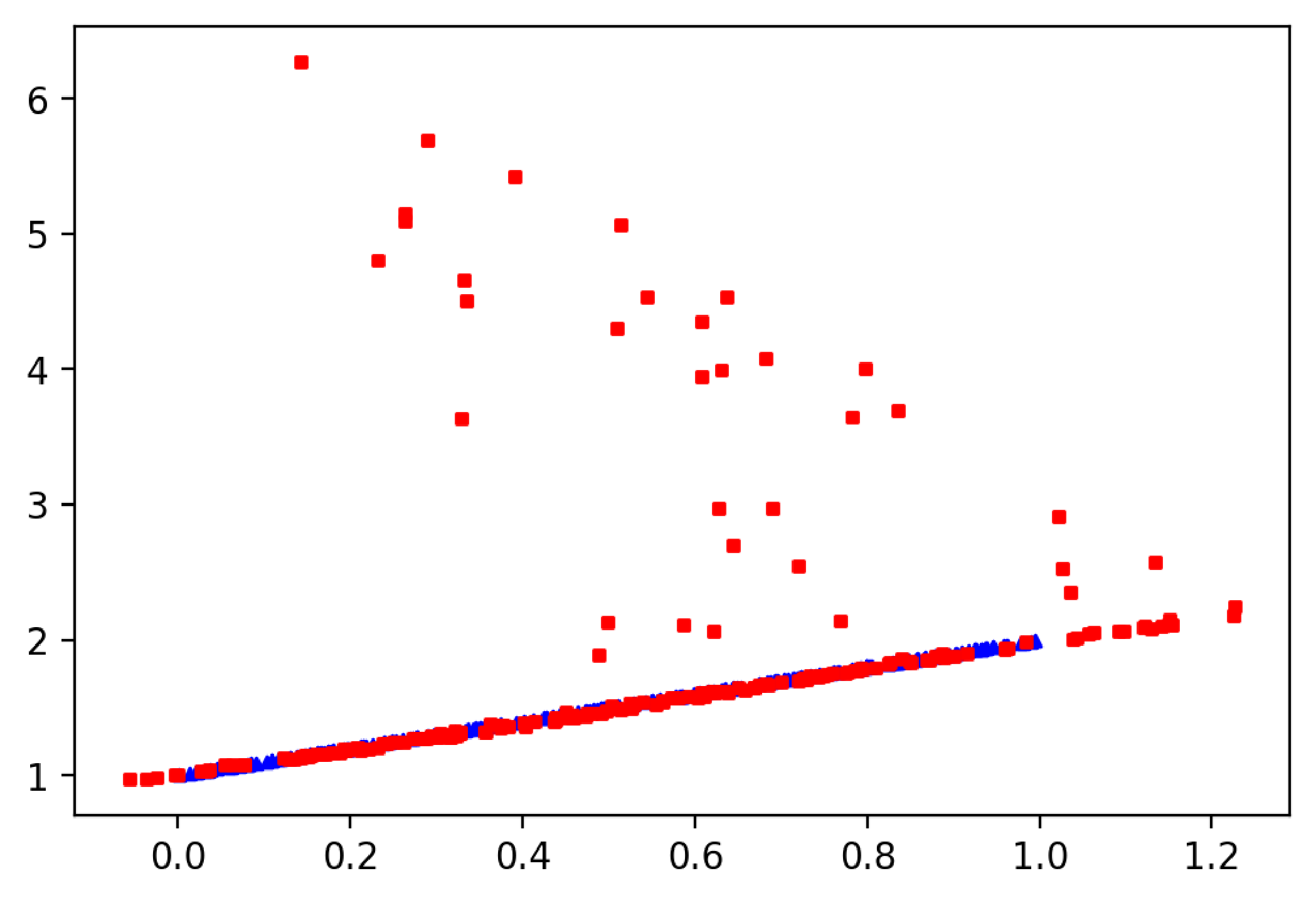}
	\caption*{ $ W_1= $ 0.3821}
\end{subfigure}
\\
\begin{subfigure}[b]{0.24\textwidth}
	\centering
	\includegraphics[width=\textwidth]{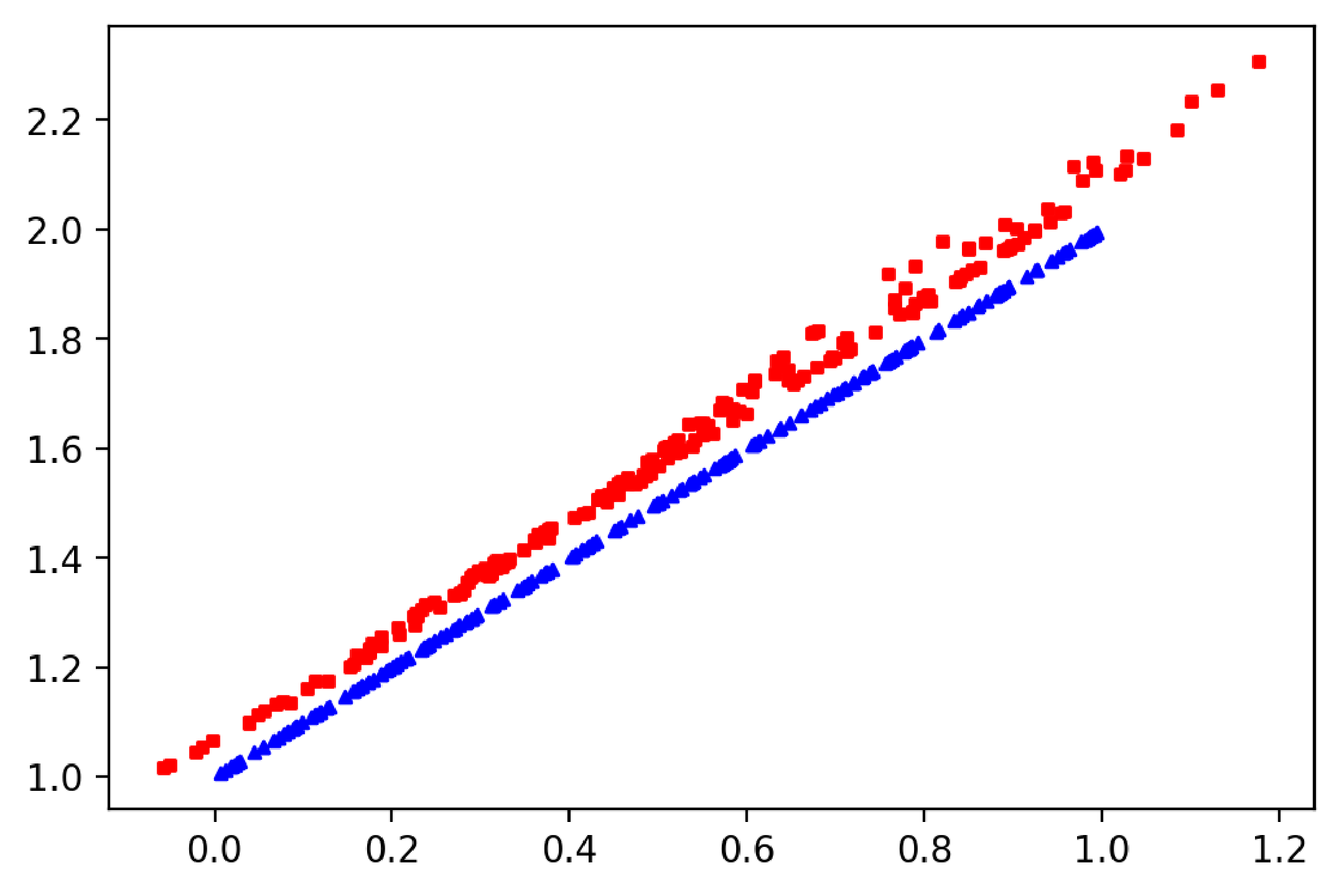}
	\caption*{ $ W_1= $  0.1164}
\end{subfigure}	
\begin{subfigure}[b]{0.24\textwidth}
	\centering
	\includegraphics[width=\textwidth]{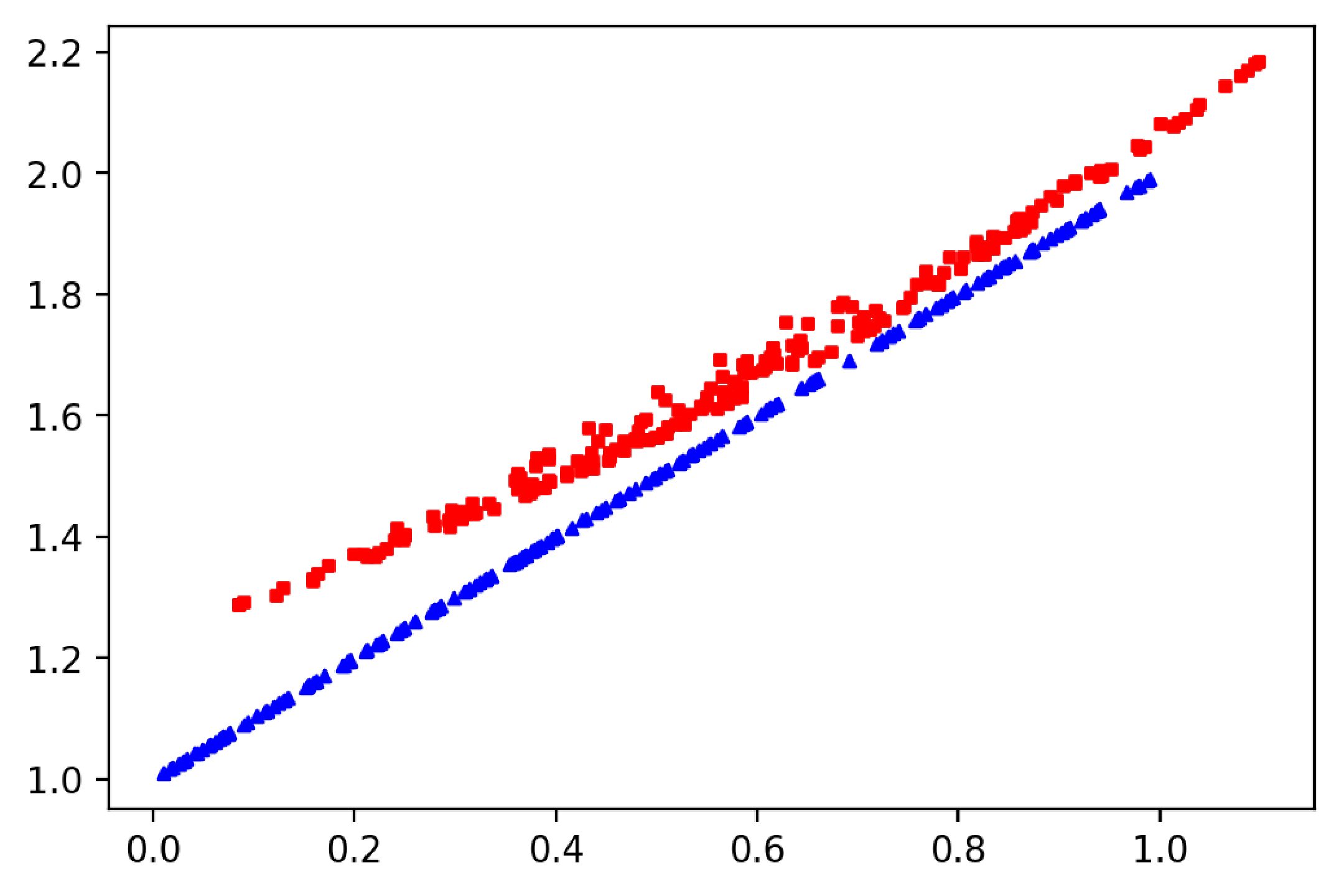}
	\caption*{ $ W_1= $  0.1985 }
\end{subfigure}
\begin{subfigure}[b]{0.24\textwidth}
	\centering
	\includegraphics[width=\textwidth]{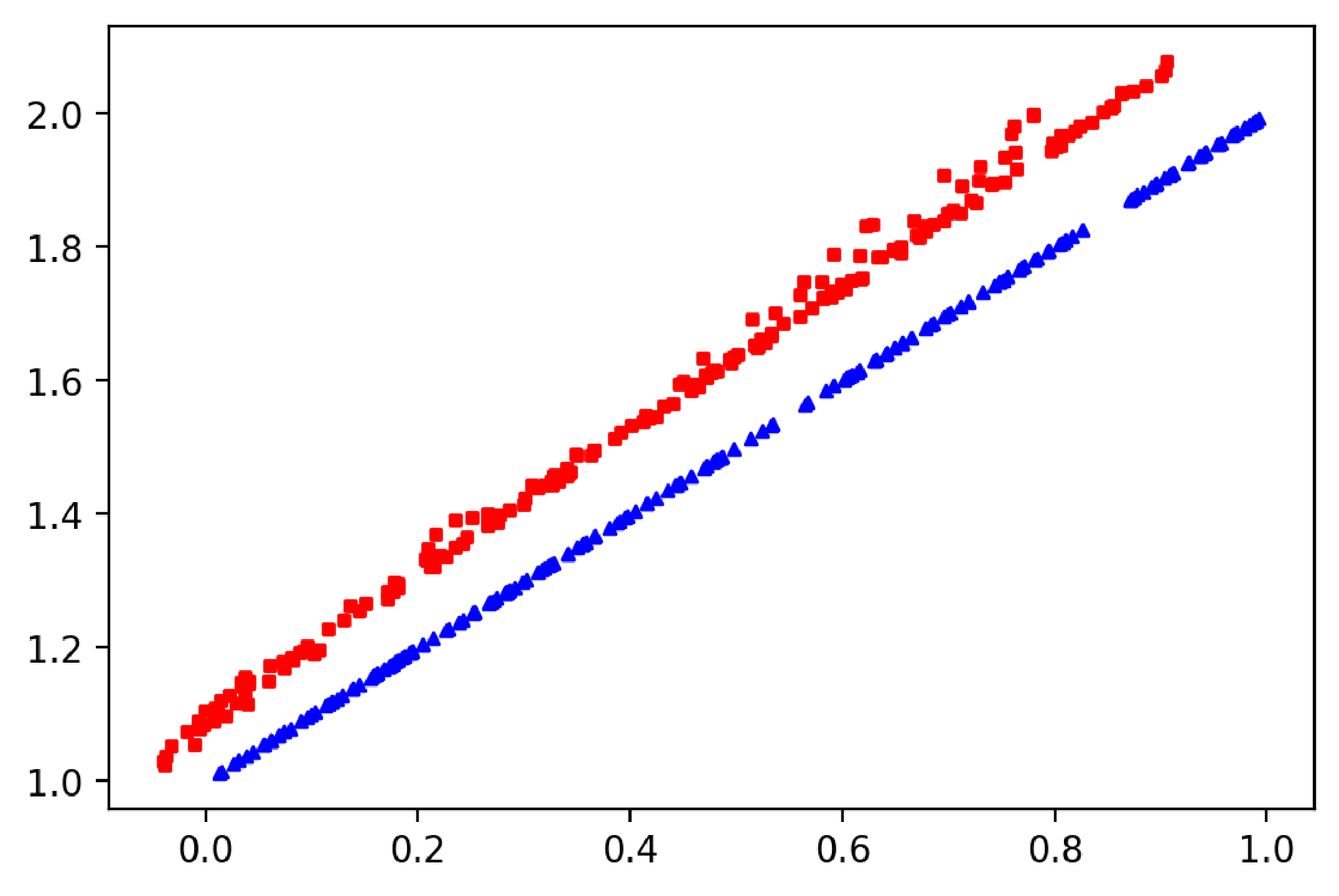}
	\caption*{ $ W_1= $ 0.1381 }
\end{subfigure}
\begin{subfigure}[b]{0.24\textwidth}
	\centering
	\includegraphics[width=\textwidth]{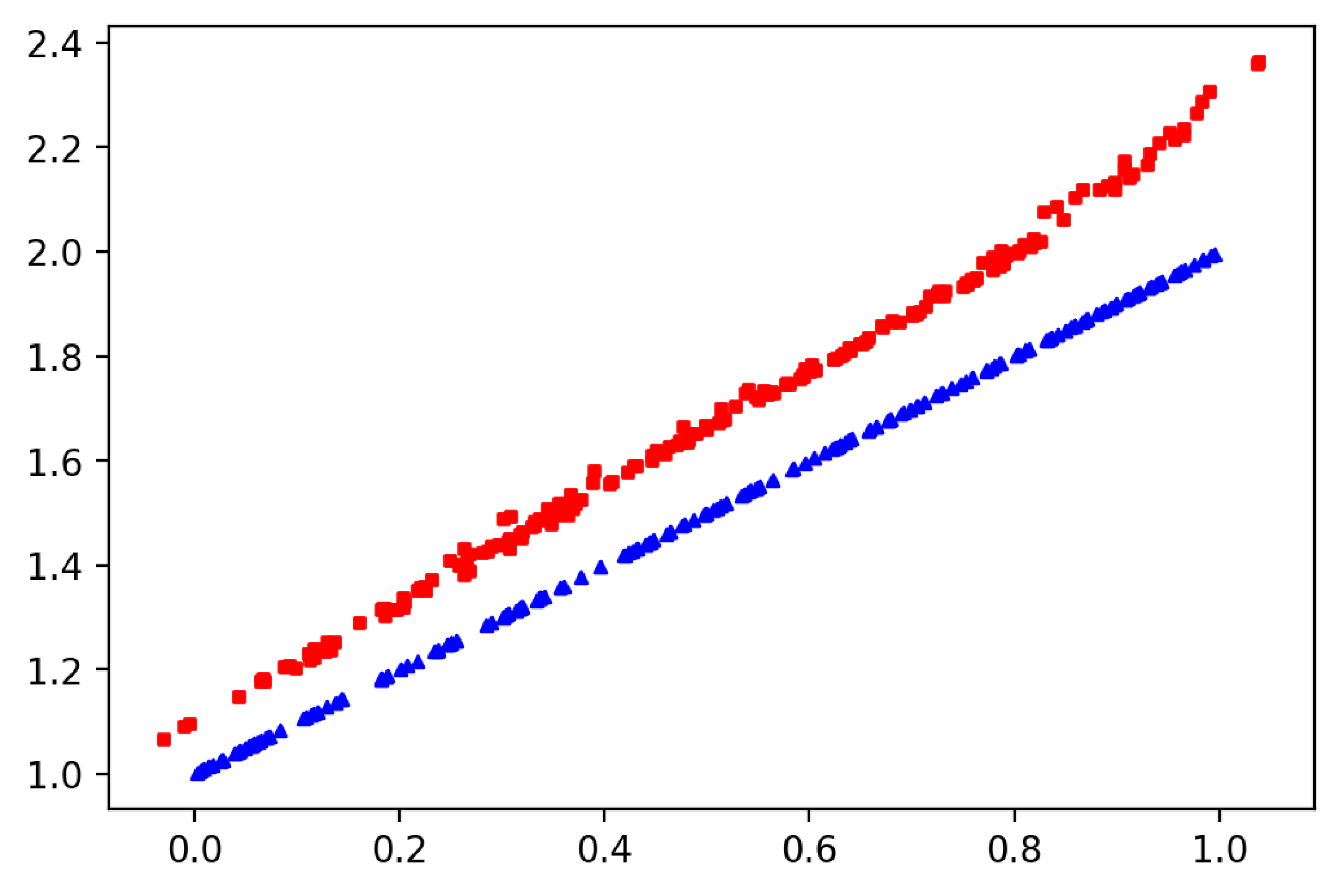}
	\caption*{ $ W_1= $ 0.2056}
\end{subfigure}
\\
\begin{subfigure}[b]{0.24\textwidth}
	\centering
	\includegraphics[width=\textwidth]{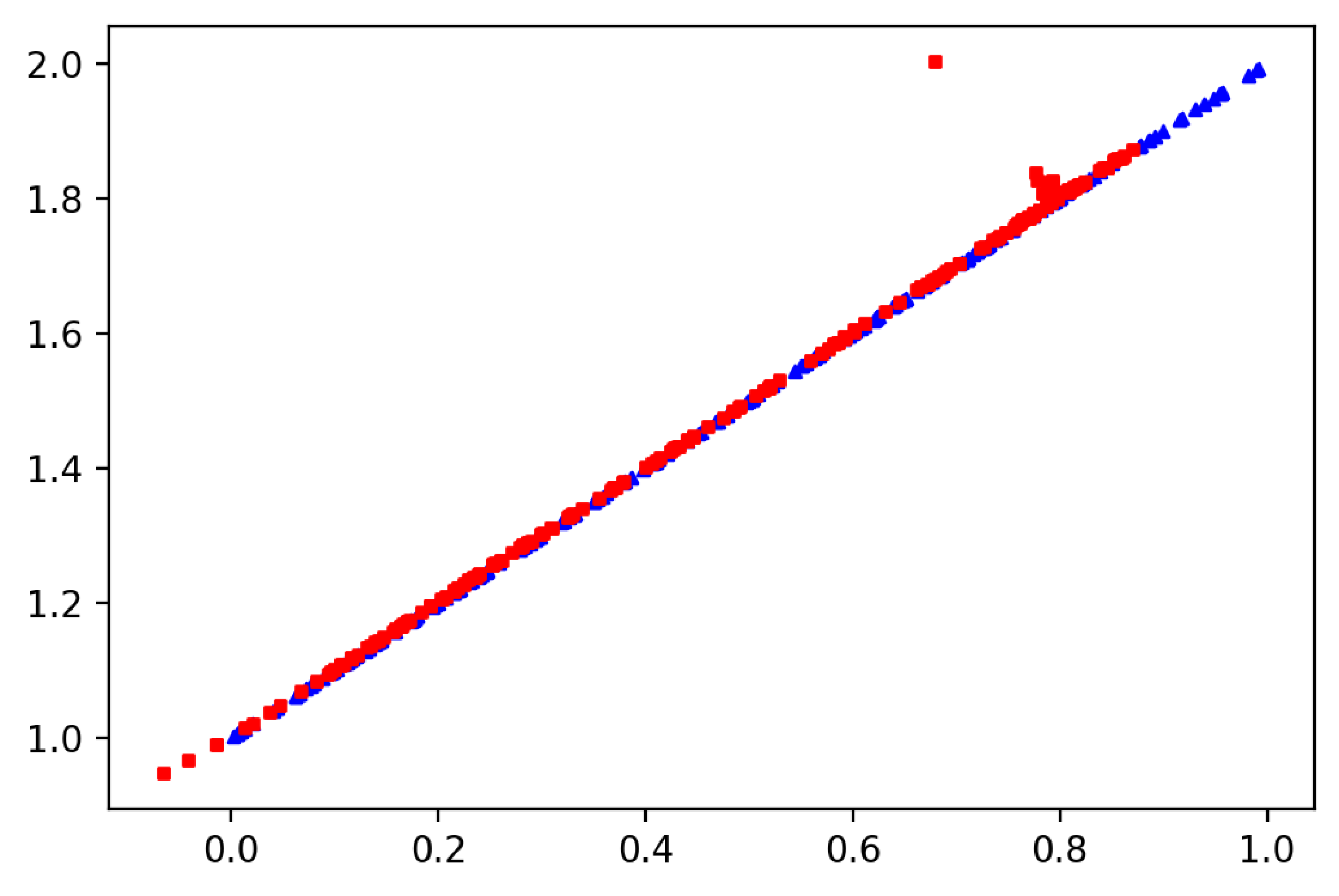}
	\caption*{ $ W_1= $  0.0607}
\end{subfigure}	
\begin{subfigure}[b]{0.24\textwidth}
	\centering
	\includegraphics[width=\textwidth]{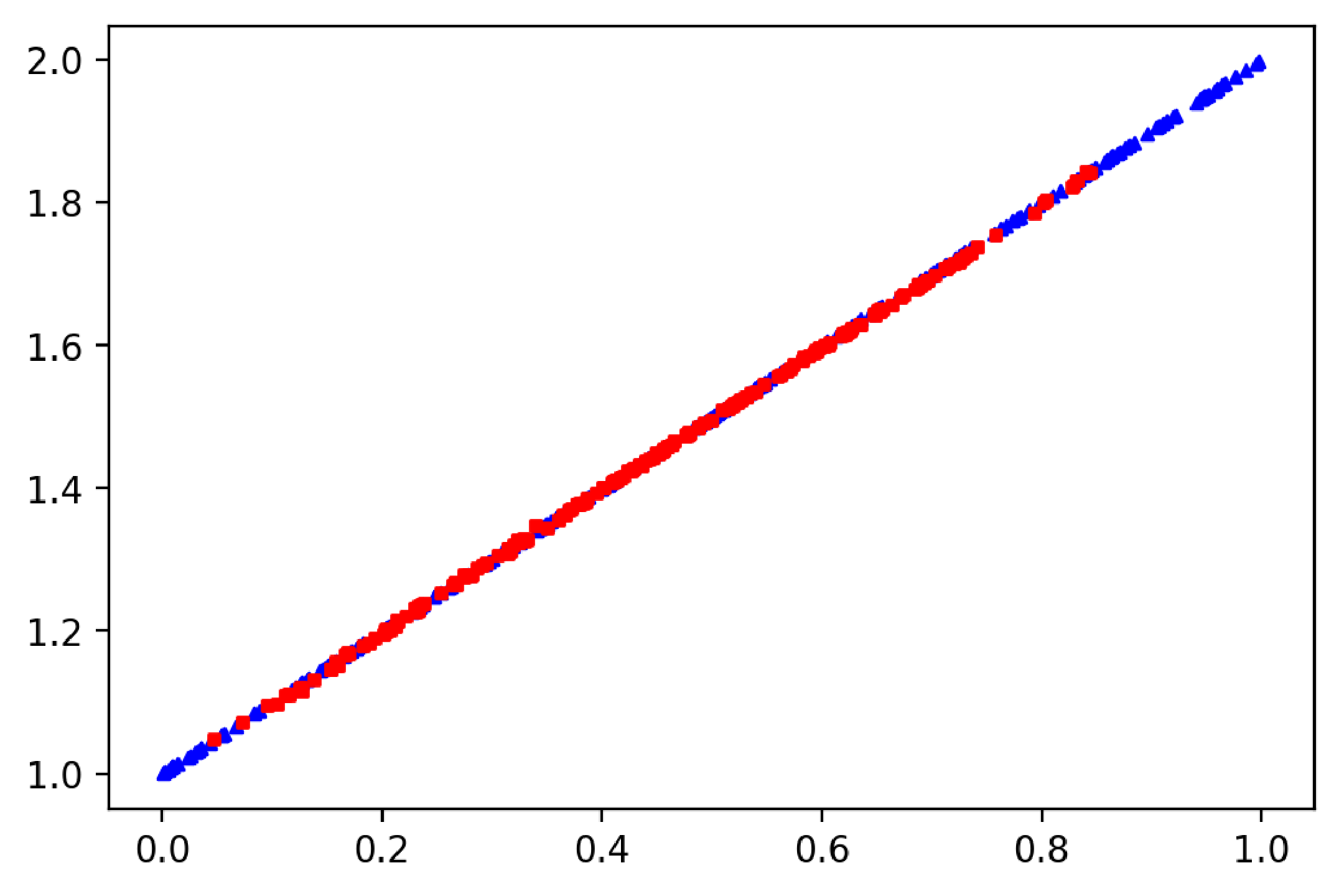}
	\caption*{ $ W_1= $  0.1446 }
\end{subfigure}
\begin{subfigure}[b]{0.24\textwidth}
	\centering
	\includegraphics[width=\textwidth]{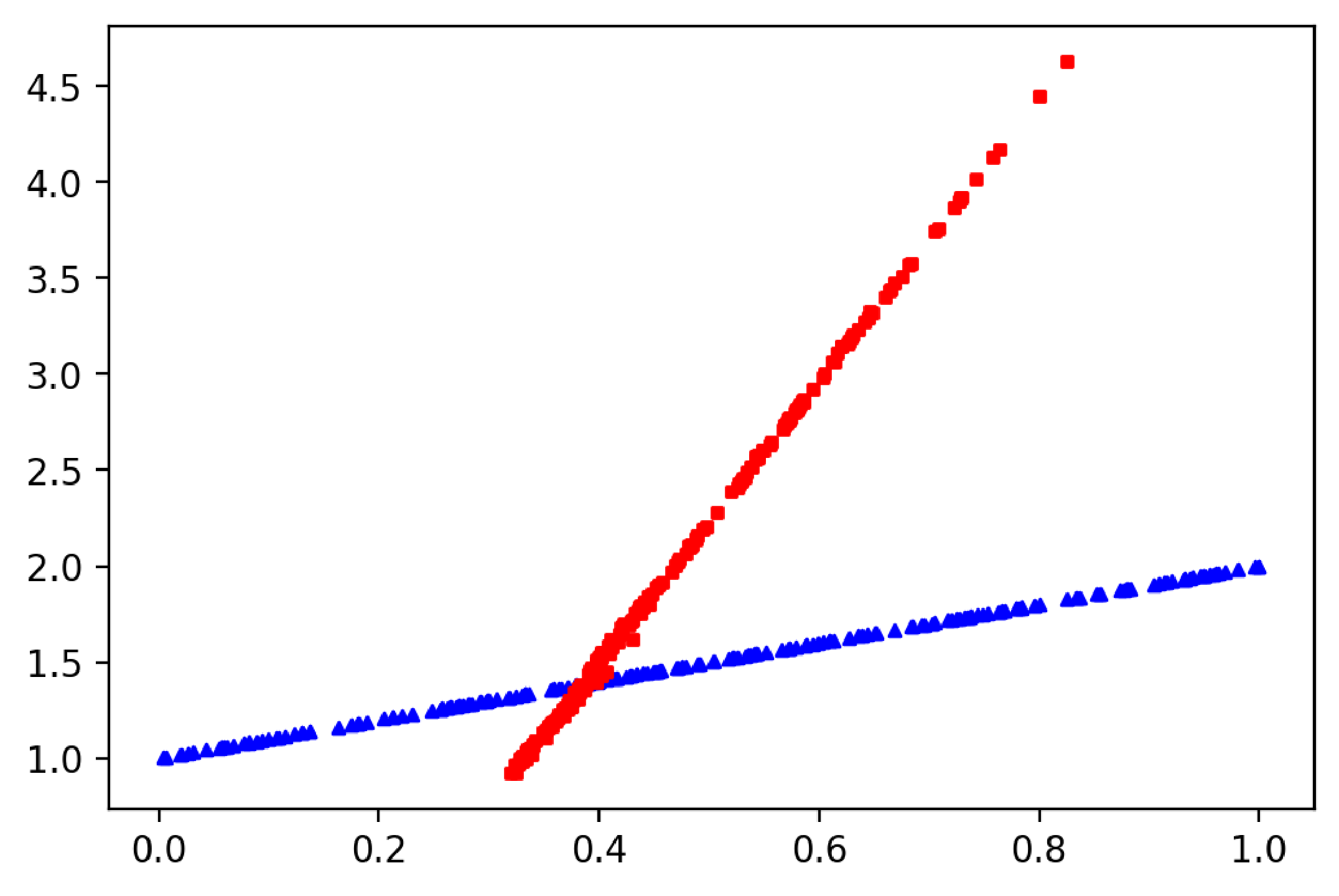}
	\caption*{$ W_1= $ 0.5962}
\end{subfigure}
\begin{subfigure}[b]{0.24\textwidth}
	\centering
	\includegraphics[width=\textwidth]{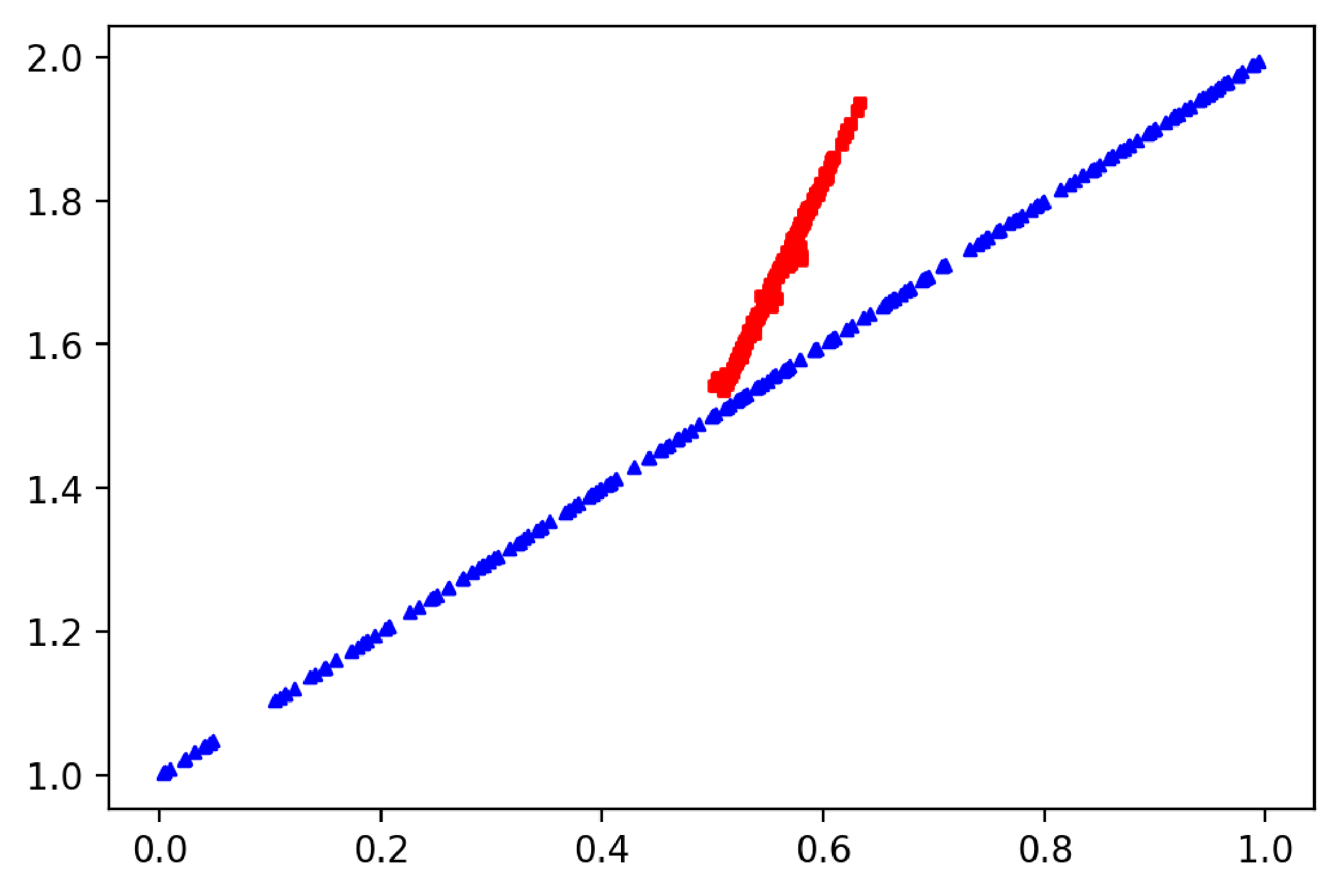}
	\caption*{ $ W_1= $ 0.2949}
\end{subfigure}
	\caption{Blue triangles: 200 data points sampled from the reference distribution (uncontaminated observations). Red squares: 200 data points generated from  WGAN (1st row), RWGAN-1 (2nd row), RWGAN-2  (3rd row), RWGAN-B(4th row), RWGAN-N with $\epsilon = 0.07$  (5th row), RWGAN-N with $\epsilon = 0.25$ (6th row) and  RWGAN-D  (7th row). Empirical Wasserstein distance $ p=1 $ between triangles and squares is provided.  (1st column: $\varepsilon = 0.1$,$\eta= 2$; 2nd column: $\varepsilon = 0.1$,$\eta= 3$; 3rd column: $\varepsilon = 0.2$,$\eta= 2$; 4th column: $\varepsilon = 0.2$,$\eta= 3$).} 

\label{syndata}
\end{figure}

\noindent \textit{Fashion-MNIST real data example.} 
We illustrate the performance of  RWGAN-1 and RWGAN-2 through comparing them with the WGAN and RWGAN-B in analysing the Fashion-MNIST real dataset, which contains 70000  grayscale images  of apparels, including T-shirt, jeans, pullover, skirt, coat, etc.
Each image is of $28\times 28$ pixels, and each pixel gets value from 0 to 255, indicating the darkness or lightness. We
generate outlying images by taking the negative effect of images already available in the dataset---a negative image is a total inversion, in which light areas appear dark and vice versa; in the Fashion-MNIST dataset, for each pixel of a negative image, it takes 255 minus the value of the pixel corresponding to the normal picture.
\textcolor{black}{By varying the number of negative images we define different levels of contamination: we consider 0, 3000, and 6000 outlying images.}  
For the sake of visualization, in Figure~\ref{normal pic} and~\ref{pic as outlier} we show normal and negative images, respectively. 
Our goal is to obtain a generator which can produce images of the same style as Figure \ref{normal pic}, even if it is trained on sample containing normal images and outliers. A more detailed explanation on how images generation works via \texttt{Pytorch} is available in  Appendix \ref{app.rwgan}. 

In Figure~\ref{fig:real_data}, we display images generated by the WGAN, RWGAN-1, RWGAN-2 and RWGAN-B under the different levels of contamination---we select $\lambda$ making use of the data driven procedure described in Section \ref{Sec:selection}. We see that all of the GANs  generate similar images when there are no outliers in the reference sample (1st row).
When there are 3000 outliers in the reference sample (2nd row), the GANs generate some negative images. However, by a simple eyeball search on the plots, we notice that WGAN produces many obvious negative images,  while RWGAN-B,  RWGAN-1 and RWGAN-2 generate less outlying images than WGAN. Remarkably, RWGAN-1 produces the smaller number of negative images among the considered methods. Similar conclusions holds for the higher contamination level with 6000 outlying images.

\begin{figure}[!http]
\centering
\begin{subfigure}[b]{0.3\textwidth}
\centering
\includegraphics[width=\textwidth]{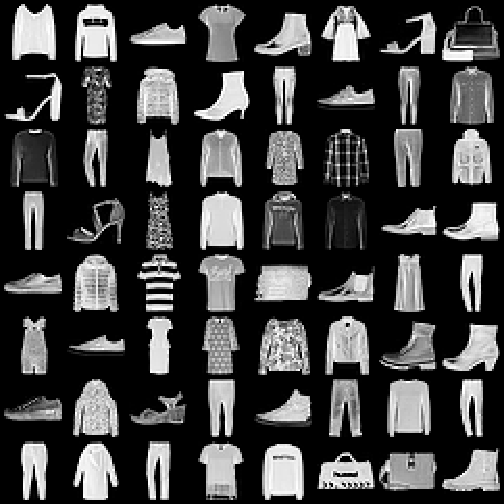}
\caption{ normal image }
\label{normal pic}
\end{subfigure}
\hspace{12mm}
\begin{subfigure}[b]{0.3\textwidth}
\centering
\includegraphics[width=\textwidth]{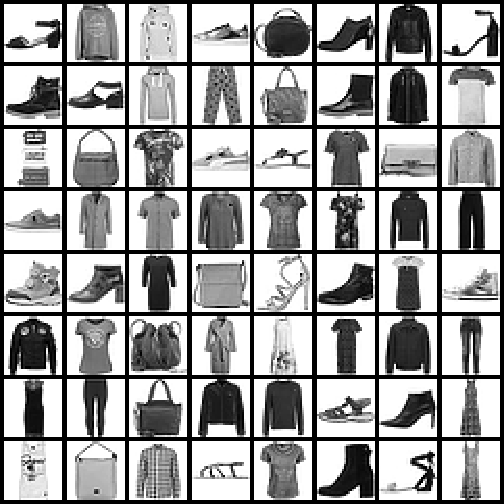}
\caption{outlying image}
\label{pic as outlier}
\end{subfigure}
\caption{Style of  clean reference sample images and  outlier images. Panel (a) contains pictures that are regarded as the right style. Panel (b)  contains  negative pictures (outliers).}

\end{figure}

To quantify the accuracy of these GANs, we train a convolutional neural network (CNN)  to classify normal images and negative images, using a training set of 60000 normal images and 60000 negative images. The resulting CNN  is 100\% accurate on a test set of size 20000. Then, we use this CNN to classify 1000 images generated by the four GANs. In Table~\ref{tab:resultofGAN}, we report the frequencies of normal images detected by the CNN.  The results are consistent with the eyeball  analysis of Figure~\ref{fig:real_data}: the  RWGAN-1 produces the smaller numbers of negative images at different contamination levels. 





\begin{figure}[http]
\centering
\begin{subfigure}[b]{0.2\textwidth}
\centering
\includegraphics[width=\textwidth]{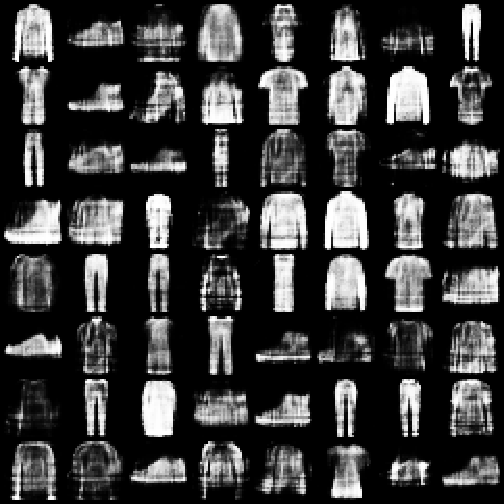}
\caption{}
\label{fig:wgan_0}
\end{subfigure}
\begin{subfigure}[b]{0.2\textwidth}
	\centering
	\includegraphics[width=\textwidth]{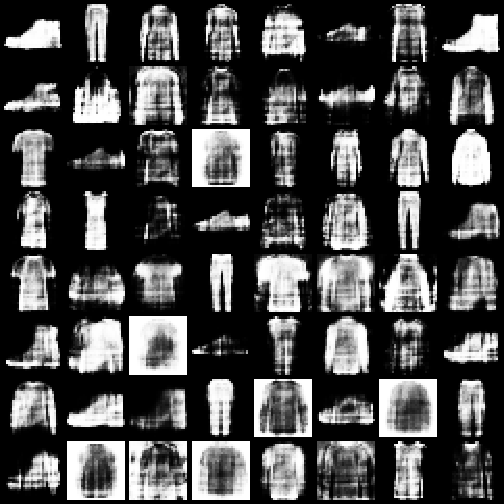}
	\caption{}
	\label{fig:wgan_5000}
\end{subfigure}
\begin{subfigure}[b]{0.2\textwidth}
	\centering
	\includegraphics[width=\textwidth]{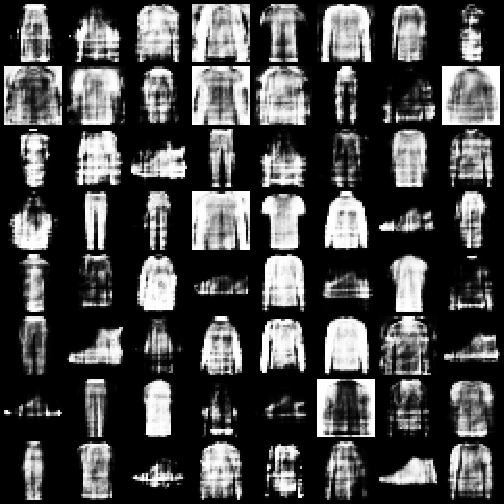}
	\caption{}
	\label{fig:wgan_10000}
\end{subfigure}
\\
\begin{subfigure}[b]{0.2\textwidth}
\centering
\includegraphics[width=\textwidth]{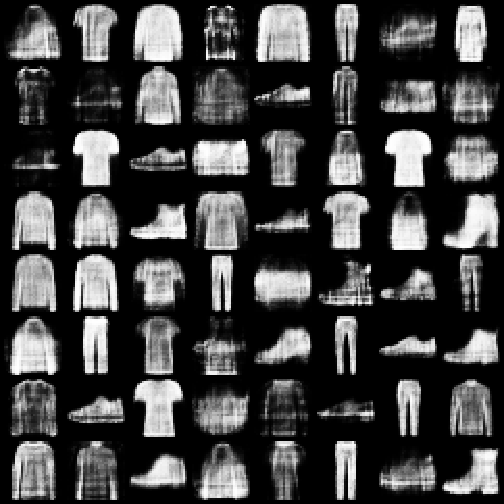}
\caption{}
\label{fig:rwgan1_0}
\end{subfigure}
\begin{subfigure}[b]{0.2\textwidth}
	\centering
	\includegraphics[width=\textwidth]{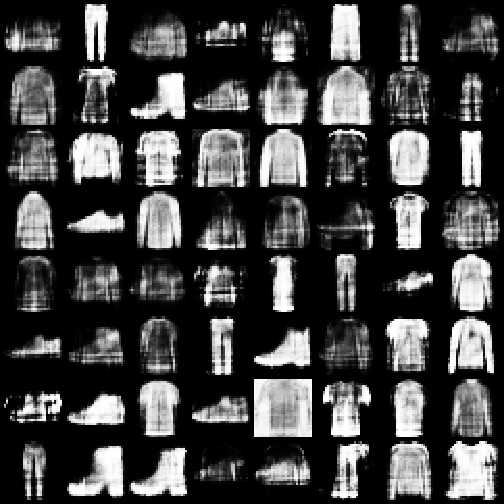}
	\caption{}
	\label{fig:rwgan1_5000}
\end{subfigure}
\begin{subfigure}[b]{0.2\textwidth}
	\centering
	\includegraphics[width=\textwidth]{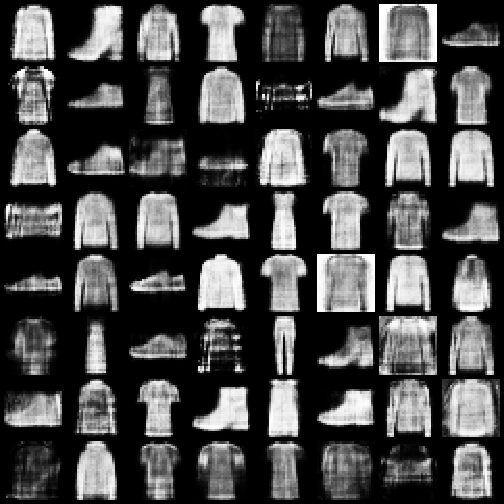}
	\caption{}
	\label{fig:rwgan1_10000}
\end{subfigure}
\\
\begin{subfigure}[b]{0.2\textwidth}
\centering
\includegraphics[width=\textwidth]{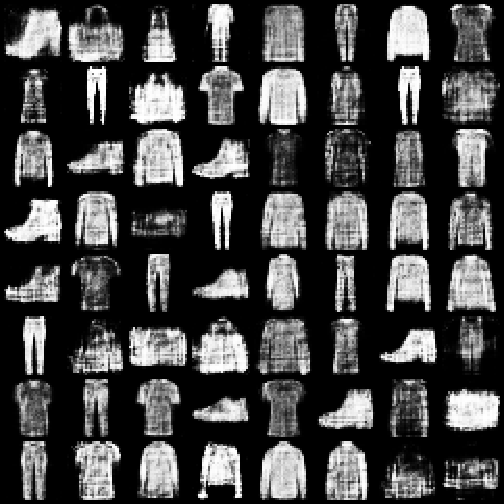}
\caption{}
\label{fig:rwgan2_0}
\end{subfigure}
\begin{subfigure}[b]{0.2\textwidth}
	\centering
	\includegraphics[width=\textwidth]{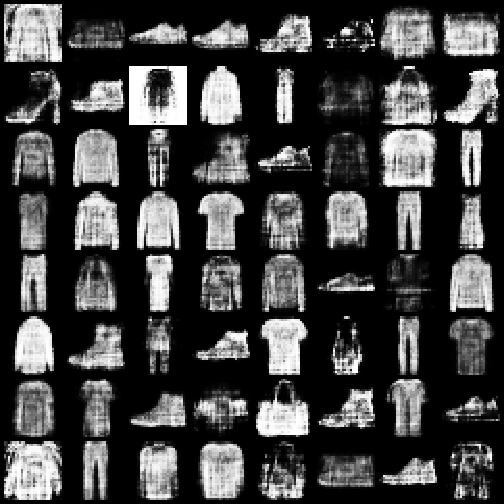}
	\caption{}
	\label{fig:rwgan2_5000}
\end{subfigure}
\begin{subfigure}[b]{0.2\textwidth}
	\centering
	\includegraphics[width=\textwidth]{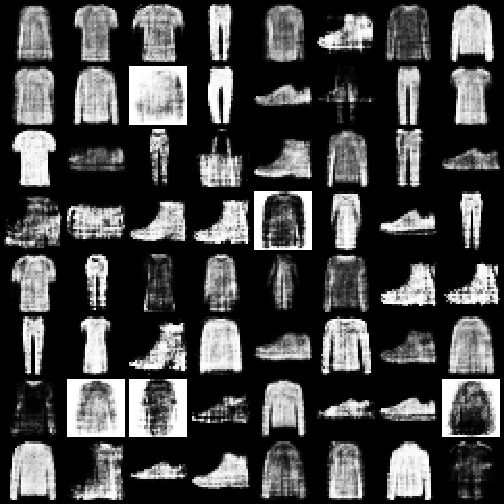}
	\caption{}
	\label{fig:rwgan2_10000}
\end{subfigure}
\\
\begin{subfigure}[b]{0.2\textwidth}
\centering
\includegraphics[width=\textwidth]{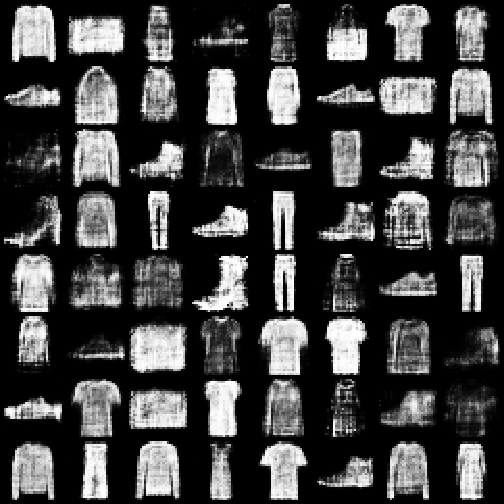}
\caption{}
\label{fig:rwganb_0}
\end{subfigure} 
\begin{subfigure}[b]{0.2\textwidth}
\centering
\includegraphics[width=\textwidth]{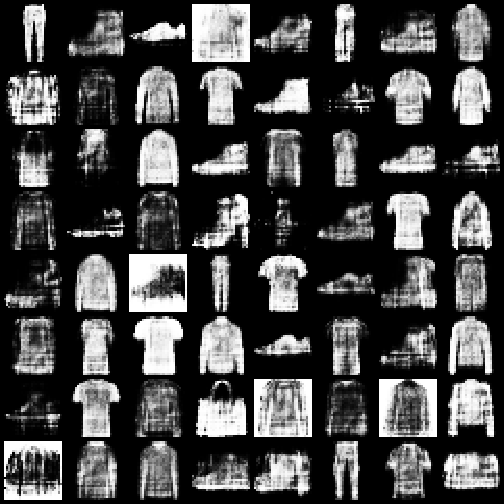}
\caption{}
\label{fig:rwganb_5000}
\end{subfigure} 
\begin{subfigure}[b]{0.2\textwidth}
\centering
\includegraphics[width=\textwidth]{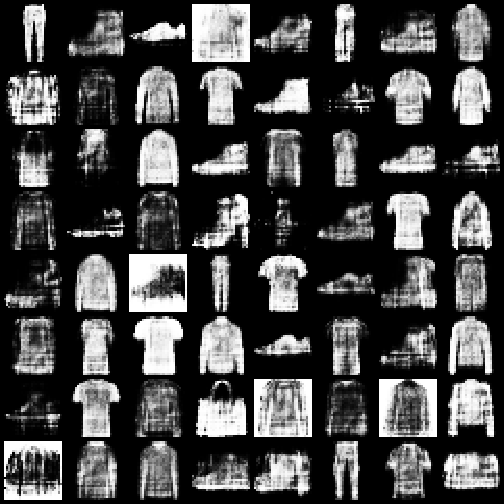}
\caption{}
\label{fig:rwganb_10000}
\end{subfigure}
\\
\caption{64 images generated from  WGAN (1st row), RWGAN-1 (2nd row), RWGAN-2  (3rd row), RWGAN-B  (4th row). 
(1st column: no outliers; 2nd column: 3000 outliers; 3rd column: 6000 outliers).}
\label{fig:real_data}
\end{figure}

\begin{table}[!htbp]
\centering
\begin{tabular}{|c|c|c|c|c|}
\hline
& \multicolumn{1}{c|}{WGAN} & \multicolumn{1}{l|}{RWGAN-1} & \multicolumn{1}{c|}{RWGAN-2} & \multicolumn{1}{l|}{RWGAN-B} \\
\hline
0 outliers & 1.000     & 1.000     & 1.000     & 1.000 \\
\hline
3000 outliers & 0.941 & 0.982 & 0.953 & 0.945\\
\hline
6000 outliers & 0.890  & 0.975 & 0.913 & 0.897 \\
\hline
\end{tabular}
\caption{Frequency of normal images output for each GAN model, under different contamination levels (the results are based on the classification by a trained CNN model).}
\label{tab:resultofGAN}%
\end{table}%

\subsection{Domain adaptation} \label{Sec.DA}

\textbf{Methodology.} We consider an application of ROBOT to  domain adaptation (DA), which is a popular machine learning problem 
and it deals with the inconsistent probability distributions of training samples and test samples. It has a wide range of applications in natural language processing (\cite{ramponi2020neural}), text analysis (\cite{daume2009frustratingly}) and medical data analysis (\cite{hu2020challenges}). 

The DA problem can be stated as follows. Let $$\{({X}_i^{\rm s}, Y_i^{\rm s}); 1\leq i \leq N_{\rm s}\}$$ denote an  i.i.d. sequence of {\it source domain (training sample) }  and $$\{({X}_j^{\rm t}, Y_j^{\rm t}); 1\leq j \leq N_{\rm t}\}$$ denote an  i.i.d. sequence of {\it target domain (test sample)}, which  have joint distributions ${\rm P_s}$ and ${\rm P_t}$, respectively. Interestingly, ${\rm P_s}$ and ${\rm P_t}$ may be two different distributions: discrepancies (also known as drift) in these data distributions depend on the type of application considered---e.g. in image analysis, domain drifts can be due to changing lighting conditions, recording device failures, or to changes in the backgrounds.

In the context of DA, we are interested in estimating the labels $Y_j^{\rm t}, j = 1, \ldots, N_{\rm t},$ from observations $$\{({X}_i^{\rm s}, Y_i^{\rm s}); 1\leq i \leq N_{\rm s}\}$$ and $\{{X}_j^{\rm t}; 1\leq j \leq N_{\rm t}\}$. 
Therefore,  the goal is to develop a learning method which is able to transfer knowledge from the source domain to the target domain, taking into account that the training and test samples may have different distributions. 

To achieve this inference goal, \cite{courty2017joint} assume that there is a  transformation $\mathcal{T}$ between  ${\rm P}_{\rm s}$ and ${\rm P}_{\rm t}$, and estimate the map  $\mathcal{T}$ via OT.
To elaborate further, let us consider that,
in practice, the population distribution of the source sample is  unknown and it is usually estimated by its empirical counterpart  
$$\widehat{\rm P}_{\rm s}^{(N_{\rm s})} = {N_{\rm s}}^{-1} \sum_{i=1}^{N_{\rm s}} \delta_{{X}_{i}^{\rm s}, Y_{i}^{\rm s}}.$$
As far as the target domain is concerned, suppose we can obtain a function $f$ to predict the unobserved response variable $Y^{\rm t}$. Thus, the empirical distribution of the target sample is  $$ \widehat{{\rm P}}_{{\rm t}; f}^{(N_{\rm t})}= N_{\rm t}^{-1} \sum_{j=1}^{N_{\rm t}} \delta_{{{X}}_{j}^{\rm t}, f({{X}}_{j}^{\rm t})}.$$ The central idea of Curtey et al. (2014) relies on transforming the data  to make the source and target distributions similar (namely, close in some distance), and use the label information available in the source domain to learn a classifier in the transformed domain, which in turn can be exploited to get the labels in the target domain. 

More formally,  the DA problem is written as an optimization problem which aims at finding the function $f$ that minimizes the  transport cost between the estimated source distribution $\widehat{\rm P}_{\rm s}^{(N_{\rm s})}$ and the estimated joint target distribution $\widehat{{\rm P}}_{{\rm t}; f}^{(N_{\rm t})}$.  
To set up this problem, \cite{courty2014domain}
propose to use  the cost function 
$$
{D}\left({x}_{1}, y_{1} ; {x}_{2}, y_{2}\right)=\alpha {d}\left({x}_{1}, {x}_{2}\right)+{L}\left(y_{1}, y_{2}\right),
$$ 
which defines a joint cost combining both the distances between the features $d\left({x}_{1}, {x}_{2}\right)$ and
  a  measure ${L}$ of the discrepancy between $ y_1 $ and $ y_2  $, where (as recommended in the Courty et al. code) 
  $\alpha=0.3$.
Then,
the optimization problem formulated in terms of 1-Wasserstein distance is 
\begin{equation}\label{equ10}
	\min _{f} W_1\left(\hat{\mathrm{P}}_{\rm s}^{(N_{\rm s})}, \hat{\mathrm{P}}_{{\rm t};f}^{(N_{\rm t})}\right) \equiv \min _{f,\mathbf{\Pi}} \sum_{i j}  {D}\left({X}_{i}^{\rm s}, Y_{i}^{\rm s} ; {X}_{j}^{\rm t}, f({X}_{j}^{\rm t})\right) \mathbf{\Pi}_{i j}^{(N_{\rm s},N_{\rm t})},
\end{equation}
which,  in the case of  squared loss function $L$ (see \cite{courty2014domain}), becomes
$$	
\min _{f}  N_{\rm t}^{-1} \sum_{j=1}^{N_{\rm t}} \left\|\hat{Y}_j-f(X_j^{\rm t})\right\|^2, \ \text{where} \ 
\hat{Y}_j =N_{\rm t} \sum_{i=1}^{N_{\rm s}} \mathbf{\Pi}_{i j}^{(N_{\rm s},N_{\rm t})} Y_i^{\rm s};
$$
Courty et al. 2014 propose to 
solve this problem
using (a regularized version of)  OT. Since OT can be sensitive to outliers, in the same spirit of \cite{courty2014domain}  but with the additional aim of reducing the impact of outlying values, we propose to use ROBOT instead of OT
in the DA problem.  The basic intuition is that ROBOT yields a matrix $ \mathbf{\Pi}^{(N_{\rm s},N_{\rm t})}$ which is able (by construction) to identify and downweight the impact of anomalous values, yielding a robust $\hat{Y}_j $, for every $j$. We provide the key mathematical and implementation details in Algorithm~\ref{alg:domain adptation}.

\begin{algorithm}[H]
	\SetAlgoLined
	\KwData{source sample $\{({X}_i^{\rm s},Y_i^{\rm s}),1\leq i \leq N_{\rm s} \}$,  target features $\{{X}_j^{\rm t}, 1\leq j \leq N_{\rm t}\}$,  robust regularization parameter $\lambda$, parametric model $ f_{\theta} $ (index by a parameter $\theta$)}
	\KwResult{Model $f_{\theta}$ }
	estimate $ \theta $ using the source sample\;
	\While{${\theta} $ does not  convergence}{
		let $ C^{(\lambda, N_{\rm s}, N_{\rm t})}_{ij} = \min\left\lbrace {D}\left({X}_{i}^{\rm s}, Y_{i}^{\rm s} ; {X}_{j}^{\rm t}, f_{\theta} \left({X}_{j}^{\rm t}\right)\right) ,2\lambda \right\rbrace $\;
		
		based on the cost matrix $(C^{(\lambda, N_{\rm s}, N_{\rm t})}_{ij})_{1\leq i\leq N_{\rm s}, 1\leq j\leq N_{\rm t}}$, 
		compute the  optimal transport plan matrix $ \mathbf{\Pi}^{(N_{\rm s},N_{\rm t})} $, and collect the set of all indices
		$$ \mathcal{I}=\{(i,j):  {D}\left({X}_{i}^{\rm s}, Y_{i}^{\rm s} ; {X}_{j}^{\rm t}, f_{\theta}\left({X}_{j}^{\rm t}\right)\right) \geq 2 \lambda, 1 \leq i \leq N_{\rm s}, 1 \leq j \leq N_{\rm t}\}$$\;
		
		set $s(i)=-\sum_{j=1}^{N_{\rm t}} \mathbf{\Pi}^{(N_{\rm s},N_{\rm t})}_{ij} \mathbbm{1}_{(i, j) \in \mathcal{I}}$, and find $\mathcal{H}$, the set of all the indices where $s(i)+1/N_{\rm s} = 0$\;
		
		remove $\ell$-th row  of $ \mathbf{\Pi}^{(N_{\rm s},N_{\rm t})} $ for all $\ell \in \mathcal{H}$ and normalize the matrix over columns to form a new matrix $ \mathbf{\Pi}^{(N^{\prime}_{\rm s},N_{\rm t})} $\;
		
		let ${L}_{\theta} =  N_{\rm t}^{-1} \sum_{j=1}^{N_{\rm t}} \left\|\hat{Y}_j-f_{\theta}(X_j^{\rm t})\right\|^2$, 
		where $\hat{Y}_j =N_{\rm t} \sum_{i=1}^{N^{\prime}_{\rm s}} \mathbf{\Pi}_{i j}^{(N^{\prime}_{\rm s},N_{\rm t})} Y_i^{\rm s}$\;	
		
		update $\theta$, which minimizes  $ {L}_{\theta} $\;
	}
		\caption{ROBUST DOMAIN ADAPTATION}
		\label{alg:domain adptation}
\end{algorithm}

\textbf{Numerical exercise.} The next Monte Carlo experiment illustrates the use and the performance of Algorithm  \ref{alg:domain adptation} on synthetic data.
We consider  $N_{\rm s}$ source data points (training sample) generated from 
\begin{equation}\label{source}
	\begin{array}{cc}
		& X_{i}^{(N_{\rm s})} \sim \mathcal{N}(-2,1),i=1,2,\ldots,[N_{\rm s}/2], \\
		& X_{i}^{(N_{\rm s})} \sim \mathcal{N}(2,1),i=[N_{\rm s}/2]+1,[N_{\rm s}/2]+2,\ldots,N_{\rm s},\\
		& Y_i^{(N_{\rm s})} = \sin (X_{i}^{(N_{\rm s})}/2) + Z_{i}^{(N_{\rm s})},i=1,2,\ldots,[9N_{\rm s}/10], \\
		& Y_i^{(N_{\rm s})} = \sin (X_{i}^{(N_{\rm s})}/2)+2 + Z_{i}^{(N_{\rm s})},i=[9n/10]+1,[9N_{\rm s}/10]+2,\ldots,N_{\rm s}, \\
		& Z_{i}^{(N_{\rm s})} \sim \mathcal{N}(0,0.1^2)\\
	\end{array}
\end{equation}
and the $ N_{\rm t} $ target points (test sample) are generated from 
\begin{equation}\label{target}
	\begin{array}{cc}
		& X_{j}^{(N_{\rm t})} \sim \mathcal{N}(-1,1),j=1,2,\ldots,[N_{\rm t}/2], \\
		& X_{j}^{(N_{\rm t})} \sim \mathcal{N}(2,1),j=[N_{\rm t}/2]+1,[N_{\rm t}/2]+2,\ldots,N_{\rm t},\\
		& Y_j^{(N_{\rm t})} = \sin ((X_{i}^{(N_{\rm t})}-2)/2) + Z_{i}^{(N_{\rm t})},j=1,2,\ldots,N_{\rm t},\\
		& Z_{j}^{(N_{\rm t})} \sim \mathcal{N}(0,0.1^2).\\
	\end{array}
\end{equation}
We notice that the source and target domain have a similar distribution but there is a shift, called \textit{drift} in the DA literature, between them.
In the training sample, outliers in the response variable are due to the drift added to the $Y_i^{(N_{\rm s})}$ for $i=[9n/10]+1,[9N_{\rm s}/10]+2,\ldots,N_{\rm s}$. Moreover, mimicking the problem of leverage points in linear regression (see \cite[Ch. 6]{hampel1986robust}), our simulation design is such that  half of the $X_{i}^{(N_{\rm s})}$s are generated from  a $\mathcal{N}(-2,1)$, whilst the other half is obtained from a shifted/drifted Gaussian $\mathcal{N}(2,1)$; the feature space in the test sample contains a similar pattern but with a different drift. 
\color{black}

 In our simulation exercise, we set $ N_{\rm s} =200 $ and $ N_{\rm t}=200 $ and we compare three methods for estimating $ f $: the first one is the Kernel Ridge Regression (KRR); the second one is the method proposed by \cite{courty2017joint}, which   is a combination of OT and  KRR; the third one is  our method, which is based on a modification of the second method, as obtained using ROBOT instead of OT. {To the best of our knowledge, this application of ROBOT to DA  is new to the literature}.



In Figure~\ref{fig:DA} we show the results. {The outliers in the training sample are clearly visible as a cluster of points which are far a part from the bulk of the data. The plot illustrates that the curve obtained via KRR fits reasonably well only the majority of the data in the source domain, but it looks 
inadequate for the target domain.}
This is due to the drift  between the target and source distributions: by design, the KRR approach cannot capture this drift since it takes into account only the source domain information. The method in \cite{courty2017joint} yields a curve that is closer to the target data than the KRR curve, but because of the outliers in the $Y_i^{(N_{\rm s})}$, it does not perform well when the feature takes on  negative values or values larger than four. Differently from the other methods, our method yields an estimated curve that is not too influenced by outlying values and that fits the target distribution better than the other methods.

\begin{figure}[http]
	\centering
	\includegraphics[scale=0.15]{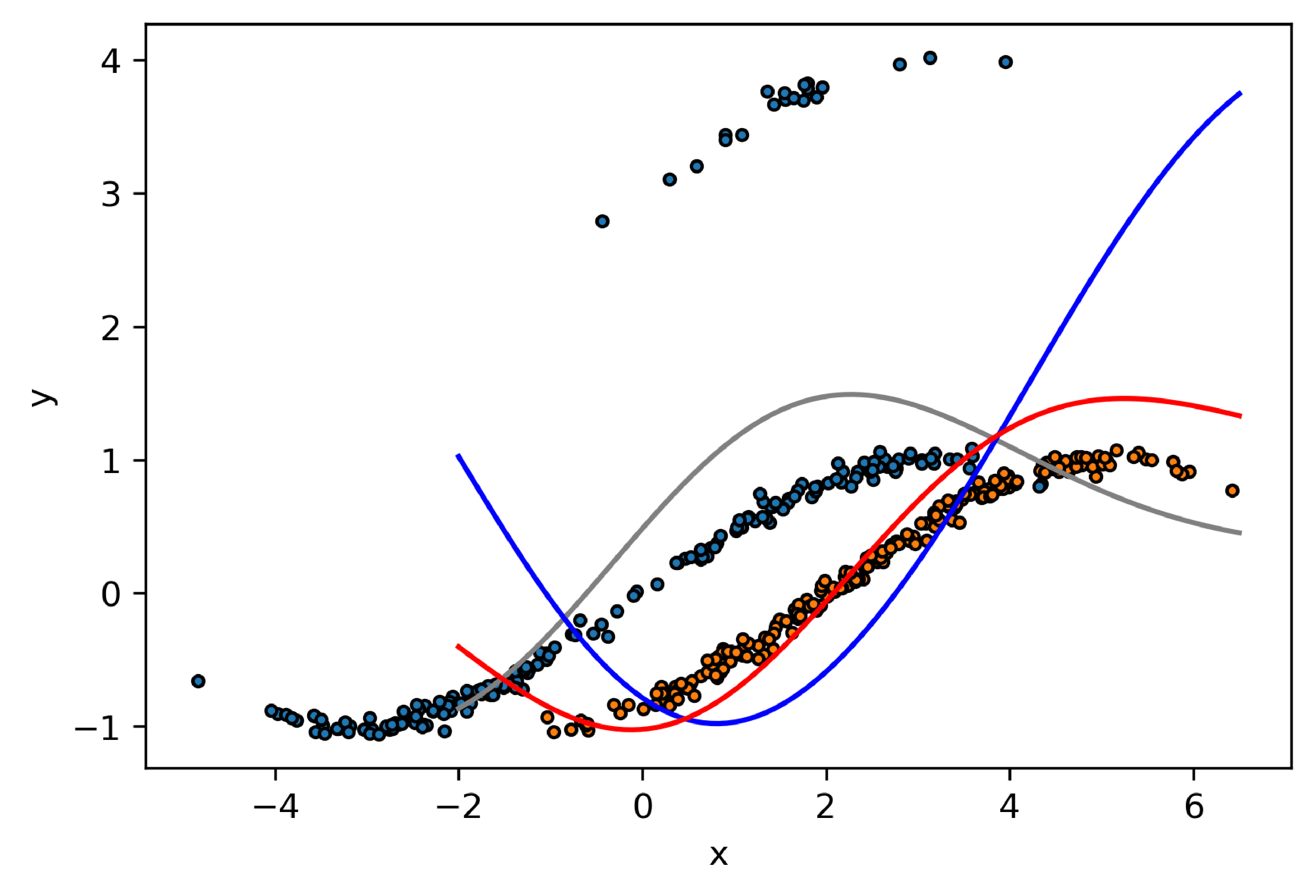}
	\caption{Scatter plots of the source and target samples, and functions estimated by different approaches.
		The grey-blue points represent the source samples, while the orange points represent the target samples. The gray, blue and red solid curves are obtained
		 by KRR, the method of \cite{courty2017joint} and our method, respectively. 		
	}
	\label{fig:DA}
\end{figure}

\section{Choice of $ \lambda$ via concentration inequalities}\label{Sec:selection}

The application of $W^{(\lambda)}$ requires the specification of $\lambda$.  In the literature on robust statistics, many selection criteria are based on first-order asymptotic theory. For instance, \cite{mancini2005optimal} and \cite{vecchia2012high} propose to select the tuning constant for Huber-type M-estimators achieving a given asymptotic efficiency at the reference model (namely, without considering the contamination). Unfortunately, these type of criteria hinge on the asymptotic  notion of influence function, which is not available for our estimators defined in Section~\ref{Robust Wasserstein estimator}; see  \cite{bassetti2006asymptotic} for a related discussion on MKE. \cite{mukherjee2021outlier} propose a criterion which needs  access to a clean data set and do not develop its theoretical underpinning. To fill this gap in the literature,
we introduce a selection criterion of $\lambda$ that does not have an asymptotic nature and is data-driven. Moreover, it does not need the computation of $\hat\theta^\lambda_n$ for each  $\lambda$, that makes its implementation fast.

To start with, let us notice that the term 
$W^{(\lambda)}(\mu_{\hat{\theta}^\lambda_n},\mu_\star)$ expresses
the robust Wasserstein distance between the fitted model ($\mu_{\hat{\theta}^\lambda_n}$) and the actual data generating measure ($\mu_\star$).   Intuitively, outlying values induce bias in the estimates, entailing abrupt changes in the distribution of $\hat{\theta}^\lambda_n$. This, in turn, entails changes in $W^{(\lambda)}(\mu_{\hat{\theta}^\lambda_n},\mu_\star)$. Therefore, the stability of the estimates and the stability of $W^{(\lambda)}(\mu_{\hat{\theta}^\lambda_n},\mu_\star)$ are tied together. 

This consideration is the stepping stone of our selection criterion for $\lambda$. 
To elaborate further, 
repeated applications of the triangle inequality yield
\begin{small}
\begin{align}
W^{(\lambda)}(\mu_{\hat{\theta}^\lambda_n},\mu_\star)&\leq W^{(\lambda)}(\mu_{\hat{\theta}^\lambda_n},\hat\mu_n)+W^{(\lambda)}(\hat\mu_n,\mu_\star) \leq W^{(\lambda)}(\mu_{\theta^*},\mu_\star)+2W^{(\lambda)}(\hat\mu_n,\mu_\star)  \nonumber \\
&\leq W^{(\lambda)}(\mu_{\theta^*},\mu_\star) +  2{\rm E}[ W^{(\lambda)}(\hat\mu_n,\mu_\star)] +2\vert W^{(\lambda)}(\hat\mu_n,\mu_\star) -{\rm E}[W^{(\lambda)}(\hat\mu_n,\mu_\star)]\vert,  \label{Eq. TrIneq}
\end{align}
\end{small}
which implies that the  
term $ W^{(\lambda)}(\hat\mu_n,\mu_\star)-{\rm E}[W^{(\lambda)}(\hat\mu_n,\mu_\star)$
 has an influence on the behavior of $W^{(\lambda)}(\mu_{\hat{\theta}^\lambda_n},\mu_\star)$.
Therefore,  
we argue that  
to reduce the impact of outliers in the distribution of $ W^{(\lambda)}(\hat\mu_n,\mu_\star)$ (hence, on the distribution of $\hat{\theta}^\lambda_n$), we need to control the concentration of $W^{(\lambda)}(\hat\mu_n,\mu)$ about its mean. To study this quantity, we make use of~\eqref{IneqI} and~\eqref{IneqO} and we define  a data-driven and non-asymptotic (i.e. depending on the finite $n$, on the estimated $\sigma$ and on the guessed  $\vert O \vert$)
selection criterion for $\lambda$. 

We refer to Appendix~\ref{App.lambda} for the mathematical detail, here, we simply state the basic intuition. Our idea hinges on the fact that  $\lambda$ controls the impact that the outliers have on the concentration bounds of $ W^{(\lambda)}(\hat\mu_n,\mu_\star)$. Thus, one may think of selecting the value(s) of $\lambda$ ensuring a desired level of stability in the distribution of $ W^{(\lambda)}(\hat\mu_n,\mu_\star)$. Operationally, we encode this stability in the ratio between the quantiles of $ W^{(\lambda)}(\hat\mu_n,\mu_\star)$ in the presence and in the absence of
contamination (see the quantity $\mathcal{Q}$, as available in Appendix~\ref{App.lambda}). Then, we measure the impact that different values of $\lambda$ have on this ratio
 computing its derivative w.r.t. $\lambda$ (see  the quantity ${\partial \mathcal{Q}}/{\partial \lambda}$, as available in Appendix~\ref{App.lambda}). Finally, we specify a tolerance  for the changes in this derivative and we identify the value(s) of $\lambda$ yielding the desired level of stability.  
 
 The described procedure allows to identify a (range of values of) $\lambda$,  not too close to zero and not too large either. This is a suitable and sensible achievement. Indeed, on the one hand, values of $\lambda$ too close to zero imply that $W^{(\lambda)}$ trims too many observations which likely include inliers and most outliers. On the other hand, too large values of $\lambda$  imply that   $W^{(\lambda)}$ trims a very limited number of outliers; see Appendix \ref{Sec.choice} for a numerical illustration of this aspect in terms of MSE as a function of $\lambda$.  Additional methodological aspects of the selection procedure and its numerical illustration are deferred to Appendix~\ref{Sec.choice} and Appendix~\ref{App.lambda}. 
\color{black} 






\appendix

\section*{Appendix}
This Appendix 
complements the results available in the main text. Specifically, Appendix~\ref{App.OT} provides some key theoretical results of optimal transport---the reader familiar with optimal transport theory may skip Appendix~\ref{App.OT}.
All proofs  are collected in Appendix~\ref{App:proof}. Appendix~\ref{App.MCresults}, as a complement to \S\ref{Sec:Applications}, provides more details and additional  numerical results of a few applications of ROBOT---see  Appendix~\ref{App.ROBOT} for an application of ROBOT to detect outliers in the context of parametric estimation in a linear regression model. In Appendix~\ref{app.rwgan}, we provide more details about implementation of RWGAN. 
Appendix~\ref{Sec: gen_choice} contains the methodological aspects of $\lambda$ selection.  All our numerical results can be reproduced using the (\texttt{Python} and \texttt{R}) code available upon request to the corresponding author.
%

\subsection{Some key theoretical details about optimal transport}\label{App.OT}

This section is a supplementary to Section~\ref{Optimal transport}. For deep details about optimal transport, we refer the reader to~Villani (2003, 2009).

Measure transportation theory dates back to the celebrated work  \cite{monge1781memoire}, in which Monge formulated a mathematical problem that in modern language can be expressed as follows. Given two probability measures $\mu \in \mathcal{P(X)}$, $\nu \in \mathcal{P(Y)}$ and a cost function $c:\mathcal{X \times Y}\rightarrow [0,+\infty] $, Monge's problem is to solve the minimization equation
\begin{equation} \label{equb1}
	\inf \left\lbrace  {\int_{\mathcal{X}} c(x,\mathcal{T}(x)) d \mu(x): \mathcal{T}\# \mu =\nu }  \right\rbrace 
\end{equation}
(in the measure transportation terminology, $\mathcal{T} \# \mu =\nu$ means that $\mathcal{T}$ is pushing $\mu$ foward to $\nu$). The solution to the Monge's problem (MP) \eqref{equb1} is called {\it optimal transport map}. Informally, one says that $\mathcal{T}$ transports the mass represented by the measure $\mu$ to the mass represented by the measure $\nu$. The optimal transport map which solves problem \eqref{equb1} (for a given $c$)  hence naturally yields  minimal cost of transportation.

It should be stressed that Monge’s formulation has two undesirable prospectives. First, for some measures, no solution to \eqref{equb1} exists. 
Considering, for instance, the case that $\mu$ is a Dirac measure while $\nu$ is not, there is no transport map which transport mass between $\mu$ and $\nu$. Second, since the set $\{\mu, {\cal T}(\nu)\}$ of all measurable transport map is non-convex, solving  problem~\eqref{equb1} is algorithmically challenging.

Monge’s problem was revisited by \cite{kantorovich1942translocation}, in which a more flexible and computationally feasible formulation was proposed. The intuition behind the the Kantorovich's problem is that mass can be disassembled and combined freely. Let $\Pi (\mu,\nu)$  denote the set of all joint probability measures of $\mu \in \mathcal{P(X)}$ and $\nu \in \mathcal{P(Y)}$. Kontorovich's problem aims at finding a joint distribution $\pi   \in \Pi (\mu,\nu)$ which minimizes the expectation of the coupling between $X$ and $Y$ in terms of the cost function $c$, and it can be formulated as
\begin{equation} \label{equb2}
	\inf \left\lbrace \int_{\mathcal{X\times Y}} c(x,y) d \pi(x,y) : \pi \in \Pi(\mu,\nu) \right\rbrace .
\end{equation}
A solution to Kantorovich’s problem (KP) \eqref{equb2} is called an {\it optimal transport plan}.  Note that Kantorovich’s problem is more general than Monge’s one, since it allows for mass splitting. Moreover, unlike $\{\mu, {\cal T}(\nu)\}$, the set $\Pi(\mu,\nu)$ is convex, and the solution to  \eqref{equb2} exists under some mild assumptions on $c$, e.g., lower semicontinuous (see \cite[Chapter 4]{villani2009optimal}). 
\cite{brenier1987decomposition} established the relationship between   optimal transport plan   and  optimal transport map when the cost function is the squared Euclidean distance. More specifically, the optimal transport plan $\pi$ can be expressed as $(\mathrm{Id},\mathcal{T})_{\# }\mu$.

The dual form (KD) of Kantorovich’s primal minimization problem \eqref{equb2} is to solve the maximization problem 
\begin{equation}\label{equb3}
	\begin{array}{ll}
		&\sup \left\lbrace \int_{\mathcal{Y}} \phi d\nu - \int_{\mathcal{X}} \psi  d\mu : \phi \in C_{b}(\mathcal{Y}) , \psi \in C_b(\mathcal{X})  \right\rbrace \\
		&{\rm s.t.}\quad \phi(y) - \psi(x) \leq c(x,y),\enspace \forall(x,y),
	\end{array}
\end{equation}
where $C_b(\mathcal{X})$ is the set of bounded continuous functions on $\mathcal{X}$. 
According to Theorem 5.10 in \cite{villani2009optimal}, if the cost function $c$ is a lower semicontinuous, there exists a solution to the dual problem such that the solutions to KD and KP coincide, namely there is no duality gap. In this case, the solution takes the form
\begin{equation}\label{equ4}
	\phi(y) =\inf \limits_{x \in \mathcal{X}} [\psi(x)+c(x, y)] \quad \text{and} \quad
	\psi(x)=\sup \limits_{y \in \mathcal{Y}}[\phi(y)-c(x, y)],
\end{equation}
where the functions $\phi$ and $\psi$ are called {\it $c$-concave} and {\it $c$-convex}, respectively, and $\phi$ (resp. $\psi$) is called the {\it $c$-transform} of $\psi$ (resp. $\phi$). A special case is that when $ c$ is a metric on $ \mathcal{X} $, equation~\eqref{equb3} simplifies to 
\begin{equation}\label{bKBD}
	\sup \left\lbrace \int_{\mathcal{Y}} \psi d\nu - \int_{\mathcal{X}} \psi  d\mu : \psi \, \text{is 1-Lipschitz continuous}    \right\rbrace ,
\end{equation}
This case is commonly known as  Kantorovich-Rubenstein duality.

Another important concept in measure transportation theory is {\it Wasserstein distance}. Let $(\mathcal{X},d)$ denote a complete metric space equipped with a metric $d:\mathcal{X} \times \mathcal{X} \rightarrow \mathbb{R}$, and let $\mu$ and $\nu$ be two probability measures on $\mathcal{X}$. 
Solving the optimal transport problem in \eqref{equb2}, with the cost function  $c(x,y) = d^p(x,y)$, would introduce a distance, called  Wasserstein distance, between $\mu$ and $\nu$. More specifically, the Wasserstein distance of order $p$ $(p \geq 1)$ is defined by 
\eqref{equ5}.

\subsection{Proofs}\label{App:proof}

\noindent \textbf{Proof of Theorem~\ref{thm3}.} First, if $\psi(y)-\psi(x) \leq c(x,y)$, then $\psi(y) -c(x,y) \leq \psi(x)$ and hence
$$ \sup \limits_{y \in \mathcal{X}}(\psi(y)-c(x, y)) \leq \psi(x) ,$$
where the supremum of the left-hand side is reached when $y=x$. Therefore, $\psi(x) $ is {\it c-convex} and $\psi(x)= \psi^{c}(x)$. Combined with Theorem 5.10 in \cite{villani2009optimal}, the proof can be completed. \qed \\


\noindent \textbf{Proof of Lemma~\ref{lemma1}.} 
We will prove that $c_{\lambda}$ satisfies the axioms of a distance. Let $x,y,z$ be three points in the space $ \mathcal{X} $. First,  the symmetry of $c_{\lambda}(x,y)$, that is, $c_{\lambda}(x,y) =c_{\lambda}(y,x) $, follows immediately from  
$$
\min\{c(x,y),2\lambda\}= \min\{c(y,x),2\lambda\}.
$$

Next, it is obvious that $c_{\lambda}(x,x) =0$. Assuming that $c_{\lambda}(x,y) =0$, then $$\min\{c(x,y),2\lambda\}=c(x,y)=0,$$ 
which entails that  $x=y$ since $c(x,y)$ is a metric.  By definition, we also have $ c_{\lambda}(x,y) \geq 0 $ for all $ x, y \in \mathcal{X}$.

Finally, note that $$c_{\lambda}(x,z)+c_{\lambda}(y,z) > 2\lambda \geq \min\{c(x,y),2\lambda \} = c_{\lambda}(x,y)$$
when $c(x,z)$ or $c(y,z) > 2 \lambda$, and $$c_{\lambda}(x,z)+c_{\lambda}(y,z) = c(x,z)+c(y,z) \geq c(x,y) \geq c_{\lambda}(x,y)$$ when $c(x,z)  \leq 2 \lambda$ and $c(y,z) \leq 2 \lambda$. Therefore, we conclude that $c_{\lambda}$ is a metric on $ \mathcal{X} $.  \qed

\vspace{4mm}

\noindent \textbf{Proof of Theorem~\ref{thm1}.} 
We prove that $ W^{(\lambda)} $ satisfies the axioms of a distance.  First, $W^{(\lambda)}(\mu,\nu) \geq 0$  is obvious. Conversely, letting $ \mu ,\nu $ be two probability measures such that $W^{(\lambda)}(\mu,\nu) = 0$, then there  exists at least  an  optimal transport plan $ \pi $ \cite{villani2009optimal}. It is clear that the support set of $ \pi$ is the diagonal $(y=x)$. Thus, for all continuous and bounded functions $\psi$,  
\[ 
\int \psi d\mu = \int \psi(x)d\pi(x,y)=\int \psi(y) d\pi(x,y) = \int \psi d\nu,
\]
which implies $ \mu=\nu $.

Next, let $\mu_{1}, \mu_{2}$ and $\mu_{3} $ be three probability measures on $\mathcal{X}$, and let $ \pi_{12} \in \Pi (\mu_1,\mu_2) $, $\pi_{23} \in \Pi (\mu_2,\mu_3)  $ be  optimal transport plans.  By the Gluing Lemma, there exist a probability measure $ \pi$ in $\mathcal{P}(\mathcal{X}\times \mathcal{X} \times \mathcal{X}) $ with marginals $ \pi_{12} $, $ \pi_{23} $  and $ \pi_{13} $ on $ \mathcal{X}\times \mathcal{X}$. Use the triangle inequality in Lemma~\ref{lemma1}, we obtain
\[ 
\begin{aligned}
	W^{(\lambda)}(\mu_1,\mu_3) &\leq \int_{\mathcal{X}\times \mathcal{X}} c_{\lambda}(x_1,x_3) d\pi_{13}(x_1,x_3) \\
	&=\int_{\mathcal{X}\times \mathcal{X} \times \mathcal{X}} c_{\lambda}(x_1,x_3) d\pi(x_1,x_2,x_3)\\
	&\leq \int_{\mathcal{X}\times \mathcal{X} \times \mathcal{X}} [c_{\lambda}(x_1,x_2) + c_{\lambda}(x_2,x_3)] d\pi(x_1,x_2,x_3)\\
	&=\int_{\mathcal{X}\times \mathcal{X}} c_{\lambda}(x_1,x_2) d\pi_{12}(x_1,x_2) +
	\int_{\mathcal{X}\times \mathcal{X}} c_{\lambda}(x_2,x_3) d\pi_{13}(x_2,x_3) \\ 
	&=W^{(\lambda)}(\mu_1,\mu_2) +W^{(\lambda)}(\mu_2,\mu_3)
\end{aligned}
\]

Moreover, the symmetry is obvious since 
$$W^{(\lambda)}(\mu,\nu) = \inf \limits_{\pi \in \Pi(\mu,\nu)} \left( \int c_{\lambda}(x,y) d \pi(x,y) \right)  = W^{(\lambda)}(\nu,\mu).$$

Finally, note that $ W^{(\lambda)} $ is finite since
\[ 
W^{(\lambda)}(\mu,\nu)= \inf \limits_{\pi \in \Pi(\mu,\nu)}  \int c_{\lambda}(x,y) d \pi(x,y)
\leq \inf \limits_{\pi \in \Pi(\mu,\nu)}  \int 2\lambda  d \pi(x,y)  \leq +\infty.
\]

In conclusion, $ W^{(\lambda)} $ is a finite distance on $\mathcal{P(X)}$. \qed
\vspace{4mm}

\noindent \textbf{Proof of Theorem~\ref{thm2}.} 	
First, we prove that $W^{(\lambda)} $ is monotonically increasing with respect to $\lambda$. Assuming that $\lambda_2 > \lambda_1  > 0$, then $c_{\lambda_2} \geq c_{\lambda_1}  $. Suppose that there is an optimal transport plan $\pi_{\lambda_{\ell}}, \ell = 1, 2$ such that $\int c_{\lambda_{\ell}} d\pi_{\lambda_{\ell}}$ reaches its minimum. Then $\pi_{\lambda_2}$ is not necessarily the optimal transport plan when the cost function is $c_{\lambda_1} $. We therefore have $\int c_{\lambda_2} d\pi_{\lambda_2} > \int c_{\lambda_1} d\pi_{\lambda_2}\geq \int c_{\lambda_1} d\pi_{\lambda_1}$.

Then we prove the continuity of $W^{(\lambda)} $.  Fixing $\varepsilon > 0 $ and letting  $2\lambda_2- 2\lambda_1 < \varepsilon $, we have
\[ 
\begin{aligned}
	\left|W^{(\lambda_2)}- W^{(\lambda_1)}\right| &= \left|\int c_{\lambda_2} d\pi_{\lambda_2}- \int c_{\lambda_1} d\pi_{\lambda_1}  \right|\\
	&\leq \left| \int c_{\lambda_2} d\pi_{\lambda_1}- \int c_{\lambda_1} d\pi_{\lambda_1}\right|\\
	&= \left|\int (c_{\lambda_2} - c_{\lambda_1}) d\pi_{\lambda_1} \right|\\
	&\leq 2\lambda_2 - 2\lambda_1 \leq \varepsilon. 
\end{aligned}
\]
Therefore, $ W^{(\lambda)} $ is continuous.   

Now we turn to the last part of Theorem~\ref{thm2}. Note that if $W_1$ exists, we have $ W^{(\lambda)} \leq W_1$ for any $ \lambda \in [0, \infty)$. Since $ W^{(\lambda)} $ is increasing with respect to $\lambda$, its limit $ W^* := \lim_{\lambda \rightarrow \infty} W^{(\lambda)}(\mu,\nu) $ exists. It remains to prove $ W^* = W_1 $.
Since $ W_1 <+\infty  $,  for any fixing $ \varepsilon >0 $, there exists a number $ R $ such that $ \int_{c > R } c d \pi < \varepsilon,$ where $ \pi $ is the  optimal transport plan corresponding to the cost function $ c $.  Let $2\lambda > R$, then $ W^{(\lambda)} > \int_{c \leq R } c d \pi  $ and $ \vert W^{(\lambda)} - W_1 \vert < \varepsilon$. Finally, since $ \varepsilon $ is arbitrarily small, we have $ W^* = W_1$.  \qed

\vspace{4mm}

In order to prove Theorem~\ref{thm4}, we need a lemma which shows that a 
Cauchy sequence in robust Wasserstein distance is tight. 

\begin{lemma}\label{lemma2}
	Let $\mathcal{X}$ be a Polish space and $(\mu_{k})_{k\in \mathbb{N}}$ be a Cauchy sequence in $(\mathcal{P(X)},W^{(\lambda)})$. Then $(\mu_k)_{k\in \mathbb{N}}$ is tight. 
\end{lemma}

\noindent \textbf{Proof of Lemma~\ref{lemma2}.}
Because $(\mu_{k})_{k\in \mathbb{N}} $ is a Cauchy sequence, we have
$$W^{(\lambda)}(\mu_{k},\mu_{l}) \rightarrow 0$$ as $ k,l \rightarrow \infty.$
Hence, for $ \varepsilon >0 $, there is a $ N $ such that  
\begin{equation}\label{A1}
	W^{(\lambda)}(\mu_{k},\mu_{l}) \leq \varepsilon \quad \text{for} \ \ k,l \geq N
\end{equation}
Then for any $k \in \mathbb{N} $, there is a $j$ such that $W^{(\lambda)}(\mu_{k},\mu_{j}) < \varepsilon^2 $ (if $k>N $, this is \eqref{A1}; if $k\leq N $, we just let $j=k $).

Since the finite set $\left\lbrace\mu_{1},\mu_{2},\dots , \mu_{N} \right\rbrace  $ is tight, there is a compact set $K $ such that $\mu_{j}[\mathcal{X} \backslash K_{\varepsilon} ] < \varepsilon $ for all $j \in \left\lbrace1, \dots , N \right\rbrace  $. By compactness, $K_{\varepsilon}$ can be covered by a finite number of small balls, that is, $K_{\varepsilon} \subset B(x_1,\varepsilon) \cup \dots \cup  B(x_m,\varepsilon)$ for a fixed integer $m$.

Now write 
\[
\begin{array}{cc}
	U_{\varepsilon}:= B(x_1,\varepsilon) \cup \dots \cup B(x_m,\varepsilon),  \\
	\tilde{U}_{\varepsilon} := \left\lbrace x \in \mathcal{X}: d(x,U) < \varepsilon \right\rbrace \subset B(x_1,2\varepsilon) \cup \dots \cup B(x_m,2\varepsilon),\\
	\phi(x):=\left(1- \dfrac{d(x,U)}{\varepsilon} \right)_{+}.
\end{array}
\]
Note that $\mathbbm{1}_{U_{\varepsilon}} \leq \phi \leq \mathbbm{1}_{\tilde{U}_{\varepsilon}}$ and $\phi$ is $(1 / \varepsilon)$-Lipschitz. By Theorem~\ref{thm3}, we find that for $j \leq N$, $\dfrac{2\lambda}{\varepsilon} \geq 1$ (this is reasonable because we need $ \varepsilon $ as small as possible) and $k$ arbitrary,
\[ 
\begin{aligned}
	\mu_{k}\left[\tilde{U}_{\varepsilon}\right] & \geq \int \phi d \mu_{k} \\
	&=\int \phi d \mu_{j}+\left(\int \phi d \mu_{k}-\int \phi d \mu_{j}\right) \\
	& \mathop{\geq}\limits_{(1)}  \int \phi d \mu_{j}-\frac{W^{(\lambda)}\left(\mu_{k}, \mu_{j}\right)}{\varepsilon} \\
	& \geq \mu_{j}[U_{\varepsilon}]-\frac{W^{(\lambda)}\left(\mu_{k}, \mu_{j}\right)}{\varepsilon}.
\end{aligned}
\]
Inequality $ (1) $ holds for
\[ 
\begin{aligned}
	\int \phi d \mu_{k}-\int \phi d \mu_{j}
	\leq& \sup_{\mathrm{range}(f)\leq 1 \atop |f(x_1)-f(x_2)|\leq \frac{1}{\epsilon}d(x_1,x_2) } \left\lbrace \int f d \mu_{k}-\int f d \mu_{j}   \right\rbrace \\
	\leq& \sup_{\mathrm{range}(f)\leq \frac{2\lambda}{\epsilon} \atop |f(x_1)-f(x_2)|\leq \frac{1}{\epsilon}d(x_1,x_2) } \left\lbrace \int f d \mu_{k}-\int f d \mu_{j}   \right\rbrace \\
	=& \sup_{\mathrm{range}(f)\leq 2\lambda \atop |f(x_1)-f(x_2)|\leq d(x_1,x_2) } \frac{1}{\epsilon} \left\lbrace \int f d \mu_{k}-\int f d \mu_{j}   \right\rbrace \\
	&=\frac{W^{(\lambda)}\left(\mu_{k}, \mu_{j}\right)}{\varepsilon}.
\end{aligned}
\]
Note that $\mu_{j}[U_{\varepsilon}] \geq \mu_{j}[K_{\varepsilon}] \geq 1-\varepsilon$ if $j \leq N$, and, for each $k$, we can find $j=j(k)$ such that $W_{1}\left(\mu_{k}, \mu_{j}\right) \leq \varepsilon^{2}$. We therefore have 
$$
\mu_{k}\left[\tilde{U}_{\varepsilon}\right] \geq 1-\varepsilon-\frac{\varepsilon^{2}}{\varepsilon}=1-2 \varepsilon.
$$
Now we have shown the following: For each $\varepsilon>0$ there is a finite family $\left(x_{i}\right)_{1 \leq i \leq m}$ such that all measures $\mu_{k}$ give mass at least $1-2 \varepsilon$ to the set $Z_{\varepsilon} :=\cup B\left(x_{i}, 2 \varepsilon\right)$. 

For $ \ell\in \mathbb{N} $, we can find $\left(x_{i}\right)_{1 \leq i \leq m(\ell)}$ such that 
$$
\mu_{k}\left[\mathcal{X} \backslash \bigcup_{1 \leq i \leq m(\ell)} B\left(x_{i}, 2^{-\ell} \varepsilon\right)\right] \leq 2^{-\ell} \varepsilon.
$$
Now we let 
$$
S_{\varepsilon} :=\bigcap_{1 \leq p \leq \infty} \bigcup_{1 \leq i \leq m(p)} \overline{B\left(x_{i}, 2^{-p} \varepsilon \right)}.
$$
It is clear that $
\mu_{k}[\mathcal{X} \backslash S_{\varepsilon}] \leq \varepsilon
$.

By construction, $ S_{\varepsilon} $ is closed and it can be covered by finitely many balls of arbitrarily small radius, so it is also  totally bounded. We note that in complete metric space, a closed  totally bounded set is equivalent to a compact set. In conclusion, $ S_{\varepsilon} $ is compact. The result then follows.  \qed

\vspace{4mm}

\noindent \textbf{Proof of Theorem~\ref{thm4}.}
First, we prove that $\mu_{k}$ converges to $\mu$ in ${\cal P}(\mathcal{X})$ if $W^{(\lambda)}\left(\mu_{k}, \mu\right) \rightarrow 0$. 
By Lemma~\ref{lemma2}, the sequence $\left(\mu_{k}\right)_{k \in \mathbb{N}}$ is tight, so there is a subsequence $\left(\mu_{k^{\prime}}\right)$ such that $\mu_{k^{\prime}}$ converges weakly to some probability measure $\widetilde{\mu}$.  Let $ h $ be the solution of the dual form of ROBOT for $\widetilde{\mu}$ and $ \mu $. Then, by Theorem~\ref{thm3},
$$
\begin{aligned}
	W^{(\lambda)}(\widetilde{\mu}, \mu) & =   \int h d  \widetilde{\mu} -\int h d \mu  \\
	&\mathop{=}\limits_{(1)}   \lim \limits_{k^{\prime} \rightarrow \infty}    \left\lbrace \int h d \mu_{k^{\prime}}-\int h d \mu   \right\rbrace\\
	&\leq \lim \limits_{k^{\prime} \rightarrow \infty}  \sup_{\mathrm{range}(f)\leq 2\lambda \atop |f(x_1)-f(x_2)|\leq d(x_1,x_2) }    \left\lbrace \int f d \mu_{k^{\prime}}-\int f d \mu   \right\rbrace\\
	&=\lim \limits_{k^{\prime} \rightarrow \infty} W^{(\lambda)}(\mu,\mu_{k^{\prime}})\\
	&\mathop{=}\limits_{(2)} \lim\limits_{k  \rightarrow \infty} W^{(\lambda)}(\mu,\mu_{k})
	=0,
\end{aligned}
$$
where $ (1) $ holds since the subsequence $(\mu_{k^{\prime}})$ converges weakly to $ \tilde{\mu} $ and $ (2) $ holds since $ (\mu_{k^{\prime}}) $ is a subsequence of $ (\mu_{k}) $. 
Therefore, $\tilde{\mu}=\mu$, and the whole sequence $\left(\mu_{k}\right)_{k \in \mathbb{N}}$ has to converge to $\mu$ weakly.

Conversely, suppose $\left(\mu_{k}\right)_{k \in \mathbb{N}}$ converges to $\mu $ weakly. 
By Prokhorov's theorem, $\left(\mu_{k}\right)$ forms a tight sequence; also, $\mu$ is tight. Let $ (\pi_k) $ be a sequence representing the optimal plan for $ \mu $  and $ \mu_k $. By Lemma 4.4 in \cite{villani2009optimal}, the sequence $\left(\pi_{k}\right)$ is  tight in $P(\mathcal{X} \times \mathcal{X})$. So, up to the extraction of a subsequence, denoted by $\left(\pi_{k_{1}}\right)$, we may assume that $\pi_{k_{1}}$  converges to $\pi$ weakly in $P(\mathcal{X} \times \mathcal{X})$. Since each $\pi_{k_{1}}$ is optimal, Theorem $5.20$ in \cite{villani2009optimal} guarantees that $\pi$ is an optimal coupling of $\mu $ and $\mu$, so this is the coupling $\pi=(\mathrm{Id}, \mathrm{Id})_{\#} \mu$. In fact, this is independent of the extracted subsequence. So $ \pi_k $ converges to $ \pi $ weakly. Finally, 
\[ 
\lim \limits_{k \rightarrow \infty} W^{(\lambda)}(\mu,\mu_k)
= \lim \limits_{k \rightarrow \infty} \int c_{\lambda} d \pi_k=
\int c_{\lambda} d \pi= 0.
\]
\qed 

\vspace{4mm}

\noindent \textbf{Proof of Theorem~\ref{thm5}.}
First, we prove   completeness.
Let $ (\mu_k )_{k\in \mathbb{N}} $ be a Cauchy sequence. By Lemma~\ref{lemma2}, $ (\mu_k )_{k\in \mathbb{N}} $ is tight and there is a subsequence $ (\mu_{k^{\prime}}) $  which convergence to a measure $ \mu $ in $ \mathcal{P(X)} $.  
Therefore, the continuity of $ W^{(\lambda)} $ (Corollary~\ref{cor1}) entails that 
\[  
\lim \limits_{l\rightarrow \infty}  W^{(\lambda)}(\mu,\mu_{l}) = \lim \limits_{k',l\rightarrow \infty} W^{(\lambda)}(\mu_{k'},\mu_{l})=0.
\]
The result of the completeness follows.

Then, we complete the proof of separability by using a rational step function approximation. 
We notice that  $\mu $ lies in $\mathcal{P(X)} $ and  is integrable. So for  $\varepsilon > 0 $,  we can find a compact set $\mathcal{K}_{\varepsilon} \subset \mathcal{X} $ such that
$\int_{\mathcal{X \textbackslash K_{\varepsilon}}} 2\lambda d\mu \leq \varepsilon.$
Letting $ \mathcal{D} $ be a countable dense set in $ \mathcal{X} $,
we can find a finite family of balls $(B(x_k,\varepsilon))_{1 \leq k\leq N}, x_k \in \mathcal{D}$,  which cover $\mathcal{K}_{\varepsilon} $. Moreover,  by letting 
\[
B_{k}= B(x_k,\varepsilon) \text{\textbackslash} \cup_{j < k} B(x_{j},\varepsilon),
\]
then $\cup_{1\leq k\leq N} {B_k} $ still covers the $\mathcal{K}_{\varepsilon} $, and $B_{k}$ and $B_{\ell}$ are disjoint for $k \neq \ell$. Then by defining a step function 
$f(x) := \sum_{k = 1}^{N} x_k \mathbbm{1}_{x \in B_k}$, 
we can easily find 
\[ 
\int c_{\lambda}(x,f(x)) d\mu(x)\leq \varepsilon \int_{\mathcal{K}_{\varepsilon}} d\mu(x) + \int_{\mathcal{X \textbackslash K_{\varepsilon}}} 2\lambda d\mu(x)  \leq 2\varepsilon.
\]

Moreover, $ f_{\#}\mu $ can be written as $ \sum_{j = 1}^{N} a_j \delta_{x_j}$.  Next, we prove that $ \sum_{j = 1}^{N} a_j \delta_{x_j}$ can be approximated to arbitrary accuracy by another step function $\sum_{j = 1}^{N} b_j \delta_{x_j}  $ with rational  coefficients $ (b_j)_{1 \leq j\leq N}$. Specifically,
\[ 
\begin{aligned}
	&W^{(\lambda)}\left(\sum_{j = 1}^{N} a_j \delta_{x_j},\sum_{j = 1}^{N} b_j \delta_{x_j}\right)\\
	\leq
	&\int \left[ \mathop{\max}\limits_{k,l} d(x_k,x_l) \right] d(\sum_{j = 1}^{N} a_j \delta_{x_j}) + \int \left[ \mathop{\max}\limits_{k,l} d(x_k,x_l) \right] d(\sum_{j = 1}^{N} b_j \delta_{x_j})\\
	\leq& 4\lambda \mathop{\sum_{j = 1}^{N}}|a_j-b_j|.
\end{aligned}
\]
Therefore, we can replace $ a_j $ with some well-chosen rational coefficients $ b_j $, and $ \mu $ can be approximated by $ \sum b_j \delta_{x_j},\, 0\leq  j \leq N$  with arbitrary precision. In conclusion,  the set of  finitely supported measures with rational coefficients is dense in $ \mathcal{P(X)} $. \qed

\vspace{4mm}

\noindent \textbf{Proof of Theorem~\ref{thm:NewCI}.}  The proof derives from \cite[Theorem 5.1]{Lei20}. To apply this result, we have to check that Condition (11) p 779 in \cite{Lei20} holds. 
For this, let $X_1',\ldots,X_n'$ denote i.i.d. data with common distribution $\mu$, independent from $X_1,\ldots,X_n$.
Denote by $I=\{i_1,\ldots,i_{|I|}\}$, fix $j\in \{1,\ldots,|I|\}$ and let $X=(X_j)_{j\in I}$, $X'=(X_{i_1},\ldots,X_{i_j}',\ldots,X_{i_{|I|}})$ (the data $X_{i_j}$ in the original dataset of inliers has been replaced by an independent copy $X_{i_j}'$).
Let $\tilde{\mu}_n^{(I)}=|I|^{-1}(\sum_{i\in I\setminus \{i_j\}}\delta_{X_i}+\delta_{X_{i_j}'})$ denote the associated empirical measure.
Let us define the function $f(X)=W^{(\lambda)}(\mu,\hat{\mu}_n^{(I)})$. We have, by the triangular inequality established in Theorem~\ref{thm1}:
\[
|f(X)-f(X')|\leqslant W^{(\lambda)}(\hat{\mu}^{(I)}_n,\tilde{\mu}^{(I)}_n)\enspace.
\]
By the duality formula given in Theorem~\ref{thm3}, we have moreover
\[
W^{(\lambda)}(\hat{\mu}^{(I)}_n,\tilde{\mu}^{(I)}_n)\leqslant \sup_{\psi\in\Psi}\frac{\psi(X_{i_j})-\psi(X_{i_j}')}{|I|}\leqslant \frac{\min(d(X_{i_j},X_{i_j}'),2\lambda)}{|I|}\enspace.
\]
Therefore, we deduce that, for any $k\geqslant 2$,
\[
\E[W^{(\lambda)}(\hat{\mu}^{(I)}_n,\tilde{\mu}^{(I)}_n)^k]\leqslant \frac{(2\lambda)^{k-2}}{|I|^k}\E[\min(d^2(X_{i_j},X_{i_j}'),(2\lambda)^2)]\enspace.
\]
Thus, the function $f$ satisfies Condition (11) p 779 in \cite{Lei20} with $M=2\lambda/|I|$ and $\sigma_j^2=\sqrt{\E[\min(d^2(X_{i_j},X_{i_j}'),(2\lambda)^2)]}/|I|$, thus by Theorem 5.1 in \cite{Lei20}, we deduce that, for any $t>0$,
\[
\P(W^{(\lambda)}(\hat{\mu}^{(I)}_n,\mu^*)-\E[W^{(\lambda)}(\hat{\mu}^{(I)}_n,\mu^*)]>t)\leqslant \exp\bigg(-\frac{|I|t^2}{2\sigma^2+4\lambda t}\bigg)\enspace.
\]
Next, letting
$$f(X)=-W^{(\lambda)}(\mu,\hat{\mu}_n^{(I)})$$
and repeating the above arguments yield
\[
\P(W^{(\lambda)}(\hat{\mu}^{(I)}_n,\mu^*)-\E[W^{(\lambda)}(\hat{\mu}^{(I)}_n,\mu^*)]<-t)\leqslant \exp\bigg(-\frac{|I|t^2}{2\sigma^2+4\lambda t}\bigg)\enspace.
\] 
Inequality~\eqref{IneqI} then follows from piecing together the above two inequalities.

Inequality~\eqref{IneqO} is obtained by plugging the first inequality into inequality~\eqref{eq:LinkedIn}.
\qed

\vspace{4mm}

\noindent \textbf{Proof of Theorem~\ref{Matt_Th2}.}
The proof is divided into three steps.
We start by using the duality formula derived in Theorem~\ref{thm3} and prove a slight extension of Dudley's inequality in the spirit of \cite{boissard2014mean} to reduce the bound to the computation of the entropy numbers of the $\Psi$ of $1$-Lipschitz functions defined on $B(0,K)$ and taking values in $[-\lambda,\lambda]$. 
Then we compute these entropy numbers using a slight extension of the computation of these numbers done, for example in \cite[Exercise 8.2.8]{vershynin2018high}.
Finally, we put together all these ingredients to derive our final results.

\textbf{Step 1: Reduction to bounding entropy.}
Using the dual formulation of $W^{(\lambda)}$, we get that 
\[
W^{(\lambda)}(\mu,\hat{\mu}_n)=\sup_{\psi\in \Psi}\frac1n\sum_{i=1}^n\psi(X_i)-\E[\psi(X_i)]\enspace.
\]
Fix first $\psi$ and $\psi'$ in $\Psi$. 
The random variables 
\[
Y_i(\psi,\psi')=\psi(X_i)-\E[\psi(X_i)]-\{\psi'(X_i)-\E[\psi'(X_i)]\}
\]
are independent, centered and take values in $[-4\lambda\wedge 4K\wedge 2\|\psi-\psi'\|_\infty,4\lambda\wedge 4K\wedge 2\|\psi-\psi'\|_\infty]$ and therefore by Hoeffding's inequality, the process $(Y_\psi)_{\psi\in \Psi}$, where the variables $$Y_\psi=\frac1n\sum_{i=1}^n\psi(X_i)-\E[\psi(X_i)],$$ has increments $Y_\psi-Y_{\psi'}$ satisfying
\begin{equation}\label{eq:SG}
	\forall s\in \R,\qquad \E\big[\exp\big(s(Y_\psi-Y_{\psi'})\big)\big]\leqslant \exp\bigg(\frac{s^2(\gamma^2\wedge \|\psi-\psi'\|_{\infty}^2)}{n}\bigg)\enspace,
\end{equation}
where, here and in the rest of the proof $\delta=\sqrt{2}(\lambda\wedge K)$.
Introduce the truncated sup-distance 
$$
d_{\infty,\delta}(f,g)=\|f-g\|_{\infty}\wedge \delta,
$$ we just proved that the process $Y_{\psi}$ has sub-Gaussian increments with respect to the distance $d_{\infty,\delta}$.

At this point, we will use a chaining argument to bound the expectation $\E[\sup_{\psi\in\Psi}Y_\psi]$.
Let $k_0$ denote the largest $k$ such that $2^{-k}\geqslant \sqrt{2}\delta$.
As $2^{-k_0}$ is larger than the diameter of $\Psi$ w.r.t. $d_{\infty,\delta}$, we define $\Psi_{k_0}=\{0\}$ and all functions $\psi\in\Psi$ satisfy $d_{\infty,\delta}(\psi,0)=d_{\infty,\delta}(\psi,\Psi_0)\leqslant 2^{-k_0}$.
For any $k\geqslant k_0$ denote by $\Psi_k$ a sequence of $2^{-k}$ nets of $\Psi$ with minimal cardinality, so $|\Psi_k|\leqslant|\Psi_{k+1}|$.
For any $k\geqslant k_0$ and $\psi\in\Psi$, let $\pi_k(\psi)\in \Psi_k$ be such that $\|\psi-\pi_k(\psi)\|_\infty\leqslant 2^{-k}$.
We have $\pi_{k_0}(\psi)=0$ for any $\psi\in \Psi$, so we can write, for any $k>k_0$,
\begin{equation}\label{eq:Chaining0}
	\psi=\psi-\pi_{k_0}(\psi)=\psi-\pi_k(\psi)+\sum_{\ell=k_0+1}^k\pi_{\ell}(\psi)-\pi_{\ell-1}(\psi)\enspace.
\end{equation}
Write $\alpha_n(\psi)=\int\psi\rmd\hat\mu_n-\int\psi\rmd\mu$, we have thus
\begin{equation}\label{eq:Chaining}
	\alpha_n(\psi)=\alpha_n(\psi-\pi_k(\psi))+\sum_{\ell=k_0+1}^k\alpha_n(\pi_{\ell}(\psi)-\pi_{\ell-1}(\psi))\enspace.
\end{equation}
We bound both terms separately.
For any function $g$, $|\alpha_n(g)|\leqslant 2\|g\|_{\infty}$, so $\sup_{\psi\in\Psi}|\alpha_n(\psi-\pi_k(\psi))|\leqslant 2^{1-k}$.
On the other hand, for any $\psi$ and $\ell$, by \eqref{eq:SG},
\[
\forall s\in \R,\qquad \E\bigg[\exp\big(s\alpha_n(\pi_{\ell}(\psi)-\pi_{\ell-1}(\psi))\big)\bigg]\leqslant \exp\bigg(\frac{s^2(\delta^2\wedge\|\pi_{\ell}(\psi)-\pi_{\ell-1}(\psi)\|_\infty^2}n\bigg)\enspace.
\]
For any $\ell>k_0$,
\[
\|\pi_{\ell}(\psi)-\pi_{\ell-1}(\psi)\|_\infty\leqslant \|\pi_{\ell}(\psi)-\psi\|_\infty+\|\psi-\pi_{\ell-1}(\psi)\|_\infty\leqslant 2^{2-\ell}\enspace, 
\]
we deduce that
\begin{equation}\label{eq:SGInc}
	\forall s\in \R,\qquad \E\bigg[\exp\big(s\alpha_n(\pi_{\ell}(\psi)-\pi_{\ell-1}(\psi))\big)\bigg]\leqslant \exp\bigg(\frac{16s^2}{n2^{2\ell}}\bigg)\enspace.
\end{equation}
Now we use the following classical result.
\begin{lemma}\label{lem:PM}
	If $Z_1,\ldots,Z_m$ are random variables satisfying 
	\begin{equation}\label{eq:AssLemmaPM}
		\forall s\in\R,\qquad \E[\exp(sZ_i)]\leqslant \exp(s^2\kappa^2)\enspace,
	\end{equation}
	then 
	\[
	\E[\max_{i}Z_i]\leqslant 2\kappa\sqrt{\log m}\enspace.
	\]
\end{lemma}
We apply Lemma~\ref{lem:PM} to the random variable $Z_{\psi,\psi'}=\alpha_n(\pi_{\ell}(\psi)-\pi_{\ell-1}(\psi))\big)$.
There are less than $|\Psi_\ell||\Psi_{\ell-1}|$ such variables, and all of them satisfy Assumption~\ref{eq:AssLemmaPM} with $\kappa=4/(\sqrt{n}2^\ell)$ thanks to \eqref{eq:SGInc}.
Therefore, Lemma~\ref{lem:PM} and $|\Psi_{\ell-1}|\leqslant |\Psi_{\ell}|$ imply that
\[
\E\bigg[\sup_{\psi\in\Psi}\alpha_n(\pi_{\ell}(\psi)-\pi_{\ell-1}(\psi))\big)\bigg]\leqslant \frac{8}{\sqrt{n}2^{\ell}}\sqrt{\log\big(|\Psi_\ell||\Psi_{\ell-1}|\big)}\leqslant \frac{8}{\sqrt{n}2^{\ell}}\sqrt{2\log\big(|\Psi_\ell|\big)}\enspace.
\]
We conclude that, for any $k>k_0$, it holds that
\begin{equation}\label{eq:Chaining2}
	\E[W^{(\lambda)}(\mu,\hat{\mu}_n)]\leqslant 2^{1-k}+\frac{8\sqrt{2}}{\sqrt{n}}\sum_{\ell=k_0+1}^k\frac{\sqrt{\log\big(|\Psi_\ell|\big)}}{2^{\ell}}\enspace.
\end{equation}
To exploit this bound, we have now to bound from above the entropy numbers $|\Psi_\ell|$ of $\Psi$.

\textbf{Step 2: Bounding the entropy of $\Psi$.}
Let $\epsilon\in(0,K\wedge\lambda)$ and let $N$ denote an $\epsilon$-net of $B(0,K)$ with size $|N|\leqslant  (3K/\epsilon)^d$ (for a proof that such nets exist, we refer for example to \cite[Corollary 4.2.13]{vershynin2018high}: For any $x\in B(0,K)$, there exists $y\in N$ such that $\|x-y\|\leqslant \epsilon$.
Let also $G$ denote an $\epsilon$-grid of $[-\lambda,\lambda]$ with step size $\epsilon$ and $|G|\leqslant 3\lambda/\epsilon$.
We define the set $\cF$ of all functions $f: B(0,K)\to[-\lambda,\lambda]$ such that 
\[
\forall x,y\in N, \qquad f(x)\in G, \qquad |f(x)-f(y)|\leqslant \|x-y\|\enspace,
\]
and $f$ linearly interpolates between points in $N$.
Starting from $f(0)$ and moving recursively to its neighbors, we see that
\[
|\cF|\leqslant \frac{3\lambda}{\epsilon}3^{|N|}\enspace,
\]
where the first term counts the number of choices for $f(0)$ and $3$ is an upper bound on the number of choices for $f(x)$ given the value of its neighbor $f(x')$ due to the constraint $|f(x)-f(y)|\leqslant \|x-y\|$.
This shows that there exists a numerical constant $C$ such that 
\[
|\cF|\leqslant \frac{3\lambda}{\epsilon}\exp\bigg(\bigg(\frac{CK}{\epsilon}\bigg)^d\bigg)\enspace.
\]
By construction, $\cF\subset\Psi$.
Fix now $\psi\in\Psi$.
By construction of $\cF$, there exists $f\in\cF$ such that $|f(x)-\psi(x)|\leqslant \epsilon$ for any $x\in N$. 
Moreover, for any $y\in B(0,K)$, there exists $x\in N$ such that $\|x-y\|\leqslant \epsilon$.
As both $\psi$ and $f$ are $1$-Lipschitz, we have therefore that
\[
|f(y)-\psi(y)|\leqslant |f(y)-f(x)|+|f(x)-\psi(x)|+|\psi(x)-\psi(y)|\leqslant 3\epsilon\enspace.
\]
We have thus established that there exists a numerical constant $C$ such that
\begin{equation}\label{eq:ENPsi}
	\forall \epsilon\leqslant K\wedge \lambda,\qquad N(\Psi,d_\infty,\epsilon)\leqslant \frac{9\lambda}{\epsilon}\exp\bigg(\bigg(\frac{CK}{\epsilon}\bigg)^d\bigg)\enspace.
\end{equation}

\textbf{Step 3: Conclusion of the proof.} Plugging the estimate \eqref{eq:ENPsi} into the chaining bound \eqref{eq:Chaining2} yields, for any $k>k_0$,
\begin{align*}
	\E[W^{(\lambda)}(\mu,\hat{\mu}_n)]&\leqslant 2^{1-k}+\frac{8\sqrt{2}}{\sqrt{n}}\sum_{\ell=k_0+1}^k\frac{\sqrt{\log\big(9*2^\ell\lambda\big)}+\big(C2^\ell K\big)^{d/2}}{2^{\ell}}\\
	&\leqslant 2^{1-k}+\frac{C}{\sqrt{n}}\bigg(\delta+(CK)^{d/2}\sum_{\ell=k_0+1}^k2^{\ell(d-2)/2}\bigg)\enspace.
\end{align*}
Here, the behavior of the last sum is different when $d=1$ and $d\geqslant 1$.

When $d=1$, the sum is convergent, so we can choose $k=+\infty$ and the bound becomes
\begin{align*}
	\E[W^{(\lambda)}(\mu,\hat{\mu}_n)]&\leqslant  \frac{C}{\sqrt{n}}\big(\delta+(CK)^{1/2}2^{-k_0/2}\big)\\
	&\leqslant C\sqrt{\frac{K\delta}{n}}\enspace.
\end{align*}

When $d\geqslant 2$, the sum is no longer convergent. For $d=2$, we have
\begin{align*}
	\E[W^{(\lambda)}(\mu,\hat{\mu}_n)]&\leqslant 2^{1-k}+\frac{C}{\sqrt{n}}\big(\delta+Kk\big)\enspace.
\end{align*}
Taking $k$ such that $2^{-k}\asymp K/\sqrt{n}$ yields, as $\delta/\sqrt{n}\leqslant K/\sqrt{n}$,
\begin{align}
	\E[W^{(\lambda)}(\mu,\hat{\mu}_n)]&\leqslant \frac{CK}{\sqrt{n}}\log\bigg(\frac{\sqrt{n}}K\bigg)\enspace. \label{Bd2}
\end{align}

Finally, for any $d\geqslant 3$, we have
\[
\E[W^{(\lambda)}(\mu,\hat{\mu}_n)]\leqslant 2^{1-k}+\frac{C}{\sqrt{n}}\bigg(\delta+(CK)^{d/2}2^{k(d/2-1)}\bigg)\enspace.
\]
Optimizing over $k$ yields the choice $2^k=n^{1/d}/(CK)$ and thus, as $\delta/\sqrt{n}\leqslant K/n^{1/d}$,
\begin{equation}
	\E[W^{(\lambda)}(\mu,\hat{\mu}_n)]\leqslant\frac{CK}{n^{1/d}}\enspace. \label{Bd3}
\end{equation}
\qed


\noindent \textbf{Proof of Theorem~\ref{thm6}.}
Before we prove Theorem~\ref{thm6}, we prove the following two lemmas.

\begin{lemma}\label{lemma3}
	Let $ (\theta_n)_{n\geq1}  $ be a sequence in $ \Theta $ and $ \theta \in \Theta $. Suppose $ \rho_{\Theta}(\theta_n,\theta)  $ $\rightarrow$ 0 implies that $ \mu_{\theta_n} $ convergence weakly to $ \mu_{\theta} $. Then the map $ (\theta,\mu) \mapsto W^{(\lambda)}(\mu_{\theta},\mu) $ is continuous, where $(\theta,\mu) \in  \Theta\times \mathcal{P(X)} $.
\end{lemma}

\begin{proof}
	The result follows directly from Corollary~\ref{cor1}.
\end{proof}

\begin{lemma}\label{lemma4}
	The function $\left(\nu, \mu^{(m)}\right) \mapsto {\rm E} W^{(\lambda)}\left(\nu, \hat{\mu}_m\right)$ is continuous with respect to weak convergence. Furthermore, if $\rho_{\Theta}\left(\theta_n, \theta\right) \rightarrow 0$ implies that $\mu_{\theta_n}^{(m)}$ converges weakly to $\mu_\theta^{(m)}$, then the map $(\nu, \theta) \mapsto$ ${\rm E} W^{(\lambda)}\left(\nu, \hat{\mu}_{\theta, m}\right)$ is continuous.
\end{lemma}

\begin{proof}
	The  proof moves along the same lines as the proof of Lemma A2 in \cite{bernton2019parameter}.
	It is worth noting that $ 0 \leq W^{(\lambda)} \leq 2\lambda $, so we can use dominated convergence theorem instead of Fatou's lemma. Because of the continuity of robust Wasserstein distance,   we have
	\[ {\rm E} W^{(\lambda)}\left(\nu, \hat{\mu}_m\right) = {\rm E} \lim _{k \rightarrow \infty} W^{(\lambda)}\left(\nu_k, \hat{\mu}_{k, m}\right) = \lim _{k \rightarrow \infty} {\rm E} W^{(\lambda)}\left(\nu_k, \hat{\mu}_{k, m}\right).\]
	This implies  $\left(\nu, \theta\right) \mapsto {\rm E} W^{(\lambda)}\left(\nu, \hat{\mu}_{\theta,m}\right)$ is continuous.
\end{proof}

In order to prove Theorem~\ref{thm6}, we introduce the concept of the epi-converge as follows.

\begin{definition}\label{def.EpiConverge}
	We call a sequence of functions $f_n: \Theta \rightarrow \mathbb{R}$  epi-converge to $f: \Theta \rightarrow \mathbb{R}$ if for all $\theta \in \Theta$,
	$$
	\begin{cases}\liminf _{n \rightarrow \infty} f_n\left(\theta_n\right) \geq f(\theta) & \text { for every sequence } \theta_n \rightarrow \theta \\ \limsup _{n \rightarrow \infty} f_n\left(\theta_n\right) \leq f(\theta) & \text { for some sequence } \theta_n \rightarrow \theta\end{cases}
	$$
\end{definition}

For any $\nu \in \mathcal{P}(\mathcal{X})$, the continuity of the map $\theta \mapsto$ $W^{(\lambda)}\left(\nu, \mu_\theta\right)$ follows from Lemma~\ref{lemma3}, via Assumption~\ref{asum2}. Next, by definition of the infimum, the set $B_{\star}(\varepsilon)$ with the $\varepsilon$ of Assumption~\ref{asum3} is non-empty. Moreover, since $\theta \mapsto W^{(\lambda)}\left(\mu_{\star}, \mu_\theta\right)$ is continuous,  the set  $\operatorname{argmin}_{\theta \in \Theta} W^{(\lambda)} \left(\mu_{\star}, \mu_\theta\right) $ is non-empty and the set $B_{\star}(\varepsilon)$ is closed. Then, by Assumption~\ref{asum3}, $B_{\star}(\varepsilon)$ is a compact set.

The next step is to prove the sequence of functions $\theta \mapsto W^{(\lambda)}\left(\hat{\mu}_n(\omega), \mu_\theta\right)$ epi-converges to $\theta \mapsto W^{(\lambda)}\left(\mu_{\star}, \mu_\theta\right)$ by applying  Proposition 7.29 of~\cite{ramponi2020neural}. Then we can obtain the results by  Proposition 7.29 and Theorem 7.31 of ~\cite{ramponi2020neural}. These steps are similar to~\cite{bernton2019parameter} and are hence omitted.\qed

\vspace{4mm}

Moreover, we state 

\begin{theorem}[Measurability of MRWE]\label{thm7}
	Suppose that $\Theta$ is a $\sigma$-compact Borel measurable subset of $\mathbb{R}^{d_\theta}$. Then under Assumption~\ref{asum2}, for any $n \geq 1$ and $\varepsilon>0$, there exists a Borel measurable function $\hat{\theta}_n^{\lambda}: \Omega \rightarrow \Theta$ that satisfies $	\hat{\theta}_n^{\lambda}(\omega) \in \operatorname{argmin}_{\theta \in \Theta} W^{(\lambda)}\left(\hat{\mu}_n(\omega), \mu_\theta\right), $ if  this set is non-empty, otherwise, $\hat{\theta}_n^{\lambda}(\omega) \in \varepsilon\text{-}\operatorname{argmin}_{\theta \in \Theta} W^{(\lambda)}\left(\hat{\mu}_n(\omega), \mu_\theta\right)$.
\end{theorem}

Th.~\ref{thm7} implies that, for any $ n \geq 1 $, there is a  measurable function $\hat{\theta}_n^{\lambda}$ that coincides  (or it is   very close) to the  minimizer of $  W^{(\lambda)}\left(\hat{\mu}_n(\omega), \mu_\theta\right)$. \\

\noindent \textbf{Proofs of Theorems \ref{thm8}, \ref{thm9}, \ref{thm10} and~\ref{thm7}.}
The proofs of these theorems follow by moving along the same lines as the proofs of Theorems  2.4, 2.5, 2.6 and 2.2 in~\cite{bernton2019parameter}. \qed


\subsection{The choice of  $\lambda $ } \label{Sec: gen_choice}

\subsubsection{Some considerations based on RWGAN} \label{Sec.choice}


Here we focus on the applications of RWGAN and MERWE and illustrate, via Monte Carlo experiments,  how $ \lambda $ would affect the results. 
First, we recall the  experiment in \S~\ref{Sec.SimWesti}: MERWE of location for sum of log-normal.  We keep other conditions unchanged, and just change the parameter $ \lambda $ to see how the estimated value changes.  We set $ n=200 $, $ \varepsilon=0.1 $ and $ \eta=4 $ and other parameters are still set as in \S~\ref{Sec.SimWesti}. The result is shown in Figure~\ref{fig:choiceofmerwe}. We can observe that the MSE is large when $ \lambda $ is small. The reason for this is that we truncate the cost matrix prematurely at this time, making robust Wasserstein distance unable to distinguish the differences between distributions. When $ \lambda $ is large, there is negligible difference between Wasserstein distance and robust Wasserstein distance, resulting in almost the same MSE estimated by MERWE and MEWE. When $ \lambda $ is moderate, within a large range (e.g., from $ \exp(2) $ to $ \exp(5.5) $), the MERWE outperforms the MEWE.

\begin{figure}[http]
	\centering
	\includegraphics[scale=0.2]{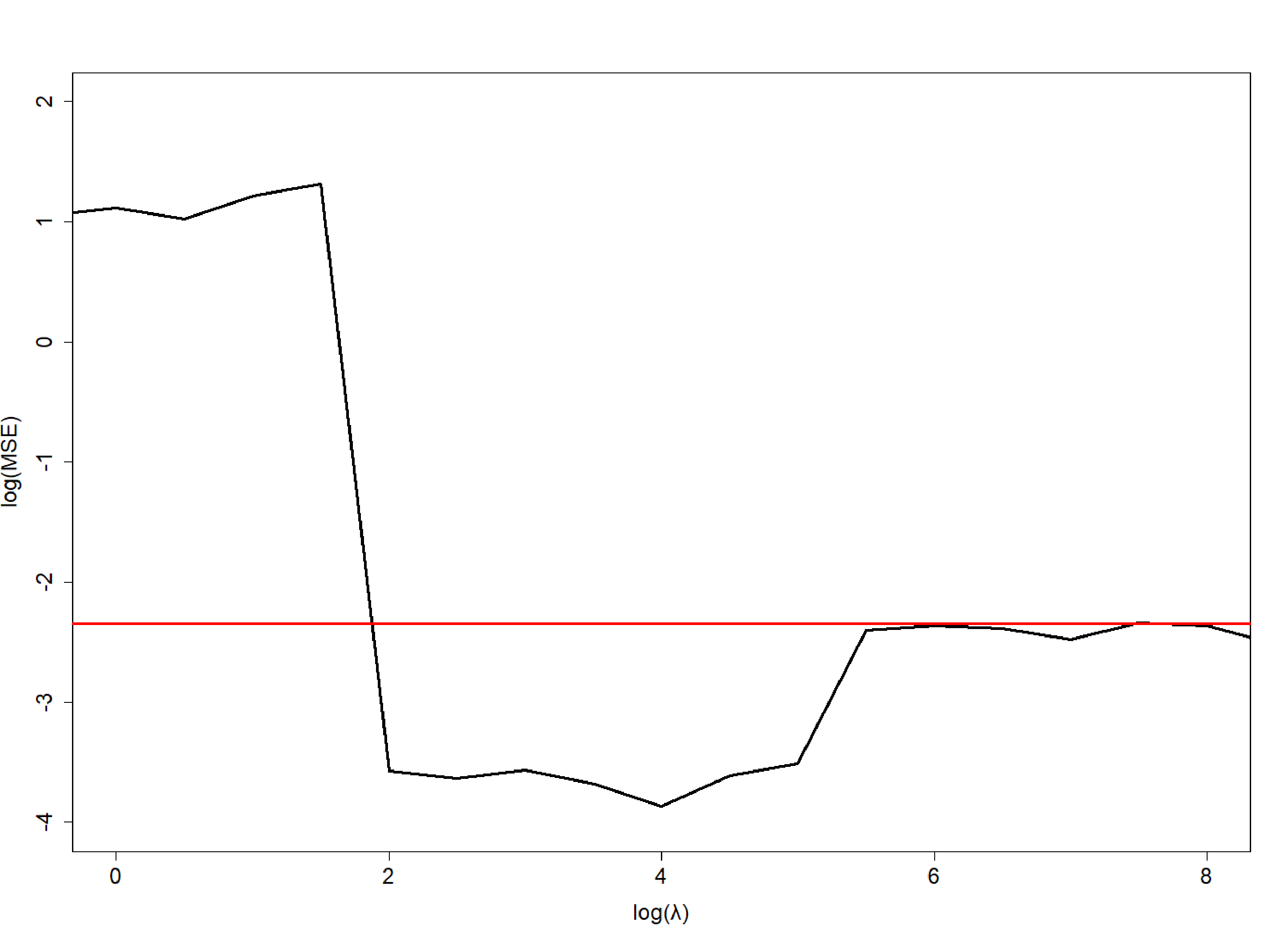}
	\caption {$\log(\rm MSE)$ of MERWE based on 1000 times experiment for different values of $\log(\lambda)$. The horizontal (red) line represents $ \log(\rm MSE) $ of MEWE.}
	\label{fig:choiceofmerwe}
\end{figure}

We consider a novel way to study parameter selection based on our Theorem~\ref{thm3},  which implies that modifying the constraint on $ \mathrm{range}(\psi) $ is equivalent to changing the parameter $ \lambda $ in equation~\eqref{equ8}. This reminds us that we can use RWGAN-1 to study how to select $\lambda$. 


We consider the same synthetic data  as in \S~\ref{app.rwgan} and set $n=1000, \varepsilon=0.1$ and $\eta=2 $. To quantify training quality of RWGAN-1, we take the   Wasserstein distance of order 1 between the clean data and data generated from RWGAN-1. Specifically,  we draw 1000 samples from model~\eqref{true model} and generate 1000 samples from the model trained by RWGAN-1 with different parameter $\lambda $. Then, we calculate the  empirical  Wasserstein distance of order 1 between them. The plot of the empirical  Wasserstein distance against $\log(\lambda)$ is shown in Figure~\ref{fig:W}. The Wasserstein distance has a trend of decreasing first and increasing later on: it reaches the minimum at $\log(\lambda) = -2$. Also, we  notice that over a large range of $\lambda$ values (i.e. from $\exp(-4)$ to $\exp(0)$), the Wasserstein distance is small: this illustrates that  RWGAN-1 has a good outlier detection ability for a choice of the penalization parameter within this range. 

\begin{figure}[http]
	\centering
	\includegraphics[scale=0.22]{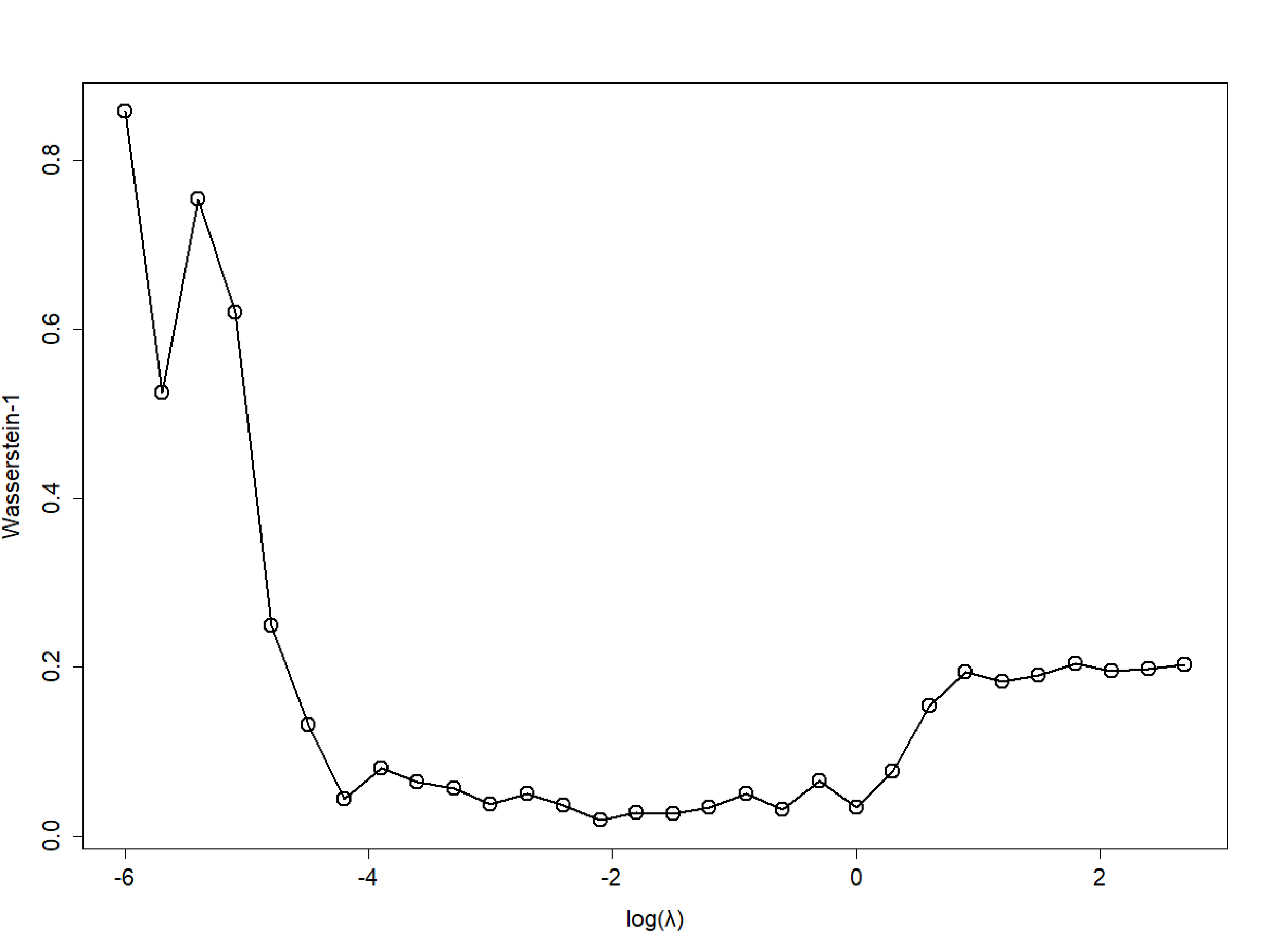}
	\caption {Plot of the empirical Wasserstein distance of order 1	 between the clean data and data generated from RWGAN-1 for different values of $\log(\lambda)$.}
	\label{fig:W}
\end{figure}

Heuristically, a large $\lambda$ means the truncated cost function is just slightly different from the non-truncated one. Therefore, RWGAN-1 will be greatly affected by outliers.  When $\lambda$ is small, the cost function is severely truncated, which means that we are modifying too much of the distribution. This deviates from our main purpose of GAN (i.e., $ \mathrm{P_r} \approx \mathrm{P_{\theta}}$). Hence, the distribution of final outputs generated from RWGAN-1 becomes very different from the reference one. 

Both experiments suggest that to achieve good performance of ROBOT, one could choose a moderate $\lambda$ within a large range. Actually, the range of a ``good" $\lambda$ depends on the size and proportion of outliers: When the size and proportion are close to zero, both moderate and large $\lambda$ will work well, since only slight or no truncation of the cost function is needed; when the size and proportion are large, a large truncation is needed, so this range becomes smaller. Generally, we prefer to choose a slightly larger $ \lambda $, as  we found in the experiment that when we use a larger $ \lambda $, the performance of ROBOT is at least no worse than that using OT.


\color{black}
\subsubsection{Selection based on concentration inequality}\label{App.lambda}

\textbf{Methodology.} In the Monte Carlo setting  of Section~\ref{Sec.SimWesti}, 
Figure \ref{fig:choiceofmerwe} in Appendix~\ref{Sec.choice} illustrates that MERWE has good performance (in terms of MSE) when $\lambda$ is within the range $[\exp(2), \exp(5)]$. The estimation of the MSE represents, in principle,  a valuable criterion that one can apply to select $\lambda$. However, while this idea is perfectly fine in the case of synthetic data, its applicability seems to be hard in real data examples, where the true underling parameter is not known. To cope with this issue, \cite{WJ05} propose to use a pilot estimators which can yield an estimated MSE. The same argument can be considered also for the our estimator, but this entails the need for selecting a pilot estimator: a task that cannot be easy for some complex models and for the models considered in this paper. Therefore, we propose and investigate the use of an alternative approach, which yields a range of values of $\lambda$ similar to the one obtained computing the MSE but it does not need to select such a pilot estimator.

To start with, let us consider Th.~\ref{thm:NewCI}. At the reference model, namely in the absence of contamination  (i.e. $ \vert O \vert = 0$ and $\vert I \vert =n$), the mean concentration in \eqref{IneqI} implies that $\left \vert W^{\lambda}(\hat\mu_n,\mu^\ast) -  {\rm E}[W^{\lambda}(\hat\mu_n,\mu^\ast)] \right\vert$ takes on values larger than  the threshold 
\beq
\sigma\sqrt{\frac{2t}{n}}+\frac{4\lambda t}{n} \label{Eq.Ozero} 
\eeq
with probability bounded by $2\exp(-t)$. In the presence of contamination, from (\ref{MattCI}) we have that the probability that the deviations of 
 $W^{\lambda}(\hat\mu_n,\mu^\ast) $ from its mean are larger than the threshold
\beq
\sigma\sqrt{\frac{|I|}n}\sqrt{\frac{2t}{n}}+\frac{4\lambda t}{n}+\frac{4\lambda|O|}n \label{Eq.Opos}
\eeq
is still bounded by $2\exp(-t)$.

Then,  to select $\lambda$  we propose to set up the following criterion, which compares the changes in   (\ref{Eq.Ozero}) and (\ref{Eq.Opos}) due to different values of $\lambda$. 
More in detail, for a fixed $ t>0 $, we  assume a contamination level $\tau$ which represents a guess that the statistician does about the actual level of data contamination. Then we compare the thresholds  at different values of $\lambda$. To this end, we compute the ratio between (\ref{Eq.Ozero}) and (\ref{Eq.Opos}) and define the quotient
\begin{equation} 
\mathcal{Q}(n,\tau , \lambda,t):=\frac{\sigma\sqrt{\frac{2t}{n}}+\frac{4\lambda t}{n}}{\sigma\sqrt{1-\tau}\sqrt{\frac{2t}{n}}+\frac{4\lambda t}{n}+ 4 \tau \lambda }.
\label{Eq. Q}
\end{equation} 
Assuming that  $\sigma$  is a twice differentiable function of $\lambda$, one can check
that $\mathcal{Q}$ decreases monotonically in $\lambda$. Indeed, the partial derivative with respect to $\lambda$ of $\mathcal{Q}(n,\tau , \lambda,t)$ is 
\begin{equation} 
\frac{\partial \mathcal{Q}}{\partial \lambda} = C_{n,\tau,t}(\lambda) \left(\lambda \sigma^{\prime}(\lambda) - \sigma(\lambda) \right), \label{Eq. derQ}
\end{equation} 
where 
$$C_{n,\tau,t}(\lambda) =  \frac{4 \sqrt{2} t \sigma(\lambda)\left((n+t) \sqrt{\frac{\tau t \sigma(\lambda)^2}{n}}-\tau\left(t \sqrt{\frac{t \sigma(\lambda)^2}{n}}+n \sqrt{\frac{\tau t \sigma(\lambda)^2}{n}}\right)\right)}{\sqrt{\frac{t \sigma(\lambda)^2}{n}} \sqrt{\frac{\tau t \sigma(\lambda)^2}{n}}\left(4 n(1-\tau) \lambda+4 t \lambda+\sqrt{2} n \sqrt{\frac{\tau t \sigma(\lambda)^2}{n}}\right)^2}$$ is positive when $\lambda > 0$.  
Moreover, let us denote by $\sigma^{\prime}$ and $\sigma^{\prime \prime}$ the first and second derivative of $\sigma$ with respect to $\lambda$. 
So the monotonicity of $\mathcal{Q}$ depends on $\gamma(\lambda) := \lambda \sigma^{\prime}(\lambda) - \sigma(\lambda)$. 
Now, notice  that $\gamma=0$ when $\lambda= 0$ and that  $\gamma^{\prime} = \lambda \sigma^{\prime \prime}(\lambda)$. 
Moreover, from the definition of $\sigma$, $\sigma$ is increasing with respect to $\lambda$ and $\sigma^{\prime \prime} < 0$ when $\lambda > 0$. 
So $\gamma^{\prime} <0$ when $\lambda >0$  and $\gamma < 0$ when $\lambda >0$. This means $\mathcal{Q}$ decreases monotonically.

Next we explain how $ {\partial \mathcal{Q}}/{\partial \lambda}$ can be applied to define a data driven criterion for the selection of $\lambda$.
We remark that for most models, the computation of $\mathcal{Q}$ and of $ {\partial \mathcal{Q}}/{\partial \lambda}$  need to be performed numerically. To illustrate this aspect and to elaborate further on the  selection of $\lambda$, we consider the  same Monte Carlo setting  of Section~\ref{Sec.SimWesti}, where $ \varepsilon = {\vert O \vert}/{n} $ is the actual  contamination level. In Appendix~\ref{Sec.choice}, we found through numerical experiments that MERWE has good performance for $\lambda$ within the $[\exp(2), \exp(5)]$. Hence, in our numerical analysis, we consider this range as a benchmark for evaluating the  $\lambda$ 
selection procedure based on $\mathcal{Q}(n,\tau , \lambda,t)$.  Moreover, we need to  specify a value for the contamination level and fix  $t$. In our numerical analysis we set $\tau = 0.1$ and  $t=1$. We remark that the specified $\tau$ coincides with the percentage of outliers $\varepsilon$: later on  we discuss the case when $\tau\neq \varepsilon$. {Moreover, we point out that different values of $t$ simply imply that the proposed selection criterion concerns the stability of other quantiles of the distribution of $W^{\lambda}(\hat\mu_n,\mu^\ast)$}.

 In Figure~\ref{fig:q}, we display the plot of $\mathcal{Q}$ with respect to $\lambda$: $\mathcal{Q}$ decreases monotonically, with slope which decreases as $\lambda \to \infty$. The plot suggests that, when $\lambda$ is small, $\mathcal{Q}$  decreases fast for small increments in the value of $\lambda$. This is due to the fact that, for $\lambda$ small, a small change of $\lambda$  has a large influence on the transport cost $c^{(\lambda)}$ and hence on the values of $W^{(\lambda)}$. On the contrary, when  $\lambda$ is sufficiently large, larger values of the trimming parameter do not have large influence on the transport cost $c^{(\lambda)}$. As a result, looking at the (estimated) slope of (estimated) $\mathcal{Q}$ allows to identify a range of $\lambda$ values which avoids 
the small values (for which $W^{(\lambda)}$ trims too many observations and it is maximally concentrated about its mean) and large values  (for which $W^{(\lambda)}$ trims too few observations and it is minimally concentrated about its mean).  This is the same information available in Figure \ref{fig:choiceofmerwe}, but differently from the computation which yields the MSE in that figure, our method does not entail either the need for knowing the true value of the model parameter or the need for a pilot estimator.\\

\textbf{Implementation.}  Operationally, to obtain the desired range of $\lambda$ values, 
one needs to create a sequence of $\{\lambda_k\}$  with $$0<\lambda_1 < \lambda_2 < \ldots < \lambda_N,\, N>1 .$$ Then, one computes
the sequence of {\it absolute slopes} $$\mathrm{AS}{(n, \tau, \lambda_k, t)} = \frac{\vert {\mathcal{Q}}{(n, \tau, \lambda_{k+1}, t)} - {\mathcal{Q}}{(n, \tau, \lambda_{k}, t)}  \vert}{\lambda_{k+1}-\lambda_{k}}, k \geq 1,$$ 
at different values of $\lambda_k$. We remark that  to compute $\mathrm{AS}$, we need to estimate $ \sigma^2$. To this end, a robust estimator of the variance should be applied. In our numerical exercises, we find numerically convenient to fix $ X_i^{\prime}$ to $\bar{X}_{\mathrm{rob}}$ which represents the geometric median for real-valued variable $X$. Then, we replace $$\rm{E}\left[\min \left(d^2\left(X_i, X_i^{\prime}\right), (2\lambda)^2\right)\right]$$ by
$$({1}/{n}) \sum_{i=1}^n \min \left\{d^2\left(X_i, \bar{X}_{\mathrm{rob}}\right), (2\lambda)^2\right\}.$$  

Once the different values of AV are available,   one needs to make use of them to select $\lambda$. To this end, we propose to follow the idea of Algorithm 1 in \cite{la2015robust}  and 
choose
$$
\lambda^{\star} :=  \min \left\lbrace  \lambda_{k}: {\mathrm{AS}}{(n, \tau, \lambda_i, t)} \leq \iota, \text{for all } \lambda_{i} \geq  \lambda_{k} \right\rbrace,
$$ where $\iota$  is a user-specified value. According to our numerical experience it is sensible to specify values of  $ \iota $ which are close to zero: they typically define estimators having a performance that remains stable across different Monte Carlo settings, while avoiding to trim too many observations.  For the sake of illustration, in Figure \ref{fig:as} we display the estimated  AS related to $\mathcal{Q}$ as in Figure \ref{fig:q}: we see that its maximum value is around 0.05. We propose to set $\iota = 0.001$, a value which is slightly larger than zero (namely, the value that  $\mathrm{AS}$ has when $\lambda$ is large). This yields  
$\lambda^{\star} \approx 26$ ($\ln \lambda  \approx 3.25$), which is within the range $[\exp(2), \exp(5)]$ where MERWE produces small MSE as displayed in Figure~\ref{fig:choiceofmerwe}. A reassuring aspect emerges from Figure \ref{fig:as}: in case one decides to set slightly larger values of $\iota$ (e.g. values between $0.008$ and $0.001$), the selection procedure is still able to provide a value of $\lambda^{\star} $ within the benchmark $[\exp(2), \exp(5)]$. 

Essentially, we conclude that any value of $\iota$ which is greater than zero and smaller than the max of AS selects a  $\lambda^{\star} $ which in turns yields a good performance (in terms of MSE) of our estimator. \\

\begin{figure}[http]
	\centering
	\includegraphics[width=0.75\textwidth]{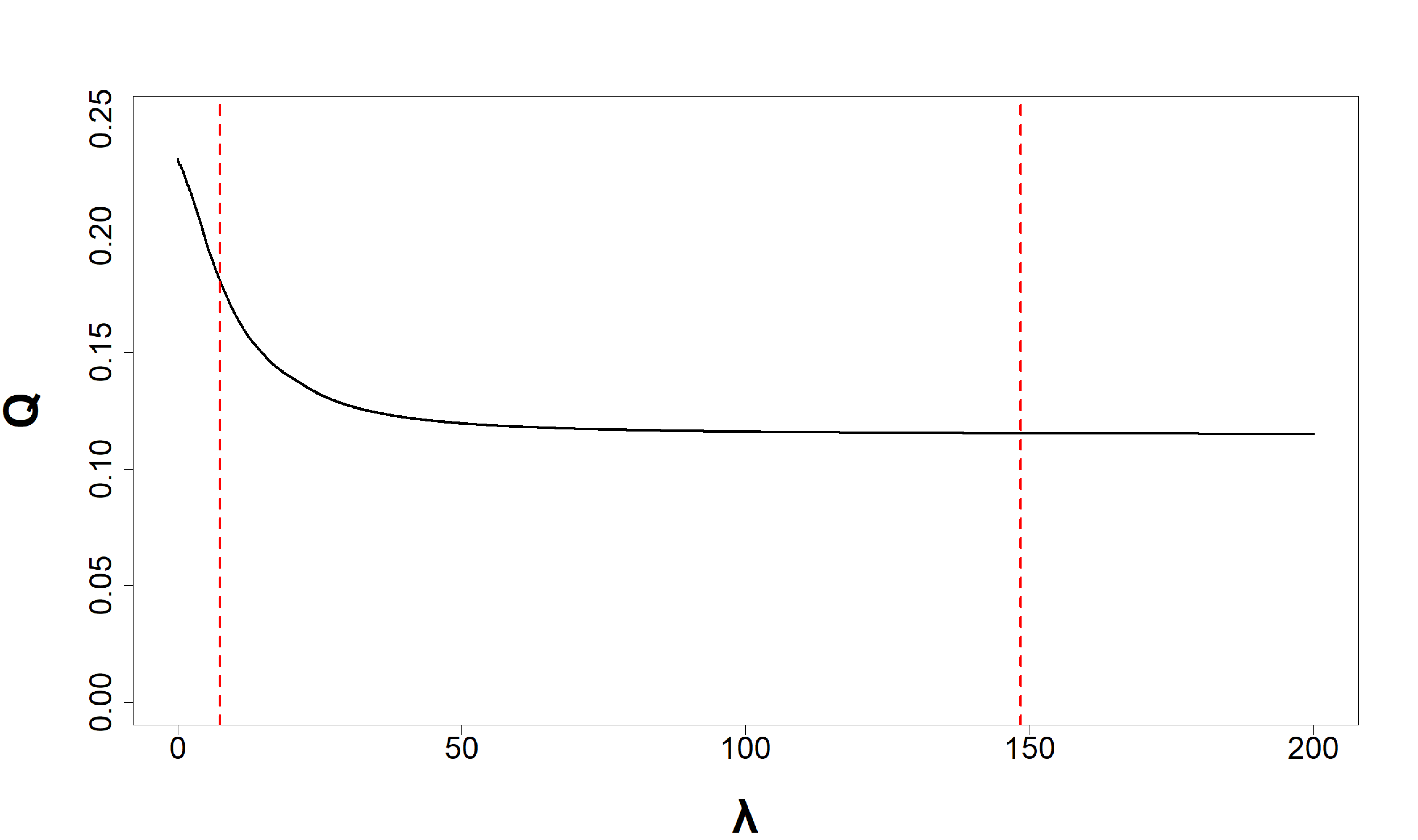}
	\caption{The curve (black) of $\mathcal{Q}$ with respect to $\lambda$. The coordinates of the vertical dashed lines (red) are $\exp(2)$ and $\exp(5)$, respectively.  }
	\label{fig:q}
\end{figure}

\begin{figure}[http]
	\centering
	\includegraphics[width=0.785\textwidth]{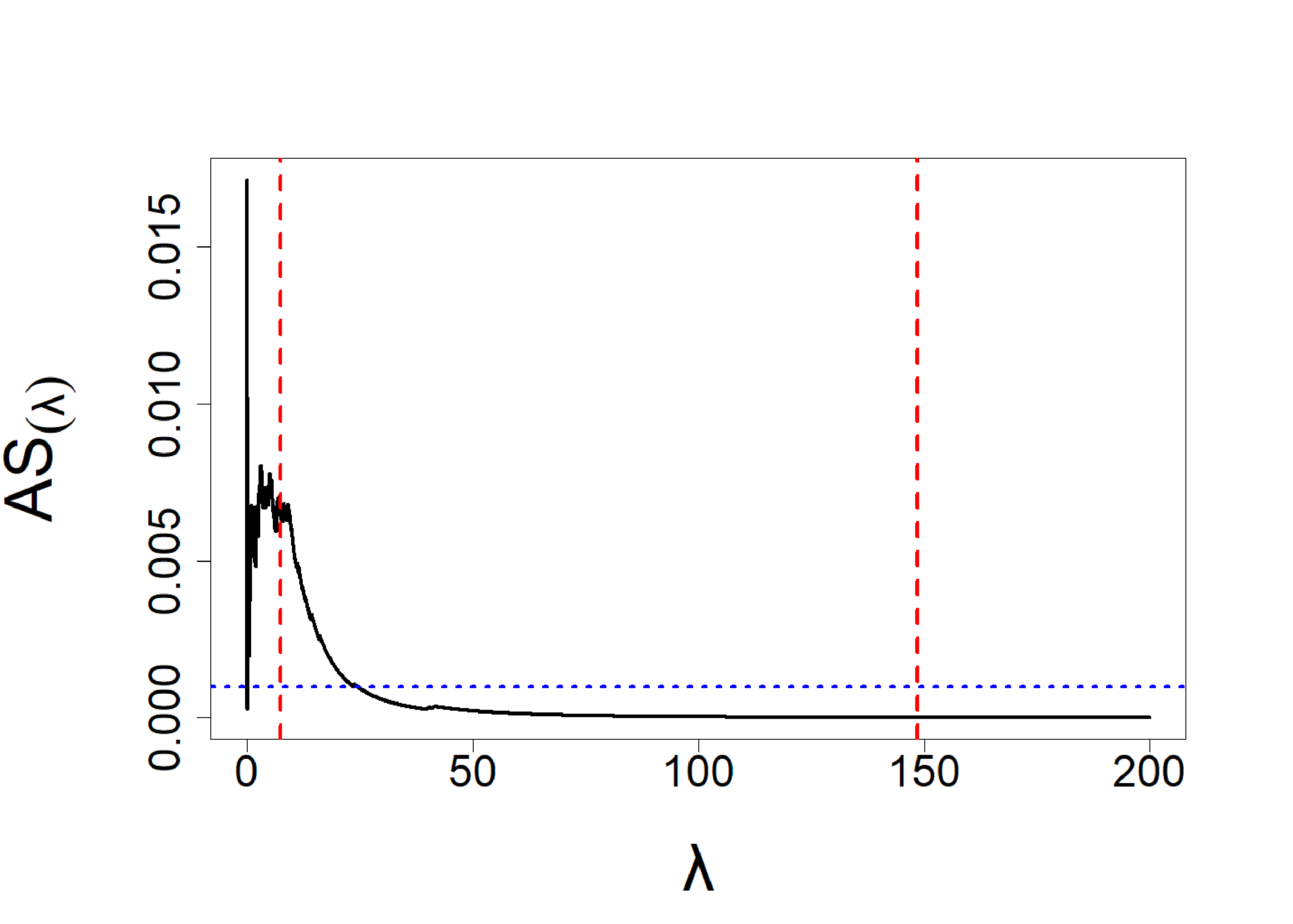}
	\caption{Plot of the $\mathrm{AS}$ curve (black) as a function of $\lambda$.
		The coordinates of the vertical dashed lines (red) are $\exp(2)$ and $\exp(5)$, respectively, and the  coordinate of horizontal dashed line (blue) is $0.001$. }
	\label{fig:as}
\end{figure}

\textbf{Remark.} We mention three interesting methodological aspects related to this selection procedure. First, the proposed criterion is based on the computation of $\mathrm{AS}$ for different $\lambda_k$  and it does not entail the need for computing the corresponding MDE $\hat\theta^\lambda_n$. This aspect makes our criterion easy-to-implement and fast. Second, one may wonder what is the distribution of $\lambda^\ast$ when many samples from the same $O\cup I$ setting are available. Put it in another way, given many samples and selecting in each sample $\lambda$ by the aforementioned procedure how does $\lambda^\ast$ change? Third, an interesting question related to the proposed criterion is: what happens to $\lambda^\ast$ if the statistician specifies a contamination level $\tau$ that is different from the true contamination level $\varepsilon$?
As far as the last two questions are concerned, we remark that Figure \ref{fig:as} is based on one sample and it cannot answer these questions. Thus, to gain deeper understanding, we keep the same setting as in the Monte Carlo experiment (which  yielded  Figure \ref{fig:as}) and we simulate 1000 samples, specifying $\tau =0.05,0.1,0.2$ and selecting $\lambda^{\star} $ by the aforementioned procedure. In Figure~\ref{fig:density} we display the kernel density estimates of the selected values. 

The plots illustrate that, for each $\tau$ value, the kernel density estimate  is rather concentrated: for instance, if $\tau=0.1$ the support of the estimated density goes from about 3 to about 3.4. Moreover, looking at the estimated densities when $\tau=0.005$ (underestimation of the number of outliers) and $\tau=0.2$ (overestimation of the number of outliers), we see that, even if $\tau \neq \varepsilon $, the selection criterion still yields values of $\lambda$ which are within the  benchmark range $[\exp(2), \exp(5)]$. 
{Similar considerations hold for other user-specified values of $t$ in the thresholds  (\ref{Eq.Ozero}) and (\ref{Eq.Opos}).}

\begin{figure}[http]
	\centering
	\includegraphics[width=0.89\textwidth]{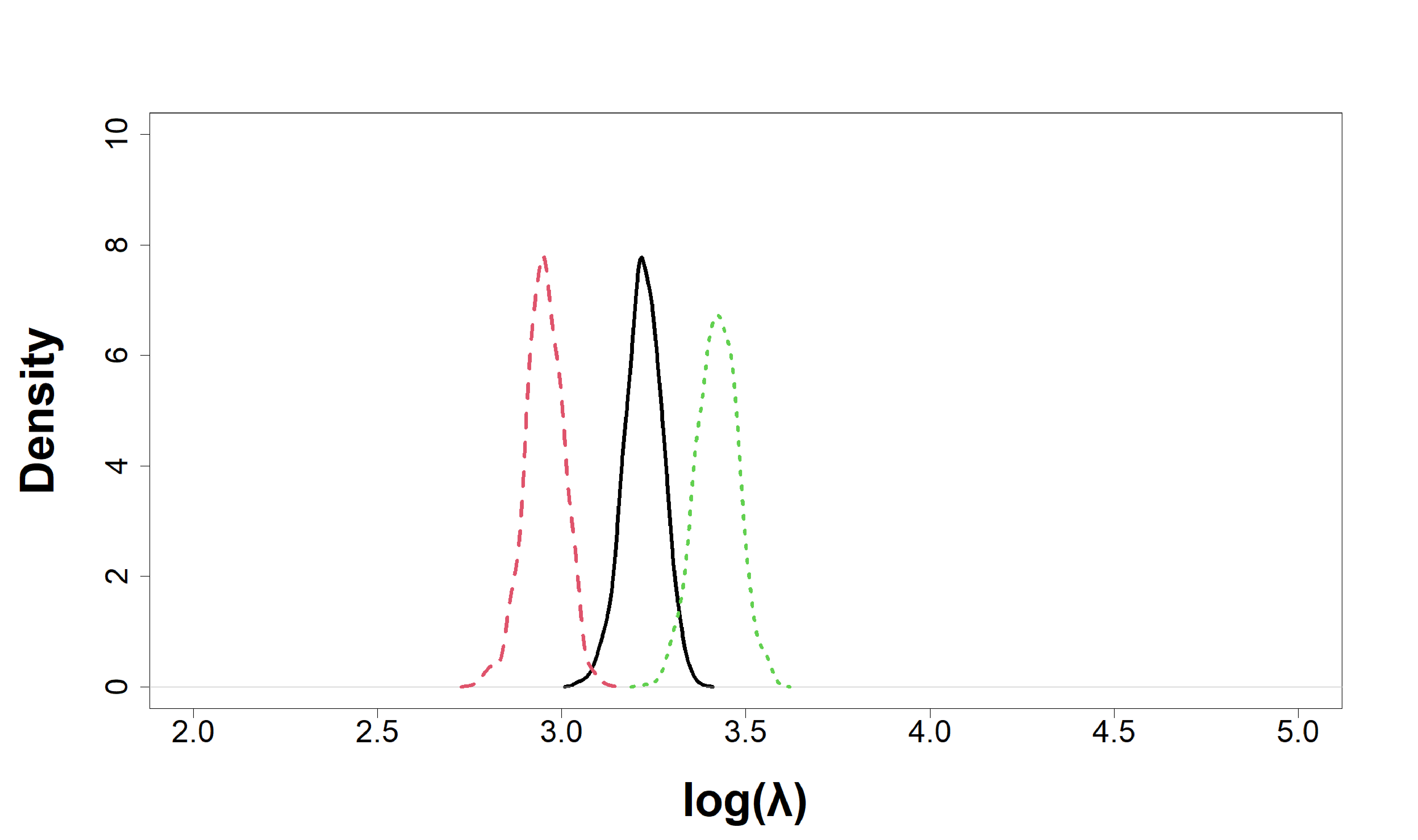}
	\caption{ Dashed line (red), solid line (black), dotted line (green) represent the Gaussian kernel density functions (bandwidth = $0.0167$) of $\log(\lambda^{\star})$ for $\tau = 0.05, 0.1, 0.2$, respectively. }
	\label{fig:density}
\end{figure}
\color{black}

\subsection{Additional numerical exercises}\label{App.MCresults}

This section, as a compliment to Section~\ref{Sec:Applications}, provides more details and additional  numerical results of a few applications of ROBOT in statistical estimation and machine learning tasks.

\subsubsection{ROBOT-based estimation via outlier detection (pre-processing)}\label{App.ROBOT}

\textbf{Methodology.} Many routinely-applied methods of model estimation can be significantly affected by outliers.  Here we show how  ROBOT can be applied for robust model estimation. 

Recall that formulation~\eqref{equ7} is devised in order to  detect and remove outliers. Then, one may exploit this feature and estimate in a robust way the model parameter.   Hereunder, we explain in detail the outlier detection procedure.

Before delving into the details of the procedure, we give a cautionary remark. Although theoretically possible and numerically effective, the use of ROBOT for pre-processing does not give the same statistical guarantees for the parameter estimations as the ones yielded by  MRWE and MERWE (see \S \ref{Robust Wasserstein estimator}).

 With this regard we mention two key inference issues. (i) There is the risk of having a masking effect: one outlier may hide another outlier, thus  if one is
removed another one may appear. As a result, it becomes unclear and really subjective when to stop the pre-processing procedure. (ii) Inference obtained via M-estimation on the cleaned sample is conditional to the outlier detection: the distribution of the resulting estimates is unknown. For further discussion see  \cite[Section 4.3]{maronna2006robust}. 

Because of these aspects, we prefer the use of MERWE, which does not need any (subjective) pre-processing and it is an estimator defined by an automatic procedure, which is, by design, resistant to outliers and whose statistical properties can be studied.

Assume we want to estimate data distribution through some parametric model $g_{\theta}$ indexed by a parameter $\theta$. For the sake of exposition, let us think of a regression model, where $\theta$ characterizes the conditional mean to which we add an innovation term. 

Now, suppose the innovations are i.i.d. copies of a random variable having normal distribution. The following ROBOT-based procedure can be applied  to eliminate the effect of outliers: (i) estimate $\theta$ based on the observations and compute the variance $\tilde{\sigma}^2$ of the residual term; (ii) solve the ROBOT problem between the residual term and a normal distribution whose variance is $\tilde{\sigma}^2 $  to detect the outliers; (iii) remove the detected outliers and update the estimate of $\theta$; (iv) calculate the variance $\tilde{\sigma}^2$ of the updated residuals; (v) repeat (ii) and (iii) until $\theta$ converges or the number of iterations is reached. More details are given in Algorithm \ref{alg:robust estimation}. For convenience, we still denote the output of Algorithm \ref{alg:robust estimation}  as $g_{\theta}$.


\begin{algorithm}[H]
	\SetAlgoLined
	\KwData{ observations $\{(X_i,Y_i)\}_{i=1}^{n} $, robust regularization parameter $\lambda$, number of iterations $M$, number of generated sample $n$, parametric model $ g_{\theta} $ (indexed by a parameter $\theta$)}
	\KwResult{Model $g_{\theta}$ }
	estimate $\theta$ based on the observations and compute the residuals and their sample standard deviation $\tilde{\sigma}$\;
	\While{$\ell \leq M $}{
		generate samples $ \tilde{\mathbf{Z}}^{(n)} =\{\tilde{Z}_j\}_{j=1}^{n}$  from $\mathcal{N}(0,\tilde{\sigma}^2)$, and let  $\hat{\mathbf{Z}}^{(n)} = \left\lbrace \hat{Z_i}; 1\leq i \leq n \right\rbrace$, where $\hat{Z_i}=Y_i-g_{\theta}(X_i)$\;
		calculate the cost matrix $\mathbf{C} := (C_{ij})_{1\leq i,j\leq n}$ between $\hat{\mathbf{Z}}^{(n)} $ and $ \tilde{\mathbf{Z}}^{(n)}$\;
		collect all the indices $\mathcal{I}=\{(i,j): C_{ij} \geq 2 \lambda\}$, and let $ C^{(\lambda)}_{ij}=2\lambda $ for $(i,j) \in \mathcal{I}$ and $ C^{(\lambda)}_{ij} =C_{ij}$ otherwise\;
		calculate the transport matrix $ \mathbf{\Pi}^{(n,n)} $ between $ \hat{\mathbf{Z}}^{(n)}$ and $\tilde{\mathbf{Z}}^{(n)}$ based on the  modified cost matrix $\mathbf{C}^{(\lambda)} := (C^{(\lambda)}_{ij})_{1\leq i,j\leq n}$\;
		set $s(i)=-\sum_{j=1}^{n} \mathbf{\Pi}^{(n,n)}(i, j) \mathbbm{1}_{(i, j) \in \mathcal{I}}$\;
		find $\mathcal{H}$, the set of all the indices where $s(i)+1/n = 0$\;
		use $\{(X_i,Y_i)\}_{i \notin \mathcal{H}}$ to update  $ \theta $ and the residuals\;
		update the sample standard deviation $\tilde{\sigma}$ of the residuals\;
		$\ell \leftarrow \ell+1$\;
	}
	\caption{ROBOT ESTIMATION VIA PRE-PROCESSING}\label{alg:robust estimation}
\end{algorithm}

To illustrate the ROBOT-based estimation procedure, consider the following linear regression model
\begin{equation}\label{eq.LinearReg}
	\begin{array}{cc}
		& X \sim \mathrm{U}(0,10), \\
		& Y =\alpha X + \beta + Z, \  Z \sim  \mathcal{N}(0,\sigma^2).
	\end{array}
\end{equation}
The observations  are contaminated by additive outliers and are generated from the model 
$$\begin{array}{cc}
	& X_i\n \sim \mathrm{U}(0,10), \\
	& Y_i\n =\alpha X_i\n + \beta + Z_i\n + \mathbbm{1}_{(i=h)}  \epsilon_i\n,\\
	& Z_i\n \sim \mathcal{N}(0,\sigma^2) \text{\: and \:}  \epsilon_i\n \sim\mathcal{N}(\eta+ X_i\n, 1),
\end{array}$$
with $i = 1, 2, \ldots, n$, where $h$ is the location of the additive outliers, $\eta$ is a  parameter to control the size of contamination, and $\mathbbm{1}_{(.)}$ denotes the indicator function. We consider the case that outliers are added at probability $ \varepsilon $.

In this toy experiment, we set  $ n=1000 $, $\sigma = 1 $, $\alpha =1$ and $\beta =1$, and we change the value of $\eta$ and the proportion $\varepsilon$ of abnormal points to all points  to investigate the impact of the size and proportion of the outliers on different methodologies. Each experiment is repeated 100 times. 

We combine the ordinary least square  (OLS) with ROBOT (we set $ \lambda=1$) through Algorithm \ref{alg:robust estimation}  to get a robust estimation of $\alpha$ and $\beta$, and we will compare this robust approach with the OLS.
Boxplots of slope estimates and intercept estimates based on the OLS and ROBOT, for different values $\eta$ and $\varepsilon$, are demonstrated in Figures \ref{fig:a} and \ref{fig:b}, respectively.


\begin{figure}[t!]
	\centering
	\begin{subfigure}[(a)]{0.45\textwidth}
		\centering
		\includegraphics[width=\textwidth]{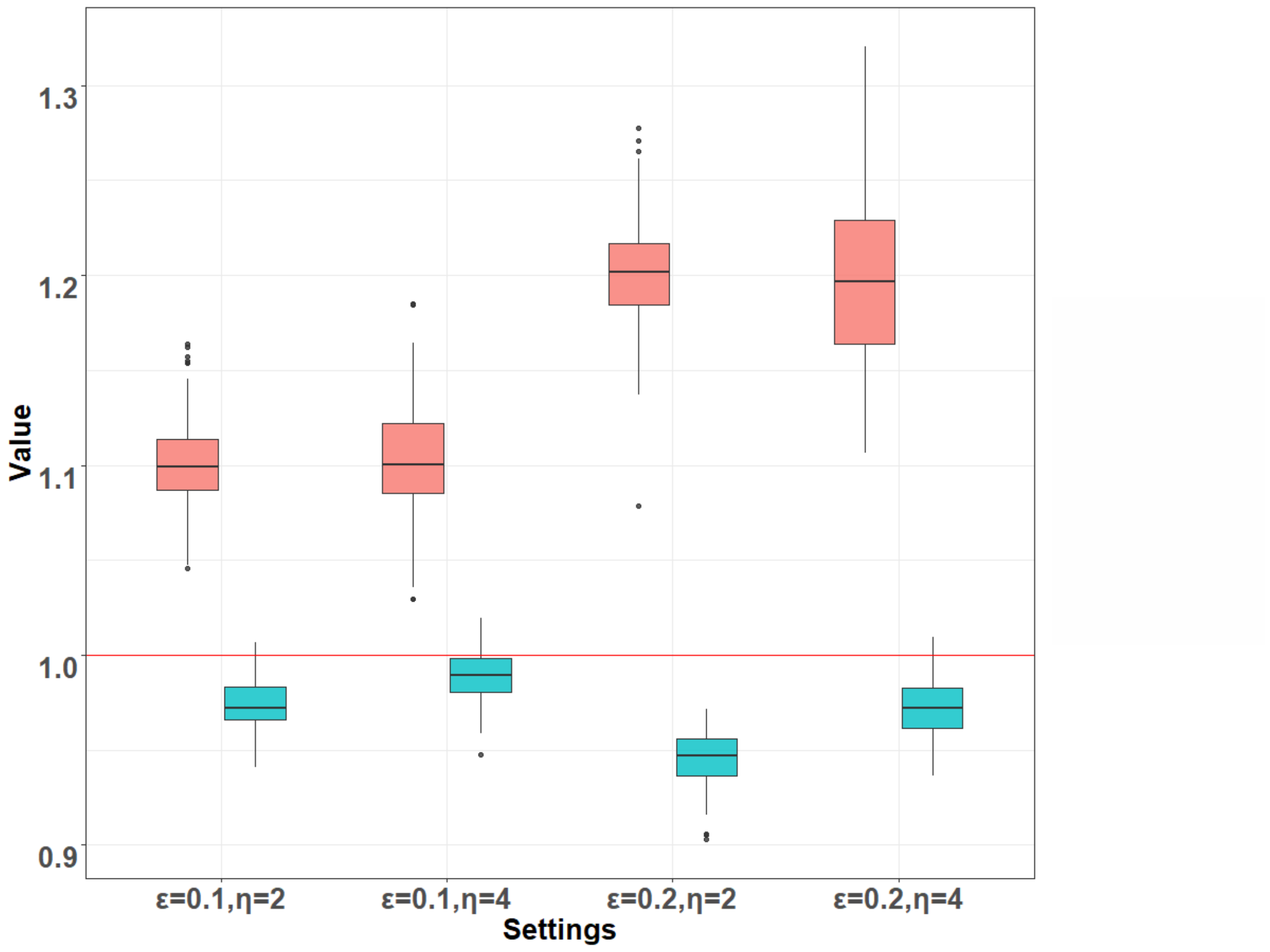}
		\caption{boxplots of slope ($\alpha$) estimates in different situation }
		\label{fig:a}
	\end{subfigure}
	\begin{subfigure}[(b)]{0.45\textwidth}
		\centering
		\includegraphics[width=\textwidth]{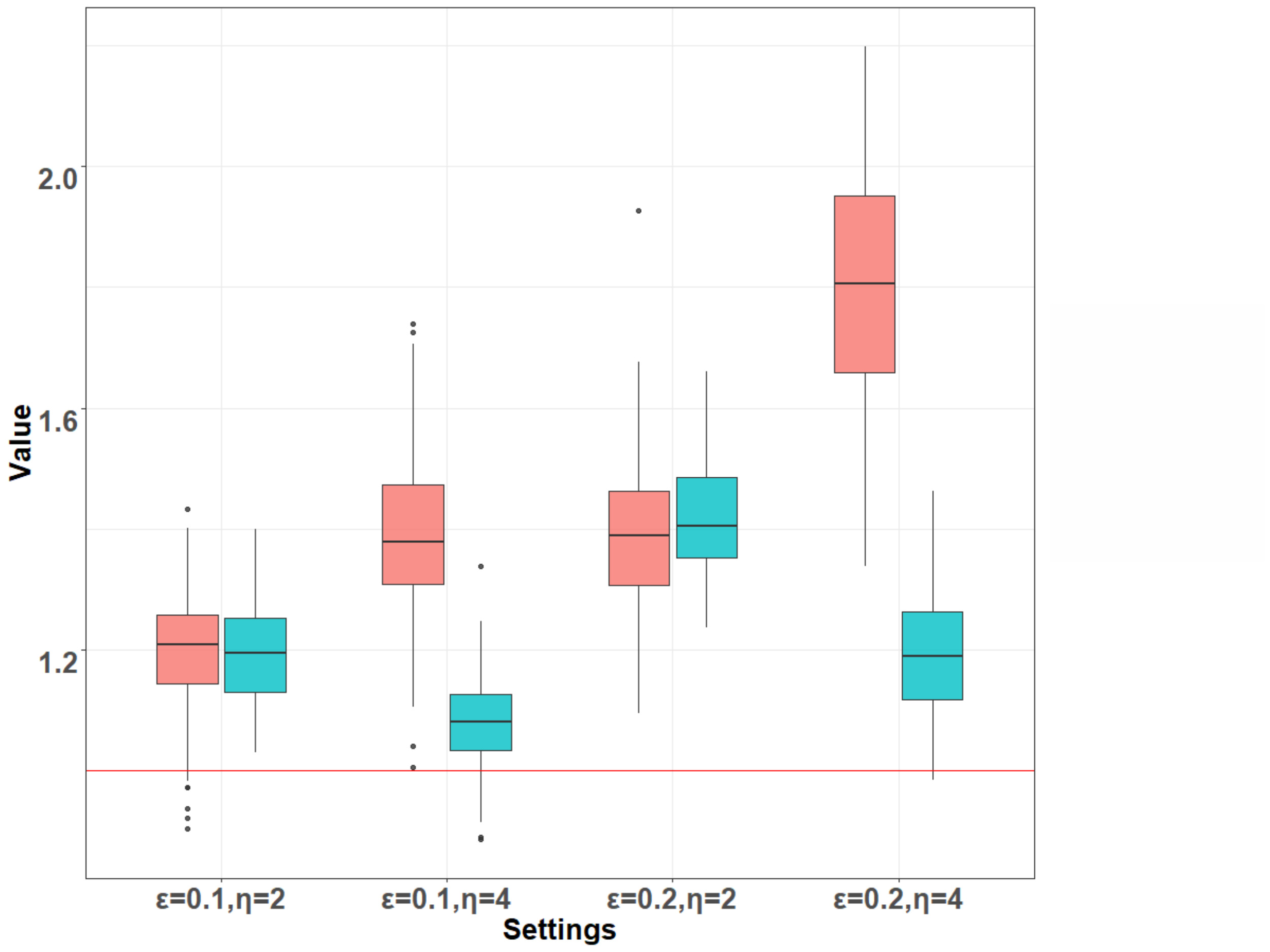}
		\caption{ boxplots of intercept ($\beta$) estimates in different situation}
		\label{fig:b}
	\end{subfigure}
	\caption{We show the estimated effect of parameters through boxplots. The red lines represent  the true values of the model parameters.}
	\label{boxplot}
\end{figure}


For the estimation of $\alpha$, Figure \ref{fig:a} shows that the OLS is greatly affected by the outliers: both the bias and variance increase rapidly with the size and proportion of the outliers. The ROBOT, on the contrary, shows great robustness against the outliers, with much smaller bias and variance. For the estimation of  $\beta$ (Figure \ref{fig:b}),
ROBOT and OLS have similar performance for $ \eta=2 $, but ROBOT outperforms OLS when the  contamination size  is large. 

%

%
%
	%


The purpose of our construction of outlier robust estimation is to let the trained model close to the uncontaminated model~\eqref{eq.LinearReg}, hence it can fit better the data generated from~\eqref{eq.LinearReg}. We, therefore, consider calculating mean square error (MSE)  with respect to the data generated from~\eqref{eq.LinearReg}, in order to compare the performance of the OLS and ROBOT. 

In simulation, we set  $ n=1000 $, $\sigma = 1 $, $\alpha =1$, $\beta =1$ and change the contamination size $\eta$. Each experiment is repeated $100$ times. 
In Figures \ref{fig:mse1} and \ref{fig:mse2}, we plot the averaged MSE of the OLS and ROBOT against $\eta$ for $\varepsilon=0.1$ and $\varepsilon=0.2$, respectively.
Even a rapid inspection of Figures \ref{fig:mse1} and \ref{fig:mse2} demonstrates the better performance of ROBOT than the OLS. Specifically, in both figures, OLS is always above ROBOT, and the latter yields much smaller MSE than the former for a large $\eta$. After careful observation, it is found that when the proportion of outliers is small, the MSE curve of ROBOT is close to 1, which is the variance of the data itself. The results show that ROBOT has good robustness against outliers of all sizes.

ROBOT, with the ability to detect outliers, hence fits the uncontaminated distribution better than OLS. Next, we use Figure~\ref{fig:iter1} and Figure~\ref{fig:iter10} to show Algorithm~\ref{alg:robust estimation}  enables us to detect outliers accurately and hence makes the estimation procedure more robust. More specifically, we can see there are many points in Figure~\ref{fig:iter1} (after 1 iteration) classified incorrectly; after 10 iterations (Figure~\ref{fig:iter10}), only a few outliers close to the uncontaminated distribution are not detected.


\begin{figure}[t!]
	\centering
	\begin{subfigure}[b]{0.45\textwidth}
		\centering
		\includegraphics[width=\textwidth]{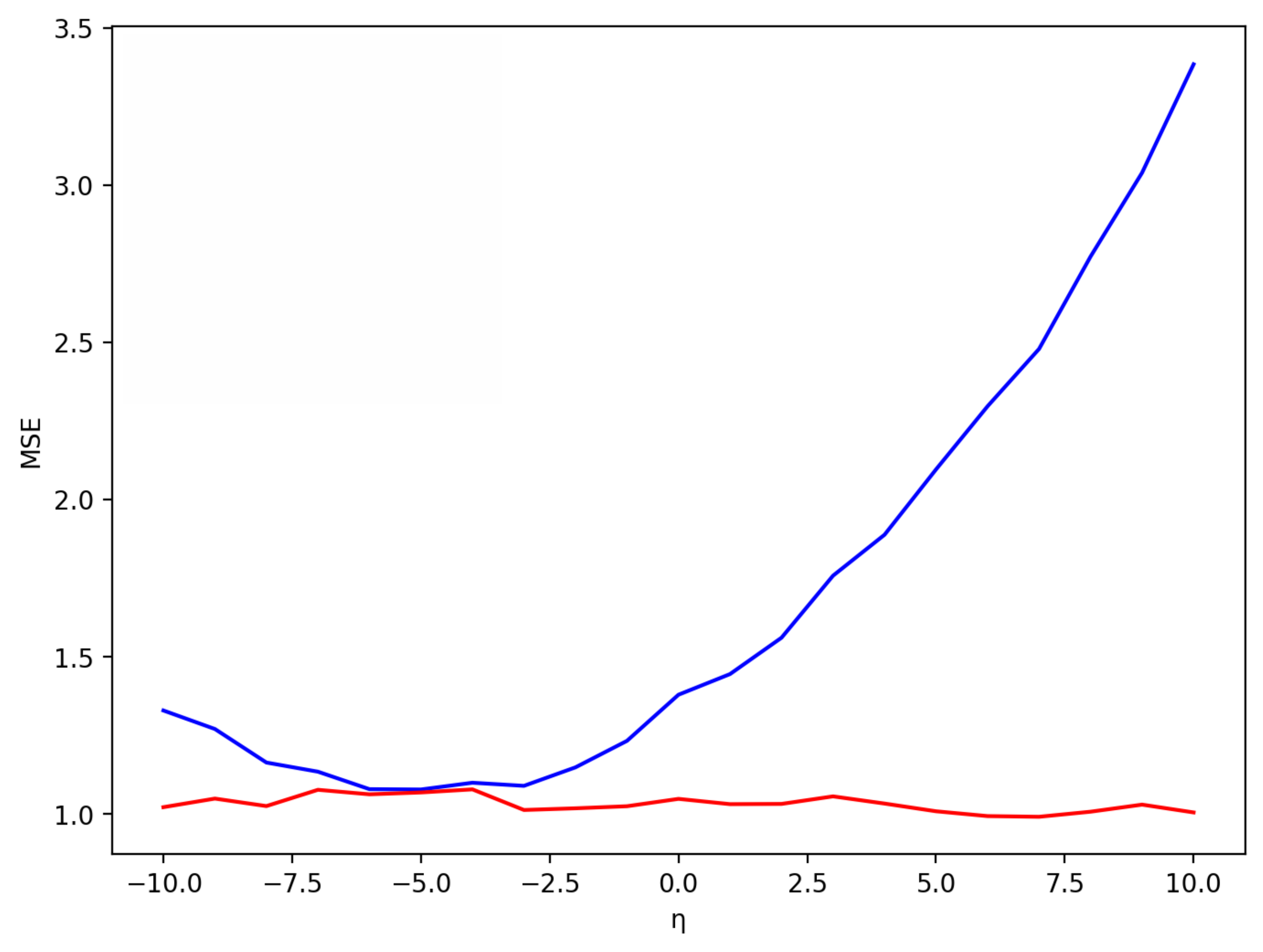}
		\caption{ MSE ($p=0.1$)}
		\label{fig:mse1}
	\end{subfigure}
	\begin{subfigure}[b]{0.45\textwidth}
		\centering
		\includegraphics[width=\textwidth]{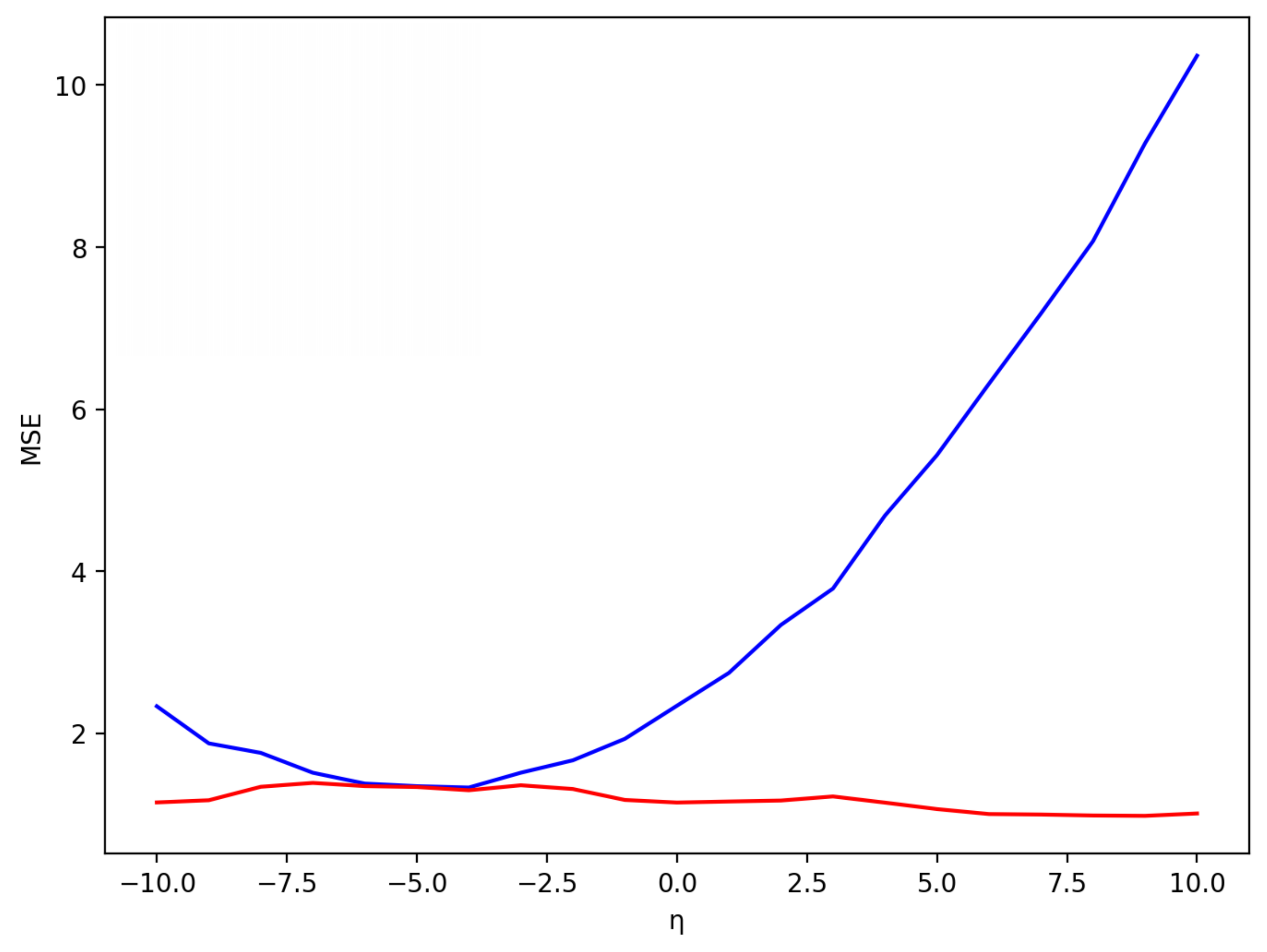}
		\caption{ MSE ($p=0.2$)}
		\label{fig:mse2}
	\end{subfigure}
	\caption{We consider using MSE based on the uncontaminated data to show how well the estimate fits the real distribution.  The left picture is the MSE corresponding to different scale parameter $ \eta $  when the proportion of outliers is 0.1. The red line represents  ROBOT-based estimation while the blue line represents OLS. On the right is the result when $ \varepsilon =0.2 $. }
\end{figure}

\begin{figure}[t!]
	\centering
	\begin{subfigure}[b]{0.48\textwidth}
		\centering
		\includegraphics[width=\textwidth]{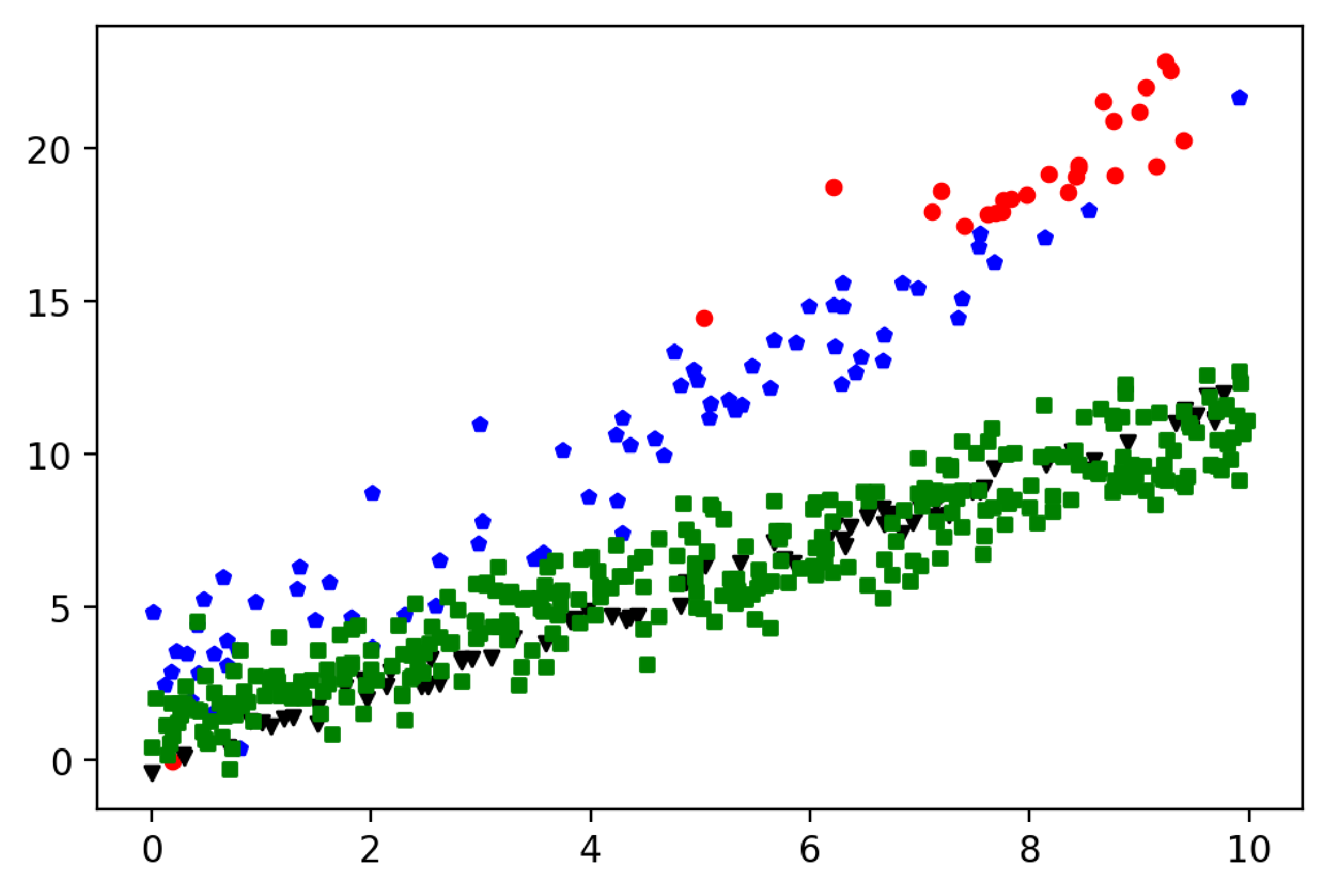}
		\caption{iteration=1}
		\label{fig:iter1}
	\end{subfigure}
	\begin{subfigure}[b]{0.48\textwidth}
		\centering
		\includegraphics[width=\textwidth]{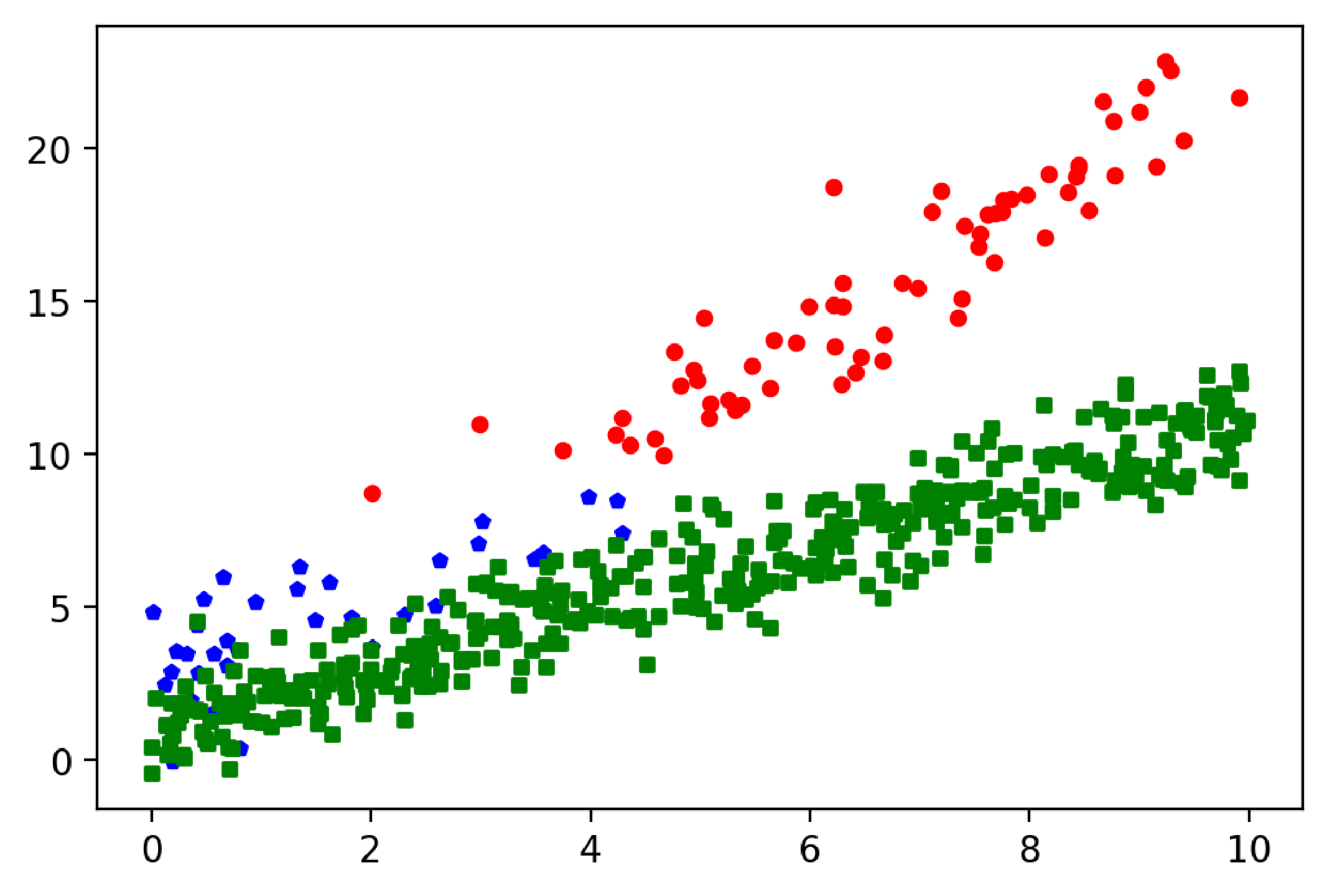}
		\caption{ iteration=10} 
		\label{fig:iter10}
	\end{subfigure}
	\label{fig:detection}
	\caption{Outlier detection for ROBOT-based estimation: the left panel is the result after 1 iteration in Algorithm~\ref{alg:robust estimation},  and the right panel is the result after 10 iterations. Circle points (red) and square points (green) represent the outliers and uncontaminated points correctly detected respectively; pentagon points (blue)  and inverted triangle (black) represent the outliers and uncontaminated points incorrectly detected. (Simulation settings: $n=500$, $ \varepsilon =0.2 $ and $ \eta =1 $).}
\end{figure}

\subsubsection{RWGAN}\label{app.rwgan}

\noindent \textbf{RWGAN-2} Here we provide the key steps needed to implement the RWGAN-2, which is derived directly from \eqref{equ7}. Noticing that \eqref{equ7} allows the  distribution of reference samples to be modified,  and there is a penalty term to limit the degree of the modification, the RWGAN-2 can thus be implemented as follows.

Letting ${X}^{\prime} \sim \mathrm{P}_{{\rm r}^{\prime}} $, where $ \mathrm{P}_{{\rm r}^{\prime}} $ is  the modified reference distribution, and  other notations be the same as in \eqref{equ11}, the objective function of the RWGAN-2 is defined as 
\begin{equation}\label{equ.rwgan2}
	\min_{\theta} \min_{\mathrm{P}_{{\rm r}^{\prime}} } \sup _{\|f_{\xi}\|_L \leq 1,{\rm range}(f_{\xi})\leq 2\lambda } \left\lbrace  {\rm E}[f_{\xi}({X}^{\prime} )]-{\rm E}[f_{\xi}(G_{\theta}({Z}))] + \lambda_{\rm m}  \Vert \mathrm{P}_{{\rm r}}-\mathrm{P}_{{\rm r}^{\prime}}  \Vert_{\rm TV} \right\rbrace,
\end{equation}
where $ \lambda_{\rm m} $ is the penalty parameter to control the modification. The algorithm for implementing the RWGAN-2 is summarized in Algorithm~\ref{alg:robust wgan}. Note that in Algorithm~\ref{alg:robust wgan}, we consider introducing a new neural network $ U_{\omega} $ to represent the modified reference distribution, which  corresponds to $ \mathrm{P_{r^{\prime}}} $ in the objective function \eqref{equ.rwgan2}. Then we can follow \eqref{equ.rwgan2} and use standard steps to update  generator, discriminator and modified distribution in turn. Here are a few things to note: (i) from the objective function, we know that the loss function of discriminator is $ {\rm E}[f_{\xi}(G_{\theta}({Z}))] -   {\rm E}[f_{\xi}({X}^{\prime} )]  $,  the loss function of modified distribution is $ {\rm E}[f_{\xi}({X}^{\prime} )] + \lambda_{\rm m}  \Vert \mathrm{P}_{{\rm r}}-\mathrm{P}_{{\rm r}^{\prime}}  \Vert_{\rm TV} $ and the loss function of generator is $  -{\rm E}[f_{\xi}(G_{\theta}({Z}))] $; (ii) we truncate the absolute values of the discriminator parameter $\xi$   to no more than a fixed constant $ c $ on each update: this  is a simple and efficient way to satisfy the Lipschitz condition in WGAN; (iii) we adopt  Root Mean Squared Propagation (RMSProp) optimization method which is recommended by \cite{arjovsky2017wasserstein} for WGAN.

\begin{algorithm}[H]
	\SetAlgoLined
	\KwData{observations $\{{X}_1, {X}_2, \ldots, {X}_n\}$, critic $f_{\xi}$  (indexed by a parameter $\xi$), generator $G_{\theta}$ (indexed by a parameter $\theta$), modified distribution $U_{\omega}$  (indexed by a parameter $\omega$), learning rate $ \alpha $, modification penalty parameter $ \lambda_{\rm m} $, batch size $ N_{\rm batch} $, number of iterations $N_{\rm iter}$, noise distribution $ \mathrm{P}_{{z}} $, clipping parameter $ c $}
	\KwResult{generation model $G_{\theta}$  }
	\While{$i \leq N_{\rm iter} $}{
		sample a batch of reference samples $\{{X}_i\}_{i=1}^{N_{\rm batch}}$\;
		sample a batch of i.i.d. noises $\{{Z}_i\}_{i=1}^{N_{\rm batch}} \sim \mathrm{P}_{{z}}$\;
		let ${L}_{\xi} =  \frac{1}{N_{\rm batch}} \sum_{i}  f_{\xi}(G_{\theta}({Z}_i)) -  \sum_{i} U_{\omega}({X}_i) f_{\xi}({X}_i) $\;	
		update  $\xi\xleftarrow{} \xi - \alpha \cdot RMSProp(\xi, \nabla_{\xi} {L}_{\xi}) $\;
		$\xi \xleftarrow{} \mathrm{clip}(\xi, -c,c)$\;
		let ${L}_{\theta} = -\frac{1}{N_{\rm batch}} \sum_{i} f_{\xi}(G_{\theta}({Z}_i)) $\;
		update  $\theta \xleftarrow{} \theta- \alpha \cdot RMSProp(\theta,\nabla_{\theta} {L}_{\theta}) $\;
		let ${L}_{\omega}= \sum_{i} U_{\omega}({X}_i) f_{\xi}({X}_i)  + \lambda_{\rm m} \sum_{i} | U_{\omega}({X}_i) -1/N_{\rm batch}| $\;
		update  $\omega\xleftarrow{} \omega- \alpha \cdot RMSProp(\omega, \nabla_{\omega} {L}_{\omega}) $\;
	}
	\caption{RWGAN-2}\label{alg:robust wgan}
\end{algorithm}

\noindent \textbf{Some details.} If we assume that the reference sample is $ m $-dimensional,  the generator is a neural network with $ m $ input nodes and $ m $ output nodes; the critic is a neural network with $ m $ input nodes and $ 1 $ output node. Also, in RWGAN-2 and RWGAN-B, we need an additional neural network (it has $ m $ input nodes and $ 1 $ output node) to represent the modified distribution.  In both the synthetic data example and Fashion-MNIST example, all generative adversarial  models are implemented based on $\mathsf{Pytorch}$  (an open source Python machine learning library \url{https://pytorch.org/}) to which we refer for additional computational details.


In RWGAN-1,  as in the case of WGAN, we also only need two neural networks representing generator and critic, but we need to make a little modification to the activation function of the output node of critic: (i)  choose a bounded  activation function (ii) add a positional parameter to the activation function.  
This is for $ f_{\xi} $ to satisfy the condition $ \mathrm{range}(f_{\xi})\leq 2\lambda $  in \eqref{equ12}. So RWGAN-1 has a simple structure similar to WGAN, but still shows good robustness to outliers. We can also see this from the objective functions of RWGAN-1 (\eqref{equ12}) and RWGAN-2 (\eqref{equ.rwgan2}): RWGAN-2 needs to update one more parameter $ \omega $ than RWGAN-1. This means more computational complexity.

In addition, we make use of Fashion-MNIST to provide an example of how generator $ G_{\theta} $ works on real-data. We  choose a  standard  multivariate normal random variable  (with the same dimension as the reference sample) to obtain the random input sample. An image in Fashion-MNIST is of $ 28\times 28 $ pixels, and we think of it as a 784-dimensional vector. Hence, we choose a  $ 784 $-dimensional standard multidimensional normal random variable $ Z $ as the random input of the generator $ G_{\theta} $,
from which we output a vector of the same 784 dimensions  (which corresponds to the values of all pixels in a grayscale image) to generate one fake image. To understand it in another way, we can regard each image as a sample point, so all the images in Fashion-MNIST constitute an empirical reference distribution $ \hat{\mathrm{P}}_{\rm r} $, thus our goal is to train $ G_{\theta} $ such that   $\mathrm{P}_{\theta} $ is very close to $ \hat{\mathrm{P}}_{\rm r} $.

\bibliographystyle{imsart-number} 
\bibliography{Bibliography-MM-MC}

%
%
%
%

\end{document}